\documentclass[preprint,aos]{imsart}

\usepackage{graphicx}
\usepackage[dvipsnames]{xcolor}
\usepackage[skins]{tcolorbox}

\usepackage[numbers]{natbib}
\usepackage{bm}

\usepackage{amsmath}
\usepackage{amsthm}
\usepackage{amssymb}
\usepackage{hyperref}
\hypersetup{breaklinks=true,
    linkcolor=magenta,
    citecolor=red,
    colorlinks=true,
pdfborder={0 0 0}}
\usepackage{cleveref}

\usepackage{thmtools}
\usepackage{thm-restate}

\DeclareMathOperator*{\argmin}{argmin}

\DeclareMathOperator{\diag}{diag}
\DeclareMathOperator{\Var}{Var}
\DeclareMathOperator{\trace}{{\mathsf{Tr}}}

\def\aaa{a}

\def\P{\mathbb P}
\def\debias{\hat{\beta}^{(d)}}
\def\Rem{\text{Rem}}
\def\sign{\text{sign}}

\def\defas{\stackrel{\text{\tiny def}}{=}}
\def\sign{\text{sign}}
\def\E{\mathbb{E}}
\def\R{\mathbb{R}}
\def\by{\bm{y}}
\def\bX{\bm{X}}
\def\hbeta{\hat\beta}
\def\hbbeta{\bm{\hat\beta}}
\def\dbeta{\dot\beta}
\def\htheta{\hat\theta}
\def\prox{\mathsf{prox}}
\def\hpsi{\hat\psi}
\def\tpsi{\tilde\psi}

\def\dX{\dot{X}}
\def\eps{\varepsilon}
\def\tbeta{\tilde\beta}
\def\hbA{\bm{\hat{A}}}

\def\df{{\mathsf{\hat{df}}}}

\newtheorem{theorem}{Theorem}
\newtheorem{proposition}[theorem]{Proposition}
\newtheorem{lemma}[theorem]{Lemma}
\newtheorem{corollary}[theorem]{Corollary}
\newtheorem{assumption}{Assumption}
\crefname{assumption}{assumption}{assumptions}

\usepackage{todonotes}
\usepackage{constants}

\numberwithin{equation}{section}
\numberwithin{theorem}{section}
\numberwithin{corollary}{section}
\numberwithin{lemma}{section}
\numberwithin{proposition}{section}

\def\green{{\color{teal} \bm{w}^T\bm\hbeta}}
\def\purple{{\color{violet} \bm{\hpsi}^T\bm X\bm w}}

\def\Aunreg{ \bigl(\sum_{i=1}^n\bm x_i \ell_{y_i}''(\bm x_i^T\bm \hbeta)\bm x_i^T\bigr)^{-1}}

\usepackage{booktabs}

\def\myred{\tfrac{\bm{\hpsi}^T \bm G \bm P_1^\perp \bm \htheta}{n\hat r^2}}
\def\mybro{\tfrac{\|\bm{G} \bm P_1^\perp \bm \htheta\|^2}{n \hat r^2}}
\begin{document}

\title{
Observable adjustments in single-index models for regularized M-estimators 
with bounded p/n 
}
\author{Pierre C Bellec}
\runauthor{P.C. Bellec}
\runtitle{Observable adjustments for regularized M-estimators}
\address{Rutgers University}

\begin{abstract}
    We consider observations $(\bm X,\bm y)$ from single index models with unknown link function, Gaussian covariates
    and a regularized M-estimator $\bm\hbeta$ constructed from convex loss function
    and regularizer.
    In the regime where sample size $n$ and dimension $p$ are both
    increasing such that $p/n$ has a finite limit,
    the behavior of the empirical distribution of $\bm\hbeta$ and
    the predicted values $\bm X\bm \hbeta$ has been previously characterized
    in a number of models: The empirical distributions are known
    to converge to proximal operators of the loss and penalty
    in a related Gaussian sequence model, which captures the interplay between ratio $\frac pn$, loss, regularization and the data generating process. This connection between
    $(\bm\hbeta,\bm X\bm\hbeta)$ and the corresponding proximal operators
    involves mean-field parameters defined as solutions 
    to a nonlinear system of equations. This system
    typically involve unobservable quantities such as the prior distribution on the index or the link function, so the mean-field parameters
    need to be estimated.
    Although estimators for the mean-field parameters have been proposed
        in specific cases, a general framework that applies
        simultaneously to a broad class of loss and penalty has so
        far been missing.

    This paper develops a different theory to describe the empirical
    distribution of $\bm\hbeta$ and $\bm X\bm\hbeta$: Approximations
    of $(\bm\hbeta,\bm X\bm\hbeta)$ in terms of proximal operators are provided
    that only involve observable adjustments
    in place of the mean-field parameters.
    These proposed observable adjustments are data-driven, e.g., do not require prior knowledge of the index or the link function. These new adjustments
    yield confidence intervals for individual components of the index,
    as well as estimators of the correlation of $\bm\hbeta$ with the index,
    enabling parameter tuning to maximize the correlation.
    The interplay between loss, regularization and the model is 
    captured in a data-driven manner, without relying on
    the nonlinear systems 
    studied in previous works.
    The results are proved to hold both strongly convex regularizers and
    unregularized M-estimation.
\end{abstract}

\maketitle


\section{Introduction}

Consider iid observations $(\bm x_i,y_i)_{i=1,...,n}$ with Gaussian
feature vectors $\bm{x}_i\sim N(\bm 0,\bm \Sigma)$, $\bm\Sigma\in\R^{p\times p}$ and response
$y_i$ valued in a set $\mathcal Y$ following a single index model
\begin{equation}
    y_i = F(\bm{x}_i^T\bm w, U_i)
\label{single-index}
\end{equation}
where $F:\R^2\to\R$ is an unknown deterministic function, $\bm{w}\in\R^p$ an unknown
index,
and $U_i$ is a latent variable independent of $\bm{x}_i$.
The index $\bm w$
is normalized with $\Var[\bm{x}_i^T\bm w] = \|\bm \Sigma^{1/2}\bm w\|^2=1$
by convention, since $\|\bm\Sigma^{1/2}\bm w\|$ can be otherwise
absorbed into $F(\cdot)$.
The triples $(\bm{x}_i, U_i, y_i)_{i=1,...,n}$ are iid 
but only $(\bm{x}_i, y_i)_{i=1,...,n}$ are observed.
Typical examples that we have in mind
for $F$ and $U_i$ in \eqref{single-index} include
\begin{itemize}
    \item \emph{Linear regression}: $F(v,u) = \|\bm\Sigma^{1/2}\bm\beta^*\| v + u$
        for some $\bm\beta^*\in\R^p$,
        $U_i\sim N(0,\sigma^2)$
        and $\bm w = \bm \beta^* / \|\bm \Sigma^{1/2}\bm \beta^*\|$.
        Equivalently, $y_i|\bm x_i \sim N(\bm x_i^T\bm \beta^*, \sigma^2)$.
    \item 
        \emph{Logistic regression}:
        $F(v,u) = 1$ if $u\le 1/(1+e^{-\|\bm\Sigma^{1/2}\bm \beta^*\| v})$ and 0 otherwise
for some $\bm\beta^*\in\R^p$,
$U_i\sim \text{Unif}[0,1]$ and
$\bm w = \bm \beta^* / \|\bm \Sigma^{1/2}\bm \beta^*\|$.
Equivalently,
$y_i|\bm x_i \sim \text{Bernoulli}(\rho'(\bm x_i^T\bm \beta^*))$
where $\rho'(u)=1/(1+e^{-u})$ is the sigmoid function.
\item \emph{1-bit compressed sensing}:
$F(v,u) = \sign(v)$
so that $y_i=\sign (\bm x_i^T\bm w)$, or 
\emph{1-bit compressed sensing with $\epsilon$-contamination}:
$F(v,u) = u\sign(v)$
for $U_i\in\{\pm1\}$ s.t. $\P(U_i=-1) =\eps$.
\item \emph{Binomial logistic regression}:
    $y_i|\bm x_i \sim \text{Binomial}(q,\rho'(\bm x_i^T\bm \beta^*))$
    for some integer $q\ge 1$ and $\rho'(u)=1/(1+e^{-u})$
    the sigmoid function.
\end{itemize}
Throughout the paper, $\bm\hbeta$ is a regularized $M$-estimator
of the form
\begin{equation}
\bm{\hbeta}(\bm y,\bm X) = \argmin_{\bm b\in\R^p}
\frac1n\sum_{i=1}^n\ell_{y_i}\bigl(\bm{x}_i^T \bm b\bigr) + g(\bm b)
\label{hbeta}
\end{equation}
where $g:\R^p\to\R$ is a convex penalty function
and for any $y_0\in \mathcal Y$, the map 
$\ell_{y_0}:\R\to\R$, 
$t\mapsto \ell_{y_0}( t )$ is a convex loss function.
For a fixed $y_0$, the derivatives of $\ell_{y_0}$ are denoted by
$\ell_{y_0}'(t)$ and  $\ell_{y_0}''(t)$ where these derivatives exist.
If $y_i\in\{-1,1\}$ or $y_i\in\{0, 1\}$ as in binary classification
(e.g., with $y_i|\bm x_i$ following a logistic regression
or 1-bit compressed sensing model), a popular loss function is for instance
the logistic loss $\ell_{y_i}(t) = \log(1+e^{-t})$ if $y_i=1$
and $\ell_{y_i}(t) =\log(1+e^{t})$ if $y_i\ne 1$.
This paper focuses on the behavior
of estimators of the form \eqref{hbeta}
in the single index model
\eqref{single-index} when the dimensions and sample sizes
are both large and of the same order,
confidence intervals for individual components
of the index $\bm{w}$, and estimation of performance
metrics of interest such as the correlation between $\bm\hbeta$ and
the unknown index $\bm w$.

\def\proxpen{
        \prox[\sigma\tau \tilde f(\cdot)]
        (\sigma\tau(\theta\beta + \delta^{-1/2}rZ))}
\def\proxpenbar{
        \prox[\bar\sigma\bar\tau \tilde f(\cdot)]
        \bigl(\bar\sigma\bar\tau(\bar\theta\beta + \delta^{-1/2}\bar r Z)\bigr)}
\def\argproxloss{
    \kappa\alpha Z_1 + \sigma Z_2
}
\def\proxloss{
    \prox[\gamma \rho] (\argproxloss)
}
\def\proxprimeloss{
    \prox[\gamma \rho]'(\argproxloss)
}
\subsection{Prior works in asymptotic behavior of M-estimators}

There is now a rich literature on the behavior
of M-estimators in the regime where $p/n$ converges to a constant
in linear models \cite[among others]{bayati2012lasso,el_karoui2013robust,stojnic2013framework,thrampoulidis2015lasso,thrampoulidis2018precise,el_karoui2018impact,mei2019generalization,celentano2019fundamental,rush2020rate_lasspo,miolane2018distribution}
and generalized models, including logistic regression
\cite{sur2018modern,salehi2019impact} and general
teacher-student models \cite{loureiro2021capturing}.
To present typical results from this literature,
assume in this section that $n,p\to+\infty$
with $n/p\to\delta$ for some constant $\delta>0$
and isotropic covariance matrix $\bm\Sigma = \frac 1 p \bm I_p$.

Consider first a linear model with $y_i= \bm x_i^T\bm\beta^* + \eps_i$
for $\eps_i$ independent of $\bm x_i$ and the unregularized estimator
$\bm\hbeta = \argmin_{\bm b\in\R^p}  \sum_{i=1}^n \mathcal L(y_i - \bm x_i^T\bm b)$
for some convex loss $\mathcal L:\R\to\R$.
This corresponds to $g=0$ and $\ell_{y_i}(u) = \mathcal L(y_i - u)$ in \eqref{hbeta}.
The works \cite{el_karoui2013robust,donoho2016high,karoui2013asymptotic,el_karoui2018impact}
showed that the behavior of $\bm\hbeta$
is characterized by the system of two equations
\begin{equation}
    \left\{
    \begin{aligned}
    \delta^{-1}\sigma^2
    &= \E\bigl[(\prox[\gamma \mathcal L ](\eps_1 + \sigma Z) - \eps_1 - \sigma Z )^2\bigr],
    \\
    1- \delta^{-1} 
    &= \E\bigl[\prox[\gamma \mathcal L ]'(\eps_1 + \sigma Z)\bigr],
    \end{aligned}
    \right.
    \label{karoui_system}
\end{equation}
with two unknowns $(\sigma,\gamma)$,
where $Z\sim N(0,1)$ is independent of $\eps_1$.
In \eqref{karoui_system} and throughout the paper, for any convex function $f:\R\to\R$,
the proximal operator $\prox[f]$ of $f$ is defined as 
$\prox[f](u) = \argmin_{v\in \R} (u-v)^2/2 + f(v)$,
and we denote by $\prox[f]'$ its derivative.
While \cite{karoui2013asymptotic} uses notation $(r,c)$ for the two unknowns
in \eqref{karoui_system}, we use $(\sigma,\gamma)$ instead
to reveal the connection
with the larger systems \eqref{sur_system} and \eqref{salehi_system} below.
The solution $(\bar \sigma,\bar \gamma)$ to \eqref{karoui_system}
is such that
$p^{-1} \|\bm\hbeta - \bm\beta^*\|^2 \to \bar \sigma^2$ in probability,
and for each fixed component $j=1,...,p$ the convergence in distribution
$\hbeta_j - \beta_j^* \to^d N(0,\bar \sigma^2)$ holds \cite{el_karoui2013robust},
so that the system \eqref{karoui_system} and its solution
captures the asymptotic distribution of $\bm\hbeta$. Equipped with the system \eqref{karoui_system} and these results,
\cite{bean2013optimal} studied the optimal loss function $\mathcal L(\cdot)$
that minimizes $\|\bm\hbeta-\bm \beta^*\|^2$
for a given $\delta = \lim \frac{n}{p}$ and given noise distribution
for $(\eps_1,...,\eps_n)$.
In linear model with normally distributed noise,
\citet{bayati2012lasso} used Approximate Message Passing (AMP)
to establish a similar phenomenon for the Lasso
$\bm\hbeta = \argmin_{\bm b\in\R^p}\|\bm y-\bm X\bm\hbeta\|^2/2 + \lambda \|\bm b\|_1$, showing that the empirical distribution of $(\hbeta_j)_{j=1,...,p}$
is close in distribution to the empirical distribution
of $(\eta(\beta_j^* + \bar\tau Z_j; \lambda\frac{\bar\tau}{\bar b}))_{j=1,...,p}$
where $\eta(x;u) = \sign(x)(|x|-u)_+$ is the soft-thresholding 
operator and $(\bar\tau,\bar b)$ is solution to a nonlinear system
of two equations of a similar nature as \eqref{karoui_system} \cite[Theorem 1.5]{bayati2012lasso}; the works \cite{miolane2018distribution,rush2020rate_lasspo,celentano2020lasso} provides explicit error bounds between
functions of the Lasso $\bm\hbeta$ and their prediction from the
nonlinear system of two equations.
Inspired by early works from \citet{stojnic2013framework},
\citet{thrampoulidis2018precise} developed the
Convex Gaussian Min-max Theorem
and obtained analogous 
systems of equations to characterize the limit of 
$p^{-1}\|\bm\hbeta-\bm\beta^*\|^2$
for a given loss-penalty pair in linear models, see for instance
the system with four unknowns 
\cite[Eq. (15)]{thrampoulidis2018precise}
for separable loss and penalty.


With the model \eqref{single-index}, our focus in this paper
is on single-index models where $\mathbb E[y_i|\bm x_i]$ is typically not
a linear function of $\bm x_i$.
To present some representative existing results of a similar nature
as those in the previous paragraph,
we now turn to
logistic regression, a particular case of the single model \eqref{single-index}.
Consider iid observations from the logistic model
\begin{equation}
y_i|\bm x_i \sim \text{Bernoulli}\bigl(
    \rho'(\bm x_i^T\bm\beta^*)
\bigr)
\qquad \text{ for }
\rho'(u) = 1/(1+e^{-u}),
\label{eq:logistic-model}
\end{equation}
and  the
M-estimator \eqref{hbeta} with the logistic loss
$\ell_{y_i}(u_i) = \rho(u_i) - y_i u_i$
for $\rho(u) = \log(1+e^u)$.
With no regularization, i.e., $g=0$,
$\bm \hbeta$ in \eqref{hbeta} is the logistic Maximum Likelihood Estimate (MLE).
\citet{sur2018modern} describe the asymptotic distribution
of the MLE as follows. For a sequence of logistic regression problems
with $n,p,\bm\beta^*$ such that $\|\bm\beta^*\|^2/p\to\kappa^2$
and $n/p\to\delta>0$,
the MLE exists with overwhelming probability
if
$
 \delta > \min_{t\in \R}\iint (z-tv)_+^2\varphi(z)2\rho'(\kappa v)\varphi(v)dzdv 
 $
\cite{candes2020phase}
where $\varphi(z)=(\sqrt{2\pi})^{-1}\exp(-\frac{z^2}{2})$
is the standard normal pdf.
On this side of the phase transition where the MLE exists,
assuming additionally that the components of $\bm\beta^*$ are iid copies
of a random variable $\beta$
(or more generally that the empirical distribution of $\bm\beta^*$
has a weak limi ),
\citet{sur2018modern} further established that
for any Lipschitz function $\phi:\R^2\to\R$
\begin{equation}
\frac1p\sum_{j=1}^p \phi\Bigl(\hbeta_j-\bar\alpha\beta_j^*,\beta^*_j\Bigr)
\to^{\P}
\E\Bigl[\phi\Bigl(\bar\sigma Z,\beta\Bigr)\Bigr] 
\label{eq:limiting_result_sur}
\end{equation}
where $Z\sim N(0,1)$ is independent of $\beta$,
the arrow $\to^\P$ denotes convergence in probability,
and $(\bar\alpha,\bar\sigma,\bar\gamma)$ is solution of the nonlinear
system of 3 equations
\begin{equation}
    \label{sur_system}
    \left\{
    \begin{aligned}
    \delta^{-1} \sigma^2 &=
    2\E\bigl[\rho'(-\kappa Z_1)\bigl(\gamma \rho'(\proxloss) \bigr)^2\bigr],
    \\
    0 &= 
    2\E\bigl[ \rho'(-\kappa Z_1) Z_1\gamma \rho'(\proxloss)\bigr],
    \\
    1-\delta^{-1}
      &=
    2
    \E\bigl[\rho'(-\kappa Z_1)
        \proxprimeloss
    \bigr]
    .
    \end{aligned}
    \right.
\end{equation}
Above,
$\rho'(u) = 1/(1+e^{-u})$ is the sigmoid function
as in \eqref{eq:logistic-model},
$\rho''(u) = \rho'(u)(1-\rho'(u))$ its derivative,
and inside the expectation in the third line
$\prox[\gamma\rho]'(u)
    =
    1/\bigl(1+\gamma \rho''(\prox[\gamma\rho](u))\bigr)$.
    \citet{sur2018modern} observed that \eqref{sur_system}
    could be solved numerically
    if $(\delta,\kappa)$ is on the side of the
    phase transition where the MLE exists with high-probability.
    In fact, the system \eqref{sur_system} admits a solution
    if and only if $(\delta,\kappa)$ is on the side of the
    phase transition where the MLE exists with high-probability:
    this was formally proved in the global-null case
    \cite{sur2019likelihood} and for the general case
    in \cite{bellec2025phase}.
\citet{salehi2019impact}
and \citet[Chapter 4]{sur2019thesis} extended 
results of the form \eqref{eq:limiting_result_sur}
to the M-estimator \eqref{hbeta}
constructed with the same logistic loss,
$\ell_{y_i}(u_i) = \rho(u_i) - y_i u_i$
for $\rho(u) = \log(1+e^u)$,
and separable penalty function of the form
$g(\bm b) = \frac 1 j \sum_{j=1}^p \tilde f(b_j)$
for some convex $\tilde f:\R\to\R$.
\citet[Chapter 4]{sur2019thesis}
presented conjectures regarding
the system governing
the asymptotic of such penalized
logistic estimator stemming from
the rigorous AMP analysis of the MLE and the
Ridge logistic case.
\citet{salehi2019impact} provided
a rigorous proof of this asymptotic
using the CGMT \cite{thrampoulidis2018precise}.
To describe the main result of \cite{salehi2019impact},
assume that the coefficients of $\bm\beta^*$ are iid copies
of a random variables $\beta$ with finite variance,
and let $(Z,Z_1,Z_2)^T\sim N(\bm{0},\bm I_3)$
be independent of $\beta$.
Let $\kappa^2=\E[\beta^2]$ and consider the system of 6 equations
\begin{equation}
\left\{
    \begin{aligned}
        \kappa^2\alpha & =\E\bigl[\beta \proxpen\bigr],
        \\
        \sqrt \delta r \gamma &= \E\bigl[Z\proxpen\bigr],
        \\
        \kappa^2\alpha^2+\sigma^2 &= \E\bigl[\{\proxpen\}^2 \bigr],
        \\
        r^2\gamma^2 &= 2\E\bigl[\rho'(-\kappa Z_1)(
                \argproxloss - \proxloss
        )^2\bigr],
        \\
        - \theta\gamma &= 2\E\bigl[
        \rho''(-\kappa Z_1)\proxloss
        \bigr],
        \\
        1-\gamma/(\sigma\tau)
                 &=
                 2
                 \E\bigl[
                     \rho'(-\kappa Z_1)
                    \proxprimeloss
                 \bigr]
    \end{aligned}
\right.
    \label{salehi_system}
\end{equation}
with unknowns $(\alpha,\sigma,\gamma,\theta,\tau,r)$.
Equation (4.3) in \cite{sur2019thesis}
provides an alternative formulation
to \eqref{salehi_system} using four equations involving iterated expectations.
As in the case of the MLE in \cite{sur2018modern} with 
\eqref{sur_system},
the system \eqref{eq:limiting_salehi} captures the interplay
between the logistic model \eqref{eq:logistic-model},
the penalty and the limit $\delta$ of the ratio $n/p$:
The main result of \cite{salehi2019impact} states that if the nonlinear
system \eqref{salehi_system} has a unique solution
$(\bar\alpha,\bar\sigma,\bar\gamma,\bar\theta,\bar\tau,\bar r)$
then for any locally Lipschitz function $\Phi:\R^2\to\R$,
\begin{equation}
\frac1p\sum_{j=1}^p \Phi\Bigl(\hbeta_j,\beta^*_j\Bigr)
\to^{\P}
\E\Bigl[\Phi\Bigl(\proxpenbar,\beta\Bigr)\Bigr] 
\label{eq:limiting_salehi}
\end{equation}
as $n,p\to\infty$ with $n/p\to\delta$.
An informal interpretation of \eqref{eq:limiting_salehi} is the approximation
\begin{equation}
\hbeta_j \approx 
\prox[\bar\sigma\bar\tau \tilde f(\cdot)]
\bigl(\bar\sigma\bar\tau(\bar\theta\beta^*_j + \delta^{-1/2}\bar r Z_j)\bigr),
\label{informal-salehi}
\end{equation}
where $Z_j\sim N(0,1)$, and \eqref{informal-salehi} holds
in an averaged sense over $j=1,...,p$.
This approximation means that in order to understand $\bm{\hbeta}$, it 
is sufficient to understand the random vector with independent components
$\hat b_j=
\prox[\bar\sigma\bar\tau \tilde f(\cdot)]
\bigl(\bar\sigma\bar\tau y_j^{seq}\bigr) 
$
in the Gaussian sequence model $y_j^{seq} \sim N(\bar\theta \beta_j^*, \frac{\bar r^2}{\delta}), j=1,...,p$.
The works \cite{gerbelot2020asymptotic,loureiro2021capturing}
further extend these results to loss and penalty functions that need
not be separable, provide a unified theory for the systems
\eqref{karoui_system}, \eqref{eq:limiting_result_sur} and \eqref{salehi_system},
and describe the relationship of these results
with predictions from the replica method in statistical physics.

The above system \eqref{salehi_system} may reduce to simpler forms
for specific penalty functions.
With Ridge penalty
$g(\bm b) = \frac{\lambda}{2p} \|\bm b\|^2$,
in the isotropic setting with
covariance $\bm\Sigma = \frac 1 p \bm I_p$, the system \eqref{salehi_system}
reduces to
\begin{equation}
    \label{ridge_system}
    \left\{
        \begin{aligned}
        \delta^{-1}\sigma^2 &= 2\E\bigl[\rho'(-\kappa Z_1)(
                \argproxloss - \proxloss
        )^2\bigr],
        \\
        -\delta^{-1}\alpha &= 2\E\bigl[
        \rho''(-\kappa Z_1)\proxloss
        \bigr],
        \\
        1 - \delta^{-1} + \lambda\gamma
                 &=
                 2
                 \E\bigl[
                 \rho'(-\kappa Z_1)
                 \proxprimeloss
             \bigr],
    \end{aligned}
    \right.
\end{equation}
see \cite[Eq. (14), (16)]{salehi2019impact}.
By integration by parts, \eqref{ridge_system}
is equivalent to \eqref{sur_system} when $\lambda=0$,
and the convergence \eqref{eq:limiting_salehi}
with $\Phi(u,v)=\phi(u-\bar\alpha v,v)$
reduces to \eqref{eq:limiting_result_sur}
due to the simple form of the proximal operator
$\prox[t \tilde f](x) = \frac{x}{1+\lambda t}$
where $\tilde f(u) = \lambda u^2/2$.

As explained in \cite{sur2018modern,salehi2019impact} among others,
the systems \eqref{sur_system}, \eqref{salehi_system}
and \eqref{ridge_system} combined with the asymptotic results
\eqref{eq:limiting_result_sur} and \eqref{eq:limiting_salehi}
are powerful tools to analyze various characteristics
of the M-estimator $\bm \hbeta$, for instance the correlation
$p^{-1} \bm \hbeta^T \bm \beta^*$ with the true regression vector
$\bm\beta^*$
by using $\Phi(a,b) = ab$ in
\eqref{eq:limiting_salehi}, 
giving $p^{-1} \bm \hbeta^T \bm \beta^*\to^\P \alpha \kappa^2$ by the first
line in \eqref{salehi_system}.
The Mean Squared Error (MSE)
$p^{-1} \|\bm \hbeta - \bm \beta^*\|^2$ is analogously 
characterized using $\Phi(a,b) = (a-b)^2$ in \eqref{eq:limiting_salehi}.
If the penalty is of the form $g=\lambda g_0$ for some tuning parameter
$\lambda>0$, solving the above nonlinear systems
for given $(\lambda,\delta)$ provide curves of performance metrics
of interest (e.g., correlation $p^{-1} \bm \hbeta^T \bm \beta^*$
or MSE $p^{-1} \|\bm \hbeta - \bm \beta^*\|^2$) as a function
of the tuning parameter and the limit $\delta$ of $n/p$
\cite{salehi2019impact}\cite[(4.8)-(4.8) and Fig. 4.1]{sur2019thesis}.
Plotting such curves by computing the solutions
$(\bar\alpha,\bar\sigma,\bar\gamma,\bar\theta,\bar\tau,\bar r)$
of \eqref{salehi_system} require the knowledge of the 
distribution $\beta$ of the components of $\bm \beta^*$,
or in the case of \eqref{sur_system} and \eqref{ridge_system}
the second moment $\kappa^2 = \E[\beta^2]\approx p^{-1} \|\bm \beta^*\|^2$.
Solving the system also requires the knowledge of the single index model:
If the Bernoulli parameter in \eqref{eq:logistic-model}
is of the form $\tilde\rho(\bm x_i^T\bm \beta^*)$ for $\tilde\rho:\R\to[0,1]$
different than the sigmoid function $\rho$, the systems \eqref{sur_system}
and \eqref{salehi_system} require an appropriate modification,
and solving the new system requires the knowledge of $\tilde\rho$
(see for instance the system in \cite[Section B.8]{loureiro2021capturing}).

The above mean-field results \eqref{eq:limiting_salehi}-\eqref{informal-salehi}
give a powerful description
of the limiting empirical distributions of $\hbbeta$ and $\bX\hbbeta$
in terms of the low-dimensional
parameters $(\bar\alpha,\bar\sigma,\bar\gamma,\bar\theta,\bar\tau,\bar r)$.
A practical situation motivating the present paper is the following:
Given two loss-penalty pairs $(\ell,g)$
and $(\tilde \ell,\tilde g)$, the practitioner may wish to pick 
the loss-penalty pair such that the resulting M-estimator in \eqref{hbeta}
has larger correlation with the true logistic coefficient $\bm\beta^*$
in \eqref{eq:logistic-model}. The asymptotic
\eqref{eq:limiting_salehi} suggests to compute or estimate $\bar\alpha$
for each loss-penalty pair and to
pick the pair with the largest $\bar\alpha$,
since $\bar\alpha$ is the limit of $\bm\hbeta^T\bm\beta^*/\|\bm\beta^*\|^2$.
To leverage this powerful description, the corresponding
low-dimensional parameters need to be computed or
estimated from the data $(\bX,\by)$.
In practice, 
the law of $\beta$ used to define the expectations in
\eqref{salehi_system} is typically unknown so that
finding $(\bar\alpha,\bar\sigma,\bar\gamma,\bar\theta,\bar\tau,\bar r)$
by solving \eqref{salehi_system} by computing the expectations
on the right-hand side of \eqref{salehi_system} is not possible.
Estimating scalar solutions of the above system is possible in certain
cases, including in linear models \cite{bayati2013estimating,miolane2018distribution,bellec2020out_of_sample,celentano2020lasso} \cite[Proposition 2.7]{el_karoui2018impact} and unregularized logistic regression \cite{sur2018modern,yadlowsky2021sloe}, but a general recipe to estimate the mean-field
parameters in single index models is lacking.

For the MLE, \eqref{eq:limiting_result_sur}-\eqref{sur_system}
also provide confidence intervals for components $\beta_j^*$
\cite{sur2018modern,zhao2020asymptotic}.
Let $z_{\alpha/2}>0$ be such that $\P(|N(0,1)|>z_{\alpha/2})=\alpha$.
The result \eqref{eq:limiting_result_sur} applied to a
smooth approximation of
the indicator function $\phi(a,b)=I\{|a|\le \bar\sigma z_{\alpha/2}\}$
yields that approximately
$(1-\alpha)p$ covariates $j=1,...,p$ are such that
\begin{equation}\beta_j^* \in \tfrac{1}{\bar\alpha} [
\hbeta_j - \bar\sigma z_{\alpha/2},
\hbeta_j + \bar\sigma z_{\alpha/2}
].
\label{conf-interval-zhao}
\end{equation}
This provides a confidence interval in an average sense.
\citet{sur2018modern} presented an independent leave-one-out analysis
(building upon \cite{el_karoui2013robust})
for null-covariates that allows characterization of the limiting distribution
of a given $\hat \beta_j^*$ as well as finite-dimensional marginal guarantees.
\citet{zhao2020asymptotic} later proved that the same confidence interval
is valid not only in this averaged sense, but also for a fixed, given component
$j=1,...,p$ under conditions on the amplitude of $\beta_j^*$.
As in the previous paragraph, such confidence interval can be constructed
provided that $\kappa^2$ is known so that the solution
$(\bar\alpha,\bar\sigma,\bar\gamma)$ to the system \eqref{sur_system} can be
computed. This motivated the ProbeFrontier \cite{sur2018modern}
and SLOE \cite{yadlowsky2021sloe} procedures
to compute approximations of $\kappa^2$
and of the solutions $(\bar\alpha,\bar\sigma,\bar\gamma)$ of the system
\eqref{sur_system}.
These procedures \cite{sur2018modern,yadlowsky2021sloe} to estimate
$(\bar\alpha,\bar\gamma,\bar\sigma)$ in \eqref{sur_system}
for the logistic MLE require the single-index model \eqref{single-index}
to be well-specified (i.e., $y_i|\bm x_i^T\bm \beta^*$
must follow an actual logistic model 
\eqref{eq:logistic-model}):
If the single-index model \eqref{single-index}
for the conditional law $y_i|\bm x_i^T\bm \beta^*$ deviates from 
the assumed model \eqref{eq:logistic-model} then \eqref{sur_system}-\eqref{salehi_system} must be modified to account for the actual generative model of the conditional law
$y_i|\bm x_i^T\bm \beta^*$, as in \cite[Section B.8]{loureiro2021capturing}.

For regularized logistic regression with \emph{non-smooth} penalty,
even if the solution $(\bar\alpha,\bar\sigma,\bar\gamma,\bar\theta,\bar\tau,\bar r)$ were known, constructing confidence intervals 
for $\beta_j^*$
from \eqref{eq:limiting_salehi} alone is typically not possible since
$\prox[\bar\sigma\bar\tau\tilde f]$ is not injective
(e.g., for the non-smooth penalty $\tilde f(u)=|u|$,
the proximal of $\tilde f$ is the
soft-thresholding operator which is not one-to-one).
The lack of injectivity comes from the multi-valued nature 
of the subdifferential: 
$x=\prox[u f](z)$ holds if and only if
$\frac{x-z}{u}\in \partial f(x)$, but $(x,u)$ alone are not sufficient
to recover $z$ if the subdifferential $\partial f(x)$ is multi-valued.
Even if the value
\begin{equation}
\prox[\bar\sigma\bar\tau \tilde f(\cdot)]
\bigl(\bar\sigma\bar\tau(\bar\theta\beta^*_j + \delta^{-1/2}\bar r Z_j)\bigr)
\label{value_prox}
\end{equation}
were known exactly (or approximately through $\hbeta_j$
as in \eqref{informal-salehi}),
recovering $\beta_j^*$ from the value \eqref{value_prox} still requires
to choose a specific element of the subdifferential of $\tilde f$
at \eqref{value_prox}, and results such as 
\eqref{eq:limiting_salehi}-\eqref{informal-salehi} are not 
informative regarding which element of the subdifferential of $\tilde f$
at \eqref{value_prox} should be used.
Because of this difficulty, in general nonlinear models such as 
\eqref{eq:logistic-model} with non-smooth penalty,
confidence intervals for components of $\bm\beta^*$ are lacking.

\subsection{Contributions}

\paragraph*{A peek at our results}
This paper develops a different theory 
to provide proximal approximations, confidence intervals as well as data-driven estimates of the
bias of $\bm\hbeta$ and of the
correlation $\bm\hbeta^T\bm\Sigma\bm w$ with the index $\bm w$ (here, correlation refers to the inner product $\langle \bm\hbeta, \bm w \rangle_{\bm\Sigma} = \bm\hbeta^T\bm\Sigma \bm w$ induced by the positive definite matrix
$\bm\Sigma$).
If $\bm \Sigma = \frac 1 p \bm I_p$
and $g(\bm b) = \frac 1 p \sum_{j=1}^p \tilde f(b_j)$
as in the previous subsection,
a by-product of this paper is the approximation
\begin{equation}
\hbeta_j \approx \prox\Bigl[\frac{1}{\hat v} \tilde f\Bigr]
\Bigl( \pm w_j\frac{ \hat t}{\hat v}
+ \frac{1}{\sqrt \delta} \frac{\hat r}{\hat v} Z_j
\Bigr)
\label{peek}
\end{equation}
for $Z_j\sim N(0,1)$ and $\pm w_j$ the $j$-th component of the index $\bm w$
in \eqref{single-index} up to an unidentifiable sign,
where $(\hat v, \hat t, \hat r)$ are observable scalars
defined below in \eqref{hat_v_r_t}.
The informal approximation \eqref{peek}
is made rigorous in \Cref{thm:proximal-hbeta}.
The approximation \eqref{peek} mimics \eqref{informal-salehi} with $\bm\beta^*$ replaced by the normalized index
$\bm w$, with the important difference
that the adjustments $(\hat v, \hat r, \hat r)$ of the previous
display are observable from the data $(\bm y,\bm X)$,
while the deterministic adjustments
$(\bar\sigma,\bar\tau,\bar\theta,\bar r)$ in \eqref{informal-salehi}
requires the knowledge of $\kappa^2$ and of the distribution
of the components of $\bm\beta^*$ to solve the system \eqref{salehi_system}.
This means that the rich information contained in the system
\eqref{salehi_system} and the limiting result
\eqref{eq:limiting_salehi} can be captured from the data $(\bm y,\bm X)$
by computing $(\hat v,\hat r,\hat t)$ in a data-driven fashion,
bypassing solving \eqref{eq:limiting_salehi}
and the theory described in the previous subsection.

\paragraph*{Main contributions} A summary of the main contributions of the paper, with pointers to the main theorems in the next sections, is as follows.

\begin{itemize}
    \item 
        The present paper introduces the adjustments $(\hat r^2,\hat t^2,\hat v)$ that appear in the proximal representation
        \eqref{peek}.
        For strongly convex penalty functions,
        \Cref{thm:confidence-intervals-Sigma} derivatives an asymptotic
        normality result for debiased versions of $\hbeta_j$,
        which yields confidence intervals and hypothesis tests
        for the $j$-th component
        $w_j$ of the index in the single index model \eqref{single-index},
        as well as proximal representations such as \eqref{peek}
        involving the adjustments $(\hat r^2,\hat t^2, \hat v)$
        in
        \Cref{thm:proximal-hbeta}.
        Similar results hold without a strongly convex penalty for unregularized M-estimators in
        \Cref{thm:unregularized}.
    \item
        Additional adjustments
        $(\hat\sigma^2,\hat \aaa^2)$
        are introduced to estimate
        from the data $(\bm X,\bm y)$
        the correlation
        $\aaa_* = \bm w^T\bm\Sigma\bm\hbeta$.
        This is of practical interest, for instance to tune hyper-parameters in order to maximize the correlation with the index $\bm w$ in the model \eqref{single-index}. Bounds on the estimation error of $\hat \aaa^2-\aaa_*^2$ are derived in
        \Cref{thm:adjustments-approximation} below for strongly convex
        penalty and in \Cref{thm:unregularized} for unregularized M-estimation.
        The estimator $\hat a^2$ takes the form
        \begin{equation*}
            \hat \aaa^2 =
            \frac{ \big(
            \tfrac{\hat v}{n}\|\bm{X} \bm \hbeta - \hat\gamma \bm \hpsi\|^2
            + \tfrac{1}{n}\bm{\hpsi}^\top\bm X\bm \hbeta
            - \frac{\hat\gamma }{n}\|\bm\hpsi\|^2
        \big)^2
        }{
                \tfrac{1}{n^2}\|\bm{\Sigma}^{-1/2}\bm{X}^T\bm \hpsi\|^2
                +  \tfrac{2 \hat v}{n} \bm{\hpsi}^T\bm X\bm \hbeta
                + \tfrac{\hat v^2}{n}
                \|\bm{X} \bm \hbeta - \hat\gamma \bm \hpsi\|^2
                -
                \tfrac{p}{n^2}\|\bm\hpsi\|^2
        }
        \end{equation*}
        where $\bm\hpsi\in\R^n$ is the vector with components
        $\hpsi_i = -\ell_i'(\bm x_i^T\bm\hbeta)$ 
        and $(\hat v,\hat\gamma)$ are defined in \eqref{hat_v_r_t} below
        from the derivatives of $\bm\hbeta(\bm y,\bm X)$.
        The expression of $\hat a^2$ above is notably involved, but
        also surprisingly accurate in simulations as shown in \Cref{sec:simulations}, in particular in \Cref{fig:L1} for L1-regularized M-estimation
        in binomial logistic models.
        \item
        The most studied generalized linear model is logistic regression
        thanks to the seminal works \cite{candes2020phase,sur2018modern,zhao2020asymptotic} for the MLE, where the nonlinear system
        \eqref{sur_system} describe the performance and distribution
        of $\bm\hbeta$. Under the assumptions in these works
        with a logistic model with true coefficient $\bm\beta^*$
        and  $n,p\to+\infty$ with a fixed ratio $p/n$,
        asymptotic normality
        of $\hbeta_j$ for a fixed coordinate $j\in[p]$ is granted
        provided $\tau_j|\beta_j^*| = O(n^{-1/2})$
        where $\tau_j^2=(\bm\Sigma^{-1})_{jj}$
        \cite[Theorem 3.1]{zhao2006model}.
        Under the same assumptions, in \Cref{sec:unregularized}
        we show that the techniques of the present paper allow to relax
        the requirement on the magnitude
        of $|\beta_j^*|$ to
        $\tau_j|\beta_j^*| = o(1)$
        for asymptotic normality of $\hbeta_j$, so that asymptotic normality
        may fail for at most a finite number of covariates $j\in[p]$.
\end{itemize}

The pivotal expressions enjoying approximate normality, as well as
the estimates $(\hat\sigma^2,\hat\aaa^2)$
developed in the paper, do not depend on a particular nonlinear function
$F$ in the model \eqref{single-index}.
Consequently, approximate normality (as well as the resuting confidence intervals),
and the validity of $(\hat\sigma^2,\hat\aaa^2)$ for estimating their targets,
are maintained if the conditional law of $y_i|\bm x_i$ is replaced by another.
This robustness to a particular model for $y_i|\bm x_i$ stands in contrast
to procedures aimed at estimating the scalar solutions to the systems
\eqref{sur_system} or \eqref{salehi_system}
by assuming a specific generalized linear model for $y_i|\bm x_i$
(cf. \cite{yadlowsky2021sloe,sawaya2023statistical} where such procedures
are developed).
This phenomenon
is illustrated in \Cref{table:LS} where simulations
show that
approximate normality of pivotal quantities and the validity of $\hat\aaa$ are
preserved if we change the law of $y_i|\bm x_i$ from a linear model
to a classification model (logistic or 1-bit compressed sensing)
or to a Poisson model.

Some of the techniques used below to derive results of the form \eqref{peek}
have been used previously to derive asymptotic normality results
in linear models for penalized least-squares estimators
\cite{bellec_zhang2019debiasing_adjust,bellec_zhang2019second_poincare}
and robust/regularized M-estimators
\cite{bellec2021derivatives,bellec2021asymptotic}.
The present paper builds upon these techniques to tackle single-index models,
and we discuss after \eqref{317} below how the arguments required 
here depart from the known results used in linear models.
While the factor $\hat r/\hat v$ of the Gaussian part $Z_j$
in \eqref{peek} is reminiscent of the construction of confidence intervals
for linear models in 
\cite{el_karoui2018impact,bellec_zhang2019debiasing_adjust,bellec_zhang2019second_poincare,bellec2021asymptotic}, this paper introduces the new adjustments $(\hat t^2,\hat\aaa^2,\hat\sigma^2)$, defined in \eqref{hat_v_r_t} below,
 which captures the interplay between the non-linearity of the model \eqref{single-index} and the M-estimator \eqref{hbeta}.

\subsection{Organization}
\Cref{sec:assum} states our working assumptions.
\Cref{sec:derivatives} obtains formula for the derivatives
of $\bm\hbeta$ with respect to $\bm X$ and defines the observable
adjustments (e.g., $\hat r,\hat v, \hat t^2$ in \eqref{peek}) and related notation that will be used throughout the paper.
\Cref{sec:main-strongly-convex} states the main results in the paper
regarding confidence intervals, proximal mapping representations
for $\bm \hbeta$ and for the predicted values $\bm X\bm\hbeta$,
and estimation of the correlation of $\bm\hbeta$ with the index.
\Cref{sec:unregularized} develops similar results for unregularized
M-estimation. \Cref{sec:simulations} presents examples and simulations.
The proofs are delayed to the appendix.

\subsection{Notation}
For an event $E$, denote by $I_E$ or $I\{E\}$ its indicator function.
Vectors are denoted by bold lowercase letters and matrices by
bold uppercase.
The euclidean norm of a vector $\bm v$ is denoted by $\|\bm v\|$,
the operator norm (largest singular value) of a matrix $\bm M$ by $\|\bm M\|_{op}$
and its Frobenius norm by $\|\bm M\|_F$. If $\bm S$ is a symmetric matrix,
$\lambda_{\min}(\bm S)$ is its smallest eigenvalue. If $\bm S,\bm T$ are
two symmetric matrices, we write $\bm S \preceq \bm T$ if and only if $\bm T-\bm S$ is positive semi-definite. The identity matrix of size $p\times p$ is $\bm I_p$.
Convergence in distribution is denoted by $\to^d$ and convergence in probability by $\to^{\P}$.

\section{Assumptions}
\label{sec:assum}

We now come back to the single index model
\eqref{single-index} with unknown function $F$.
That is, we assume from now on
$$
    y_i = F(\bm{x}_i^T\bm w, U_i),
    \qquad
    U_i \text{ independent of } \bm x_i,
    \qquad
    \Var[\bm x_i^T\bm w]
    =
    \bm w^T \bm \Sigma \bm w = 1.
$$
Each of our results will require a subset of the following assumptions.
Let $\mathcal Y\subset \R$ be the set of allowed values for $y_i$,
i.e., a set such that $\P(y_i\in \mathcal Y)=1$.
For instance, $\mathcal Y=\R$ for continuous response regression
and $\mathcal Y=\{\pm 1\}$ or $\mathcal Y = \{0, 1\}$
in binary classification.

\begin{assumption}
    \label{assum}
    Let $\delta>0,\kappa \ge 1$ be constants independent of $n,p$.
    Consider the model \eqref{single-index}
    for some unknown nonrandom $\bm{w}\in\R^p$ with $\Var[\bm x_i^T\bm w] = 1$.
    Assume that $\frac{1}{2\delta}\le \frac p n\le \frac{1}{\delta}$,
    that $\bm{X}$ has iid $N(\bm{0},\bm{\Sigma})$ rows for $\bm \Sigma$ with 
    condition number bounded as
    $(\|\bm{\Sigma}\|_{op}\|\bm \Sigma^{-1}\|_{op})\le \kappa$,
    that $\ell_{y_0}(\cdot), g(\cdot)$ are convex
    with $\ell_{y_0}(\cdot)$ continuously differentiable
    and $\ell_{y_0}'(\cdot)$ 1-Lipschitz for any fixed $y_0\in\mathcal Y$.
\end{assumption}

\begin{assumption}
    \label{assumStrongConvex}
    For all $\bm b, \bm{\tilde b}\in\R^p$
    we have
    $g(\bm 0)\le g(\bm b)$ and
    for some constant $\tau>0$ independent of $n,p$,
    \begin{equation}
        (\bm{b}-\bm{\tilde b})^T(\bm d - \bm{\tilde d}) \ge \tau \|\bm \Sigma^{1/2}(\bm b- \bm{\tilde b})\|^2,
        \qquad
        \forall \bm d\in,\partial g(\bm b),\bm{\tilde d}\in\partial g(\bm{\tilde b})
    .
    \label{eqStrongConvex}
    \end{equation}
    The loss satisfies $\mathbb P(\ell_{y_i}'(0)=0)=0$
    where $y_i$ satisfies the model \eqref{single-index}.
\end{assumption}
In \eqref{eqStrongConvex} above, $\partial g(\bm b) = \{\bm d\in\R^p: \forall \bm {\tilde b}, g(\bm{\tilde b})\ge g(\bm b) + \bm d^T(\bm{\tilde b}-\bm b)\}$
denotes the subdifferential of $g$ at $\bm b \in\R^p$.
The assumption that $\ell_{y_i}'(0)=0$ happens with probability 0 is mild: For instance it is always satisfied for the logistic loss or the probit
negative log-likelihood. The goal if this assumption is to
ensure
\begin{equation}
    \textstyle
\P\bigl(
\sum_{i=1}^n \ell_{y_i}'(\bm x_i^T \bm\hbeta)^2 > 0
\bigr) = 1.
\label{eq:psi-ne-0-proba-1}
\end{equation}
Indeed, in the event
$\sum_{i=1}^n \ell_{y_i}'(\bm x_i^T \bm\hbeta)^2 = 0$, it must be
that $\bm 0\in \partial g(\bm \hbeta)$ so that $\bm \hbeta$ is a minimizer 
of $g$, so that $\bm\hbeta=\bm 0$ since $\bm 0$ is the only minimizer of $g$
by strong convexity. This implies $\sum_{i=1}^n\ell_{y_i}'(0)^2=0$ which has probability 0.
Thus,  \eqref{eq:psi-ne-0-proba-1} holds under \Cref{assumStrongConvex};
this allows us to divide by $\sum_{i=1}^n \ell_{y_i}'(\bm x_i^T \bm\hbeta)^2$
with probability one.

The strongly-convex
\Cref{assumStrongConvex} simplifies the analysis significantly,
especially in the $p>n$ regime. This assumption lets us provably establish
\eqref{peek} as well as
consistency results for the newly proposed estimators,
and lets us exhibit the underlying structure of single-index model regression,
without much technical argument in the proofs.
We believe that relaxing strong convexity on the penalty
should currently be within reach for L1 penalties 
(for instance for the logistic
Lasso), although under additional and suboptimal assumptions granting
that $\|\hbbeta\|_0/n$ is smaller than a constant with high-probability,
using arguments made in linear regression under assumption (iii)
of Theorem 2.1 of \cite{bellec2020out_of_sample}.
For the logistic Lasso, $\|\hbbeta\|_0/n$ may be controlled for instance
by \cite[Proposition 3.7]{bellec2019first} or references therein.
We believe that strong convexity is an artefact of the proof for some
penalty functions, although not all. For instance, with the logistic loss
and if the penalty
$g$ in \eqref{hbeta} encodes
a convex cone constraint (e.g.,
that all $p$ coefficients are non-negative), then
$\hbbeta$ does not exist when $p/n$ is too large \cite{han2022gaussian}.
Strong convexity makes sure that $\hbbeta$ always exists.
Beyond existence of $\hbbeta$, the Lasso problem \cite{miolane2018distribution}
exhibits a precise phase transition regarding sparsity levels
below which $\|\bm\Sigma^{1/2}(\hbbeta-\bm\beta^*)\|^2$ is uniformly bounded.
In the present paper, strong convexity also ensures that
$\|\bm\Sigma^{1/2}\hbbeta\|^2$ and several other key quantities
are uniformly bounded (cf. \eqref{eq:bound-A-op-recap-before-proof}-\eqref{eq:bound-htheta} in the proof). 
We conjecture that if the sequence of regression problems
is such that these quantities are uniformly bounded from above and below
in probability, then the conclusions of the paper still apply.

Alternatively to strong convexity of the penalty
in \Cref{assumStrongConvex}, we will derive results
for unregularized M-estimation when $p<n$ under the following
assumptions.

\begin{assumption}
    \label{assumLowDim}
    Assume $\frac{1}{2\delta}\le \frac p n\le \frac{1}{\delta}<1$
    and that the penalty is $g=0$.
\end{assumption}

\begin{assumption}
    \label{assumExtra}
    Assume that $\ell_{y_0}$ is twice continuously differentiable
    for all $y_0\in \mathcal Y$, that
    $\sup_{u\in\R}\max_{y_0\in\mathcal Y}|\ell_{y_0}'(u)|\le1$,
    and that the maps
    $u\mapsto \min_{y_0\in\mathcal Y}(\ell_{y_0}'(u))^2$
    and
    $u\mapsto \min_{y_0\in\mathcal Y}\ell_{y_0}''(u)$
    are both continuous and valued in the open interval $(0,+\infty)$.
    Here, the loss $\ell$ (as a function $\mathcal Y\times \R\to\R$)
    is assumed independent of $n,p$.
\end{assumption}

The requirement in \Cref{assumExtra} that 
$u\mapsto \min_{y_0\in\mathcal Y}\ell_{y_0}''(u)$ is positive and continuous
can be interpreted as local strong convexity, since $\ell_{y_0}$
is then strongly-convex on every compact.
Relaxing such local strong convexity involves technical details,
in particular to bound from below $\trace[\bm V]$ where $\bm V$ is defined in
\Cref{thm:derivatives} below. This requires additional technical arguments that
we do not pursue here, but that have been laid out in linear models for the
Huber loss and its variants in \cite{koriyama2023out}.

\section{Derivatives and observable adjustments}
\label{sec:derivatives}

Let us define some notation
used throughout.
Denote by $\partial g(\bm b) \subset \R^p$ the subdifferential
of a convex function $g:\R^p\to\R$ at a point $\bm b\in\R^p$.
The KKT conditions of the minimization problem \eqref{hbeta}
with convex penalty $g$ and convex loss $\ell_{y_i}$ read
\begin{equation}
\bm{X}^T\bm \hpsi (\bm y,\bm X)\in n \partial g(\bm \hbeta(\bm y,\bm X))
\end{equation}
where $\bm{y}\in\R^n$ is the response vector with components $y_1,...,y_n$,
$\bm{X}$ is the design matrix with rows $\bm x_1^T,...,\bm x_n^T$,
and $\bm{\hpsi}(\bm y,\bm X)\in\R^n$ has components
$\hpsi_i(\bm{y},\bm X) = - \ell_{y_i}'(\bm x_i^T\bm \hbeta(\bm y,\bm X))$.
The minus sign here is used so that, in the special case of linear models with the square loss
$\ell_{y_i}(u) = (y_i-u)^2/2$, the usual $i$-th residual
is $\psi_i(\bm y,\bm X)=y_i - \bm x_i^T\bm \hbeta$.
Similarly to $\bm\hbeta(\bm y, \bm X)$ in \eqref{hbeta},
we view $\bm\hpsi(\bm y,\bm X)$ as a function of $(\bm y, \bm X)$,
i.e., $\bm\hpsi : \R^n \times \R^{n\times p} \to \R^n$.
We let $\ell'$ act componentwise and denote
by $\ell_{\bm y}'(\bm u)\in\R^n$ the vector
with components $\ell_{y_i}'(u_i)$ for every $\bm{u},\bm y\in\R^n$,
so that 
\begin{equation}
\bm{\hpsi}(\bm y,\bm X) = - \ell_{\bm y}'\bigl(\bm X\bm \hbeta(\bm y,\bm X)\bigr).
\label{eq:def-psi}
\end{equation}
Similarly, $\ell''$ acts componentwise on vectors in $\R^n$
so that $\ell_{\bm y}''(\bm u)$ has components
$[\ell_{\bm{y}}''(\bm u)]_i = \ell_{y_i}''(u_i)$ for $\bm y,\bm u\in\R^n$
and each $i=1,...,n$.
This notation highlights that derivatives will always be 
with respect to $\bm{u}$ for a fixed value of $\bm y$.
If the context is clear, we write simply $\bm{\hbeta}$ for $\bm \hbeta(\bm y,\bm X)$
and $\bm{\hpsi}$ for $\bm \hpsi(\bm y,\bm X)$; in this case
the functions $\bm{\hbeta},\bm \hpsi$ and their derivatives are
implicitly taken at the observed data $(\bm{y},\bm X)$.

The observable adjustments $(\hat t,\hat v)$ in \eqref{peek} are
constructed from the derivatives of $\bm\hbeta$ with respect to $\bm X$.
The following proposition provides the structure of these derivatives,
extending the corresponding result for linear models 
\cite[Theorem 1]{bellec2021derivatives} to the present setting.

\begin{restatable}{proposition}{thmDerivatives}
    \label{thm:derivatives}
    Assume that {$\ell_{y_0}$ is convex with} $\ell_{y_0}'$ 1-Lipschitz
    for all $y_0\in\mathcal Y$ and that
    \eqref{eqStrongConvex} holds
    for some some $\tau > 0$ and positive definite $\bm{\Sigma}\in\R^{p\times p}$.
    Then
    for any fixed $\bm{y}\in \mathcal Y^n$ 
    the function $\bm{\hbeta}(\bm y, \cdot)$ is differentiable
    almost everywhere in $\R^{n\times p}$ with
    derivatives
    $$
    (\partial/\partial x_{ij})\bm{\hbeta}(\bm{y}, \bm X)
    = \bm{\hat{A}}\bigl(\bm{e}_j \hpsi_i  - \bm X^T\bm D \bm{e}_i \hbeta_j\bigr)
    $$
    where $\bm{D}\defas\diag(\ell_{\bm{y}}''(\bm{X}\bm{\hbeta}))$
    and some matrix $\bm{\hat{A}}\in\R^{p\times p}$ with
    \begin{equation}
        \label{eq:boundA}
        \|\bm{\Sigma}^{1/2} \bm{\hat{A}}\bm \Sigma^{1/2}\|_{op}
        \le (n\tau)^{-1}
    \end{equation}
    and such that the following holds:
    If $\bm D \bm X \bm \hbeta \ne \bm \hpsi$ we have
    \begin{align}
        0\le \quad\df\quad &\le n
        \qquad\qquad\text{ where }
        \df \defas \trace[\bm{X}\bm{\hat{A}} \bm X^\top \bm D],
        \label{eq:def-V-df-bound-ideal-case}
        \\
        \trace[\bm D] / \bigl(1+\hat c\bigr)\le
        \,\trace[\bm V]~
        &\le
        \trace[\bm D] 
        \qquad\text{ where }
        \bm{V}\defas\bm{D}-\bm{D}\bm{X}\bm{\hat{A}}\bm{X}^T\bm{D},
        \nonumber
    \end{align}
    where $\hat c= \frac{1}{n\tau} \|\bm D^{1/2}\bm X\bm\Sigma^{-1/2}\|_{op}^2$,
    while if $\bm D \bm X \bm \hbeta = \bm \hpsi$ we have the slightly
    weaker
    \begin{align}
        -\hat c \le \quad\df\quad &\le n +\hat c,
        \label{eq:bound-df}
        \\
        -4\hat c
        +
        \trace[\bm D] / (1+ \hat c  )
        \le
        \,\trace[\bm V]~
           &\le
        \trace[\bm D] 
        +4\hat c.
        \label{eq:bound-trV-D}
    \end{align}
\end{restatable}

The proof is given in \Cref{sec:proof_derivatives}.
Closed form expressions for
$\hbA$ are available in some cases
using known techniques:
for Ridge regularization 
it is given after \eqref{KKT-ridge}
in \Cref{sec:simulations},
for L1 regularization before 
\eqref{df-L1},
and for the group-lasso 
penalty the closed-form formula
can be obtained by the same argument
as in 
\cite[Propositoin 4.2]{bellec_zhang2019second_poincare}.
If $\bm\hpsi \ne \bm D\bm X\bm \hbeta$, which we expect to happen
in practice, \eqref{eq:bound-df}-\eqref{eq:bound-trV-D} become the simpler
bounds in \eqref{eq:def-V-df-bound-ideal-case}.
This shows that
$\trace[\bm D]=\sum_{i=1}^n\ell_{y_i}''(\bm x_i^T\bm\hbeta)$ and $\trace[\bm V]$ are of the same order,
since $\hat c$ is of constant order by \eqref{eq:7-norm-G} below.
It is unclear why the special case $\bm\hpsi = \bm D\bm X\bm \hbeta$
needs to be handled separately, and we believe that it is 
an artefact of the proof. If $\bm\hpsi = \bm D\bm X\bm \hbeta$
we still obtain the slightly weaker \eqref{eq:bound-df}-\eqref{eq:bound-trV-D}, which is sufficient for our purpose.

By \Cref{thm:derivatives},
since $\bm\hpsi$ in \eqref{eq:def-psi} is given by
$\bm{\hpsi}
=
-\ell_{\bm{y}}'(\bm{X}\bm{\hbeta})
$ we find by the chain rule
$
(\partial/\partial x_{ij})\bm{\hpsi}
=
\bm{D}
[
- \bm{X} \bm{\hat{A}} \bm{e}_j \hpsi_i 
- (\bm{I}_n- \bm{X} \bm{\hat{A}} \bm{X}^T \bm{D}) \bm{e}_i \hbeta_j
]
$.
Summarizing the derivatives of both $\bm\hpsi$ and $\bm\hbeta$ side by side, 
for all $i\in[n], j\in[p]$,
\begin{equation}
    \begin{aligned}
    (\partial/\partial x_{ij})\bm{\hbeta}(\bm y,\bm X)
    &=
    \bm{\hat{A}}\bm{e}_j \hpsi_i  - \bm{\hat{A}} \bm X^T\bm D \bm{e}_i \hbeta_j
    ,
    \\
    (\partial/\partial x_{ij})\bm{\hpsi}(\bm y,\bm X)
    &= - \bm{D} \bm{X}\bm{\hat{A}}\bm{e}_j \hpsi_i
- \bm{V} \bm{e}_i \hbeta_j.
    \end{aligned}
    \label{eq:derivatives-psi-hbeta}
\end{equation}

The random quantities $\trace[\bm V]$, $\df$ and $\trace[\bm\Sigma \hbA]$ defined
using the above derivatives are intimately linked to certain quadratic
functions of $(\hbbeta,\bX\hbbeta,\bm\hpsi,\bX^\top\bm\hpsi)$ in several approximations that we
now describe informally. For simplicity of presentation, it is best
to perform a change of variable to transform the initial problem with
data $(\bm X,\by)$ where $\bm x_i\sim N(\bm 0,\bm \Sigma)$ to an isotropic
problem with data $(\bm G, \by)$ where $\bm G$ has iid $N(0,1)$ entries.
Define $\bm G = \bm X \bm \Sigma^{-1/2}$ as well as
$\bm\htheta = \bm\Sigma^{1/2}\bm\hbeta$,  $\bm\theta^*=\bm\Sigma^{1/2}\bm w$
and $\bm A = \bm \Sigma^{1/2} \hbA \bm \Sigma^{1/2}$.
The vector $\bm\theta^*$plays the role of the index after this change of variable
(in the proof in \eqref{eq:change-of-variable-G}-\eqref{eq:rewrite-after-change-variable}, we additionally rotate the coordinate system so that the rotated index becomes the first canonical basis vector, but this additional rotation is not useful for the discussion of the current paragraph).
Define
$\bm G^\perp = \bm G (\bm I_p - \bm\theta^*(\bm\theta^*)^T)$ and
$\bm\htheta^\perp=(\bm I_p - \bm\theta^*(\bm\theta^*)^T)\bm\htheta$ 
where $\bm I_p - \bm\theta^*(\bm\theta^*)^T$ is the orthogonal projection
on the orthogonal complement of the span of $\bm\theta^*$.
Then $(\by,\bm G\bm\theta^*)$ is independent of 
$\bm G^\perp$ and
\begin{align}
    p\|\bm \hpsi\|^2
    &\approx
    \|(\bm G^\perp)^T \bm \hpsi + \trace[\bm V]\bm \htheta^\perp \|^2 
    &&\text{by Prop.~\ref{prop26} with }\bm u= \bm \hpsi, \bm Z= \bm G^\perp,
    \label{main_Gamma_2}
    \\
    n \|\bm\htheta^\perp\|^2
    &\approx
    \| \bm G^\perp \bm\htheta - \trace[\bm A] \bm \hpsi\|^2
    &&\text{by Prop.~\ref{prop26} with }\bm u= \bm \htheta^\perp, \bm Z^T= \bm G^\perp.
    \label{main_Gamma_4}
\end{align}
On the other hand, by
\Cref{prop25} applied to either $\bm Z = \bm G^\perp$ or $\bm Z = (\bm G^\perp)^T$,
\begin{align}
    \bm \hpsi^T \bm G^\perp \bm \htheta
    &\approx
    -\trace[\bm V] \|\bm \htheta^\perp\|^2
    + \trace[\bm A] \|\bm \hpsi\|^2,
    &&
    \text{ with }\bm u= \bm \hpsi,\bm f = \bm\htheta^\perp,
    \label{main_Gamma_1_star}
    \\
    p\|\bm \hpsi\|^2
    &\approx
    \|(\bm G^\perp)^T \bm \hpsi\|^2
    + \trace[\bm V]
    \bm \hpsi^T \bm G^\perp \bm\htheta
    + \df \|\bm \hpsi\|^2
    &&
    \text{ with }\bm u= (\bm G^\perp)^T\bm \hpsi,\bm f = \bm\hpsi,
    \label{main_Gamma_3}
    \\
    n \|\bm \htheta^\perp\|^2
    &\approx 
    \|\bm G^\perp \bm\htheta\|^2
    - \trace[\bm A] \bm \hpsi^T \bm G^\perp \bm\theta
    + \df \|\bm \htheta^\perp\|^2
    &&
    \text{ with }\bm u= \bm G^\perp\bm\htheta,\bm f = \bm\htheta^\perp,
    \label{main_Gamma_5_star}
    \\
    0 & \approx
    (\bm G \bm\theta^*)^T
    (\bm G^\perp \bm\htheta - \trace[\bm A] \bm \hpsi)
    &&
    \text{ with }\bm u= \bm G\bm\theta^*,\bm f = \bm\htheta^\perp.
    \label{main_Gamma_6_star}
\end{align}
These informal approximations are established and used in several proofs
of the present paper, see in particular \eqref{Rem_i} and \Cref{lemma:fiveqs}
where
$\Gamma_1^*$ corresponds to \eqref{main_Gamma_1_star},
$\Gamma_2$ to \eqref{main_Gamma_2},
$\Gamma_3$ to \eqref{main_Gamma_3},
$\Gamma_5^*$ to \eqref{main_Gamma_5_star} and
$\Gamma_6^*$ to \eqref{main_Gamma_6_star}.
The proof in
\Cref{sec:proof_Gammas} uses specific combinations of these approximations
to arrive at the estimators
$(\hat\aaa^2, \hat\sigma^2)$ defined in \eqref{hat_v_r_t} below
which provides observable estimates of the unknown quantities
$\aaa_*^2 = ((\bm\theta^*)^T\bm\htheta)^2 = (\bm w^T \bm\Sigma\bm\hbeta)^2$
(inner product with the index, squared) and
$\sigma_*^2 = \|\bm\htheta^\perp\|^2 = \|\bm\Sigma^{1/2}\bm\hbeta\|^2 - \aaa_*^2$ (squared norm of the projection onto the orthogonal complement of the index).
For convenience, we gather here several notable combinations of the
above approximations:
\begin{align}
    0
    &\approx (\trace[\bm A]\trace[\bm V] - \df )\|\bm\hpsi\|^2
    &&\text{by } \eqref{main_Gamma_2}-\eqref{main_Gamma_3}+\trace[\bm V]\eqref{main_Gamma_1_star},
    \label{314}
    \\
    0
    &\approx (\trace[\bm A]\trace[\bm V]- \df)\|\bm\htheta^\perp\|^2
    &&\text{by } \eqref{main_Gamma_4}-\eqref{main_Gamma_5_star}-\trace[\bm A]\eqref{main_Gamma_1_star}
\end{align}
which extends the identity $\trace[\bm A]\trace[\bm V]\approx \df$ known in linear models \cite{bellec2021derivatives}. Furthermore,
\begin{align}
    \trace[\bm A]^2 \|\bm \hpsi\|^4
    &\approx
    (p\|\bm \hpsi\|^2 - \|(\bm G^\perp)^T\bm \hpsi\|^2)
    \|\bm\htheta^\perp\|^2
    + (\bm \hpsi^T \bm G^\perp\bm\htheta)^2,
    \label{316}
    \\
    \trace[\bm V]^2 \|\bm \htheta^\perp\|^4
    &\approx (n\|\bm\htheta^\perp\|^2 - \|\bm G^\perp \bm\htheta\|^2)
    \|\bm \hpsi\|^2
    + (\bm \hpsi^T \bm G^\perp\bm\htheta)^2,
    \label{317}
\end{align}
which follows from $\trace[\bm A]^2\|\bm\hpsi\|^4\approx
(\trace[\bm V]\|\bm\htheta^\perp\|^2 + \bm \hpsi^T \bm G^\perp\bm\htheta)^2$
from \eqref{main_Gamma_1_star} combined with
$\|\bm\htheta^\perp\|^2 \eqref{main_Gamma_2}$ to obtain \eqref{316},
and using a symmetric argument to obtain \eqref{317}.
The approximations \eqref{main_Gamma_2}-\eqref{main_Gamma_5_star} are
consequences of the probabilistic inequalities in \Cref{prop25,prop26}
that were previously used in linear models \cite{bellec2020out_of_sample,bellec2021derivatives}.
The proofs of the present paper depart from previous works in linear models
due to the presence of the index $\bm\theta^*$ in \eqref{main_Gamma_2}-\eqref{main_Gamma_6_star}
through its orthogonal complement $\bm I_p - \bm\theta^* (\bm\theta^*)^T$.
In known mean-field asymptotic results,
single-index regularized logistic regression \cite{salehi2019impact}
(see \cite[Section B.8]{loureiro2021capturing} for generalizations)
is characterized by the system \eqref{salehi_system} with 6 unknowns---two
more
compared to the linear model
setting of \cite{thrampoulidis2018precise} featuring a system with 4 unknowns.
Here, in the same vein, the single-index model and corresponding
asymptotic normality \eqref{peek} require additional
scalar adjustments compared to the linear model results
\cite{bellec2020out_of_sample,bellec2021derivatives,bellec2021asymptotic}.
These new scalar adjustments, $(\hat a^2,\hat t^2,\hat \sigma^2)$
defined in \eqref{hat_v_r_t} below, are developed to estimate
the unknown quantities $(a_*^2, t_*^2, \sigma_*^2)$ where
\begin{equation}
    \aaa_* = (\bm\theta^*)^T \bm\htheta,
    \quad
    \sigma_*^2 = \|(\bm I_p - \bm\theta^*(\bm\theta^*)^T)\bm\htheta\|^2,
    \quad
    t_* = (\bm\theta^*)^T(\bm G^T \bm\hpsi + \trace[\bm V]\bm\htheta) /n.
    \label{unknown}
\end{equation}
The above quantities are unknown because the index direction
$\bm\theta_*$ used in the inner products is unknown.
Estimating $(\aaa_*,\sigma_*^2,t_*)$ is key in order to obtain approximate
normality results of the form \eqref{peek} with observable
parameters. This will be clear in light of
\eqref{eq:intermediary-1-over-p}-\eqref{eq:intermediary-t-t*} below:
the proof first obtains an approximate normality result involving $t_*$,
and then use the approximation $t_*^2 \approx \hat t^2$ to obtain
the approximate normality result \eqref{peek} that involves
only observable coefficients.

Let us illustrate how the approximations \eqref{main_Gamma_2}-\eqref{main_Gamma_6_star} can be used to estimate $(a_*^2,t_*^2,\sigma_*^2)$
in the above isotropic setup.
From \eqref{main_Gamma_2} we get
$$
    \|\bm G^T \bm \hpsi + \trace[\bm V]\bm \htheta \|^2 
    -
    p\|\bm \hpsi\|^2
    \approx
    t_*^2 n^2
$$
so that the left-hand side divided by $n^2$ provides an estimate of
$t_*^2$ (this estimate is denoted by $\tilde t^2$ in \eqref{eq:tilde-t-aaa-sigma} below and in the proof).
To estimate $a_*^2$, write
$a_*t_* n =\bm\htheta^T \bm\theta^* (\bm\theta^*)^T(\bm G^T \bm \hpsi + \trace[\bm V]\bm \htheta)$ so that
$$
\bm\hpsi^T\bm G\bm\htheta + \trace[\bm V]\|\bm\htheta\|^2
-
a_*t_* n
=
\bm\hpsi^T\bm G^\perp \bm\htheta + \trace[\bm V]\|\bm\htheta^\perp\|^2
\approx \trace[\bm A]\|\bm\hpsi\|^2
$$
by \eqref{main_Gamma_1_star}.
Since $a_*t_*$ is the only quantity of interest, isolate it from the
other terms and take the square to find
\begin{equation}
    \label{leading_up_to}
(
\bm\hpsi^T\bm G\bm\htheta + \trace[\bm V]\|\bm\htheta\|^2
- \trace[\bm A]\|\bm\hpsi\|^2
)^2
\approx a_*^2 t_*^2 n^2.
\end{equation}
As we already have an estimate $\tilde t^2$ of $t_*^2$,
dividing by $\tilde t^2n^2$ provides an estimate of $a_*^2$.
These estimates are different than the ones we propose in
\eqref{hat_v_r_t} below in that for anisotropic design with covariance
$\bm\Sigma$,
the last two terms in the left-hand side of \eqref{leading_up_to}
depend on $\bm\Sigma$, whereas the estimates in \eqref{hat_v_r_t}
proposed below have no dependence on $\bm\Sigma$ except for
one term in the estimate $\hat t^2$ of $t_*^2$.
The full display of the techniques used to make
\eqref{main_Gamma_2}-\eqref{main_Gamma_6_star} rigorous
and to develop estimates of $(\aaa_*,\sigma_*^2,t_*)$
in the anisotropic setting can be found in the proof in
\Cref{sec:proof_Gammas}.

We now leave aside the isotropic change of variable with notation
$(\bm\htheta,\bm\theta^*,\bm G,\bm G^\perp, \bm A)$ of the previous
two paragraphs,
and come back to the notation
used throughout the main text with estimator
$\bm\hbbeta$, index $\bm w\in\R^p$, design matrix $\bX$
with rows $\bm x_i\sim N(\bm 0,\bm\Sigma)$, and the matrices
$\hbA,\bm V$ and scalar $\df$ of \Cref{thm:derivatives}.

\subsection{Definition of the observable adjustments}
We now define the scalar adjustments proposed
to estimate $\aaa_*^2,\sigma_*^2$ and
to correct the bias and scaling of $\bm{\hbeta}$
for estimation and confidence interval about components
of the index $\bm{w}$.             
Define the scalars $(\hat v, \hat r,\hat \gamma, \hat t^2, \hat\aaa^2, \hat\sigma^2)$
by
\begin{equation}
    \label{hat_v_r_t}
    \left\{
    \begin{aligned}
        \hat v &\defas \tfrac{1}{n} \trace[\bm{V}], 
        \\
        \hat r &\defas (\tfrac{1}{n}\|\bm{\hpsi}\|^2)^{1/2},
        \\
        \hat \gamma
        &\defas
        \tfrac{\df}{n\hat v}
        = \tfrac{\df}{\trace[\bm{V}]},
        \\
        \hat t^2 &\defas
        \tfrac{1}{n^2}\|\bm{\Sigma}^{-1/2}\bm{X}^T\bm \hpsi\|^2
        +  \tfrac{2 \hat v}{n} \bm{\hpsi}^T\bm X\bm \hbeta
        + \tfrac{\hat v^2}{n}
        \|\bm{X} \bm \hbeta - \hat\gamma \bm \hpsi\|^2
        -
        \tfrac{p}{n}\hat r^2,
    \\
    \hat \aaa^2 &\defas
    \hat t^{-2}
    \big(
    \tfrac{\hat v}{n}\|\bm{X} \bm \hbeta - \hat\gamma \bm \hpsi\|^2
    + \tfrac{1}{n}\bm{\hpsi}^\top\bm X\bm \hbeta
    - \hat\gamma \hat r^2
    \big)^2
    ,
    \\
    \hat \sigma^2
    &\defas
    \tfrac{1}{n}\|\bm{X}\bm \hbeta - \hat \gamma \bm \hpsi\|^2
    - \hat \aaa^2.
    \end{aligned}
    \right.
\end{equation}
The motivation behind the normalizations $1/n$ and $1/n^2$
seen in some terms in \eqref{hat_v_r_t} above is to make
$(\hat v, \hat r,\hat \gamma, \hat t^2, \hat\aaa^2, \hat\sigma^2)$
of constant order in typical settings.
\Cref{thm:loss-proximal-representation} below shows that
$\hat a^2$ above estimates $\aaa_*^2$ where
$\aaa_* = 
\bm{w}^T\bm\Sigma\bm \hbeta$,
and $\hat \sigma^2$ above estimates 
$\sigma_*^2 = \|\bm\Sigma^{1/2} \bm{\hbeta}\|^2 - \aaa_*^2$,
and $\hat\gamma$ estimates
$\gamma_* = \trace[\bm\Sigma\bm{\hat{A}}]$.
Next, for each covariate index $j\in[p]$, define
the de-biased estimate $\debias_j$ by
\begin{equation}
    \begin{split}
    \debias_j 
    &\defas
    \hbeta_j + \trace[\bm{V}]^{-1} \bm{e}_j^T\bm \Sigma^{-1}\bm{X}^T\bm \hpsi 
    \\&=
    \hbeta_j + \hat v^{-1} \bm{e}_j^T ( n\bm{\Sigma} )^{-1}\bm{X}^T\bm \hpsi
    \end{split}
    \label{debias_j}
\end{equation}
where $\bm{e}_j$ is the $j$-th canonical basis vector in $\R^p$.
This de-biased estimate is valid
for any loss-penalty pair
satisfying \Cref{assum}.
For unregularized M-estimation ($g=0$ in \eqref{hbeta}), it holds $\bm X^T\bm\hpsi=0$ so that the second term in \eqref{debias_j} vanishes
and the knowledge of $\bm\Sigma$ is not needed.
For regluarized M-estimation,
knowledge or estimation
of $\bm e_j^T \bm\Sigma^{-1}\bm X^T$ is required. For L1 regularization, this is possible under additional sparsity assumptions using the techniques developed in the early works on de-biasing the Lasso \cite{ZhangSteph14,JavanmardM14a,GeerBR14}, see for instance
\cite[Section 2.3]{bellec_zhang2019debiasing_adjust}.
Finally, throughout the paper, let
\begin{equation}
    \label{Omega_jj}
    \Omega_{jj} \defas (\bm{\Sigma}^{-1})_{jj}
\end{equation}
be the $j$-th diagonal element of $\bm \Sigma^{-1}$.


%

Alternatively to the adjustments
$(\hat \gamma, \hat t, \hat\aaa^2, \hat\sigma^2)$
in \eqref{hat_v_r_t},
the quantities $\frac 1 n \|\bm{X} \bm \hbeta - \hat\gamma \bm \hpsi\|^2$
and $\hat\gamma$ in the expressions for $(\hat t,\hat\aaa^2,\hat\sigma^2)$
in \eqref{hat_v_r_t}
may be replaced by $\|\bm{\Sigma}^{1/2}\bm \hbeta\|^2$
and $\gamma_*=\trace[\bm{\Sigma} \bm{\hat{A}}]$
thanks to the approximations
$\frac 1 n \|\bm{X} \bm \hbeta - \hat\gamma \bm \hpsi\|^2\approx \|\bm\Sigma^{1/2}\bm\hbeta\|^2$
and $\hat\gamma \approx \gamma_*$ justified in \Cref{thm:adjustments-approximation} below. This
gives the alternative estimates
\begin{equation}
    \label{eq:tilde-t-aaa-sigma}
    \left\{
    \begin{aligned}
    \tilde t^2 &= \|\bm{\Sigma}^{-1/2}\bm X^T\bm \hpsi/n + \hat v \bm \Sigma^{1/2} \bm\hbeta\|^2 - (p/n)\hat r^2,
    \\
    \tilde \aaa^2 &= \tilde t^{-2}
    \bigl(
    \hat v \|\bm{\Sigma}^{1/2}\bm \hbeta\|^2
    + \bm{\hpsi}^T \bm X\bm \hbeta/n - \gamma_*\hat r^2
    \bigr)^2,
    \\
    \tilde\sigma^2 &= \|\bm{\Sigma}^{1/2}\bm \hbeta\|^2
    -\tilde \aaa^2.
    \end{aligned}
    \right.
\end{equation}
The expressions $(\hat t^2,\hat\aaa^2,\hat\sigma^2)$ in \eqref{hat_v_r_t} are preferred
when the covariance $\bm\Sigma$ is unknown, as
$\bm{\Sigma}$ only occurs in \eqref{hat_v_r_t} in the expression of $\hat t^2$.
For unregularized $M$-estimation (penalty $g=0$),
the term $\bm{\Sigma}^{-1/2}\bm X^T\bm \hpsi$ in $\hat t$ is equal to 0
by the optimality conditions of the optimization problem \eqref{hbeta},
so that $(\hat r, \hat v, \hat \gamma,\hat t^2,\hat\aaa^2,\hat\sigma^2)$
are all computable without any knowledge of the covariance $\bm{\Sigma}$.
The special form of
$(\hat r, \hat v, \hat \gamma,\hat t^2, \hat\aaa^2, \hat\sigma^2,\df)$
in unregularized $M$-estimation is detailed in
\Cref{sec:unregularized}.

\section{Main results under strong convexity}
\label{sec:main-strongly-convex}

Throughout, $\C,\C,...$ denotes absolute constants,
and $\C(\delta),\C(\delta,\tau),...$ denote constants that depend
only on $\delta$ and only on $(\delta,\tau)$, respectively.

\subsection{Confidence intervals for individual components of the index}

\begin{restatable}{theorem}{thmConfidence}
    \label{thm:confidence-intervals-Sigma}
    Let \Cref{assum,assumStrongConvex} be fulfilled
    and let $\Omega_{jj}>0$ as in \eqref{Omega_jj}.
    Then for all $j=1,...,p$, there exists $Z_j\sim N(0,1)$ such that
    \begin{equation}
        \label{eq:thmconfidence-eq1}
        \frac 1 p 
    \sum_{j=1}^p
    \E\Bigl[
    \Bigl(
        \frac{ \sqrt n}{\Omega_{jj}^{1/2}}
    \Bigl(
    \frac{ \hat v } {\hat r}
    \debias_j
    -
    \frac{\pm \hat t} {\hat r}
    w_j
    \Bigr)
    -  Z_j
    \Bigr)^2
    \Bigr]
    \le \frac{ \C(\delta,\tau,\kappa)}{\sqrt p}
    \end{equation}
    where $\pm$ denotes the sign of the unknown scalar
    $t_* 
    \defas
    \bm{w}^T(\trace[\bm{V}] \bm \Sigma \bm{\hbeta} + \bm{X}^T\bm \hpsi)/n$,
    and $\hat t = \max(0,\hat t^2)^{1/2}$.
    If additionally \Cref{assumExtra} holds then for some event $E$
    with $\P(E)\to 1$, we have
    $\max\{\frac{1}{|\hat v|},\hat r\}I_E\le \C(\delta,\tau,\ell)$ almost
    surely and
    \begin{equation}
        \label{eq:second_conclusion_strongly_convex_with_event_E}
        \frac 1 p 
    \sum_{j=1}^p
    \frac{1}{\Omega_{jj}}
    \E\Bigl[
        I_E
    \Bigl|
        \sqrt n
    \Bigl(
    \debias_j
    -
    \frac{\pm \hat t} {\hat v}
    w_j
    \Bigr)
    - 
    \Omega_{jj}^{1/2} 
    Z_j 
    \frac{ \hat r } {\hat v}
    \Bigr|^2
    \Bigr]
    \le \frac{ \C(\delta,\tau,\kappa,\ell)}{\sqrt p}.
    \end{equation}
\end{restatable}

In the single index model \eqref{single-index},
the sign of $\bm{w}$ is not identifiable as replacing $\bm w$ with $-\bm w$
and $F$ with $(a,u)\mapsto F(-a,u)$ in \eqref{single-index}
leaves $\bm{y}$ unchanged.
Consequently, for each covariate $j=1,...,p$ we focus
on confidence intervals for $w_j$
up to an unidentifiable sign denoted by the random variable
``$\pm$'' in the above displays.
Disambiguating and estimating the sign $\pm = \sign(t_*)$
    under additional assumptions is discussed in \Cref{sec:sign} below.
We now describe the confidence intervals that stem from \Cref{thm:confidence-intervals-Sigma}.
By \eqref{eq:thmconfidence-eq1},
for any $\eps>0$, the set
\begin{equation}
J_p(\eps)
=\Bigl\{j\in [p]:
    ~
W_2\Bigl(N(0,1),~
    \Omega_{jj}^{-1/2}\sqrt n \Bigl(\frac{\hat v}{\hat r}\debias_j - \frac{ \pm \hat t}{\hat r}w_j\Bigr)
\Bigr)^2 > \eps
\Bigr\}
\label{J_p_eps}
\end{equation}
has cardinality at most $\sqrt p C(\delta,\tau,\kappa)/\eps$,
where $W_2$ is the 2-Wasserstein distance.
This justifies the approximation
\begin{equation}
\sqrt n
(\hat v \debias_j -  \pm \hat t w_j)
\approx \hat r ~ \Omega_{jj}^{1/2}Z_j
\label{eq:informal-approximation-theorem31}
\end{equation}
for most coordinates.
The quantile-quantile plots in \Cref{table:LS,table:ridge} illustrate 
this normal approximation.
Using the bound \cite[Proposition 1.2]{ross2011fundamental_stein_method} 
between the Kolmogorov distance and the 1-Wasserstein distance,
for any random variable $U$ we have
$\sup_{u\in \R}|\P(U\le u) - \P(N(0,1)\le u)|
\le (2/\pi)^{1/4} \sqrt{W_1(U,N(0,1))}$
where $W_1(U,N(0,1))$ denotes the 1-Wasserstein distance, which is 
always smaller than $W_2(U,N(0,1))$.
This implies that for $J_p(\eps)$ in \eqref{J_p_eps},
$$
\sup_{j\in [p]\setminus J_p(\eps)}
\sup_{\alpha\in(0,1)}
\Big|\mathbb P\Bigl(
    \frac{\hat v}{\hat t} \debias_j 
    -
    \frac{\hat r}{\hat t}
    \frac{z_{\frac\alpha2}}{ \sqrt n}
    \Omega_{jj}^{1/2}
\le
\pm w_j
\le
\frac{\hat v}{\hat t} \debias_j 
+
\frac{\hat r}{\hat t}
\frac{z_{\frac\alpha2}}{\sqrt n}
\Omega_{jj}^{1/2}
\Bigr) - \Bigl(1-\alpha\Bigr)
\Big| \le 
2 \Bigl(\frac{2\eps}{\pi}\Bigr)^{1/4}
$$
where $z_{\alpha/2}$ is the standard normal quantile defined by
$\mathbb P(|N(0,1)|>z_{\frac\alpha2})=\alpha$.
This provides a confidence interval for the $j$-th component $w_j$
of the index $\bm{w}$, up to the unknown sign $\pm=\sign(t_*)$.
By the same argument and taking $\eps=\eps_{n,p}$ depending on $n,p$ and converging to 0, for instance $\eps_{n,p}=1/\log n$, we obtain
$$
\sup_{j\in [p]\setminus J_{n,p}}
\Big|\mathbb P\Bigl(
    \frac{\hat v}{\hat t} \debias_j 
    -
    \frac{\hat r}{\hat t}
    \frac{z_{\frac\alpha2}}{ \sqrt n}
    \Omega_{jj}^{1/2}
\le
\pm w_j
\le
\frac{\hat v}{\hat t} \debias_j 
+
\frac{\hat r}{\hat t}
\frac{z_{\frac\alpha2}}{\sqrt n}
\Omega_{jj}^{1/2}
\Bigr)
- (1-\alpha)
\Big|\to0
$$
for the set $J_{n,p}=J_p(\eps_{n,p})\subset [p]$ which has cardinality at most
$\sqrt{p}\C(\delta,\tau,\kappa) / \eps_{n,p}$.

Unpacking the proof of \Cref{thm:confidence-intervals-Sigma} reveals
that \eqref{eq:thmconfidence-eq1} is a consequence of
\begin{align}
    \label{eq:intermediary-1-over-p}
    \frac 1 p 
\sum_{j=1}^p
\E\Bigl[
\Bigl(
    \frac{ \sqrt n}{\Omega_{jj}^{1/2}}
\Bigl(
\frac{ \hat v } {\hat r}
\debias_j
-
\frac{t_*} {\hat r}
w_j
\Bigr)
-  Z_j
\Bigr)^2
\Bigr]
&\le \frac{ \C(\delta,\tau,\kappa)}{p},
\\
\E\Bigl[\frac{1}{\hat r^2}|t_*^2 - \hat t^2|\Bigr]
&\le \frac{\C(\delta,\tau)}{\sqrt p}.
    \label{eq:intermediary-t-t*}
\end{align}
The first line is of the same form as
\eqref{eq:thmconfidence-eq1} with $\pm\hat t$ replaced
by $t_*$ and a right-hand side of order $1/p$, much smaller
than the right-hand side of \eqref{eq:thmconfidence-eq1}.
The right-hand side of order $1/\sqrt p$ in \eqref{eq:thmconfidence-eq1} is paid due to the approximation $t_*^2\approx \hat t^2$ in
\eqref{eq:intermediary-t-t*}, which features 
an error term of order $1/\sqrt p$. The approximation
\eqref{eq:intermediary-t-t*} is only used in the terms of \eqref{eq:thmconfidence-eq1} for which $w_j\ne 0$, so that
if $\mathcal N = \{j=1,...,p: w_j =0\}$ denotes the set of null covariates,
\eqref{eq:intermediary-1-over-p} implies the bound
\begin{equation}
    \frac 1 p 
\sum_{j \in \mathcal N}
\E\Bigl[
\Bigl(
    \frac{ \sqrt n}{\Omega_{jj}^{1/2}}
\Bigl(
\frac{ \hat v } {\hat r}
\debias_j
\Bigr)
-  Z_j
\Bigr)^2
\Bigr]
\le \frac{ \C(\delta,\tau,\kappa)}{p}
\label{eq:bound-1/p-null}
\end{equation}
with a smaller right-hand side of order $\frac{1}{p}$.
To perform an hypothesis test of
$$
H_0: w_j = 0
\qquad
\text{ against }
\qquad
H_1: |w_j| > 0
$$
at level $\alpha\in(0,1)$,
the test that rejects $H_0$ if 
\begin{equation}
\label{eq:when-to-reject}    
\tfrac{\hat v}{\hat r}\sqrt n | \debias_j| > z_{\frac\alpha 2}\Omega_{jj}^{1/2}
\end{equation}
has type I error in $[\alpha-\epsilon,\alpha+\epsilon]$
for all components $j\in [p]\setminus J_p^0$ where $|J_p^0|\le \C(\delta,\tau,\kappa)$, i.e., the desired type I error holds for all null covariates except
a finite number at most.

That the asymptotic normality \eqref{eq:informal-approximation-theorem31}
must fail for some coordinates is a real
phenomenon, due to the variance estimate being wrong for 
certain coordinates for the Lasso (see the discussion regarding
possible ``Variance Spike'' discussed in Section 3.7 of
\cite{bellec_zhang2019second_poincare}).

\subsection{Proximal mapping representation for $\hbeta_j$}

For isotropic designs with $\bm{\Sigma} = \frac1n\bm I_p$,
the de-biased estimate \eqref{debias_j} reduces in vector form to
$\bm{\debias}
= \bm{\hbeta} + \hat v^{-1} \bm X^T \bm \hpsi
$
with
$\bm X^T \bm \hpsi \in n  \partial g(\bm \hbeta) $,
where $ \partial g(\bm{\hbeta})$ is the subdifferential of the penalty
$g$ at $\bm\hbeta$. By the definition of the proximal operator in $\R^p$, we thus have for any $\bm b\in\R^p$
$$
\bm{\hbeta} + \hat v^{-1} \bm X^T\bm \hpsi = \bm b
\qquad
\text{ iff }
\qquad
\bm{\hbeta} = \prox\bigl[(n/\hat v)g\bigr]\bigl(\bm b\bigr).
$$
That is, in this isotropic setting, \Cref{thm:confidence-intervals-Sigma} lets us express $\bm{\hbeta}$ as a proximal operator of the penalty
function $g$ scaled by $n/\hat v$. The following result makes such proximal
approximation precise in the case of separable penalty.

\begin{restatable}{corollary}{thmProxPenalty}
    \label{thm:proximal-hbeta}
    Let \Cref{assum,assumStrongConvex} be fulfilled and set $\bm{\Sigma}=\frac 1 n\bm I_p$.
    Assume that the penalty $g$ is separable, of the form
    $g(\bm{b}) = \frac 1 n \sum_{j=1}^p g_j(b_j)$ for convex functions $g_j:\R\to\R$.
    Then
    \begin{equation}
    \frac 1 p
    \sum_{j=1}^p
    \E\Bigl[
    \frac{\hat v^2}{\hat r^2}
    \Bigl|
        \hbeta_j 
        -
        \prox\Bigl[\frac{1}{\hat v} g_j\Bigr]
    \Bigl(
        \frac{\hat r}{\hat v} Z_j +  \frac{\pm \hat t}{\hat v} w_j
    \Bigr) 
    \Bigr|^2
\Bigr] \le \frac{ \C(\delta,\tau,\kappa)}{\sqrt p}
    \label{eq:penalty-proximal-thm}
    \end{equation}
    where $Z_j\sim N(0,1)$ for each $j\in[p]$
    and $\pm = \sign(t_*)$ as in 
    \Cref{thm:confidence-intervals-Sigma}.
\end{restatable}
\begin{proof}
    By the KKT conditions,
    with $\bm{\Sigma} = \frac 1 n \bm I_p$ and separable penalty $g$,
    $\debias_j \in \hbeta_j + \hat v^{-1}\partial g_j (\hbeta_j)$.
    Since $\Omega_{jj} = n$,
    \Cref{thm:confidence-intervals-Sigma} gives
    that $\hbeta_j + \hat v^{-1}\partial g_j(\hbeta_j)) \ni
    \frac{\hat r}{\hat v}Z_j + (\pm\frac{\hat t}{\hat v})w_j
    + \frac{\hat r}{\hat v}\Rem_j$
    where $\partial g_j(\hbeta_j)$ is the subdifferential of $g_j$
    at $\hbeta_j$ and $\Rem_j$ are such that $\frac{1}{p}\sum_{j=1}^p\E[\Rem_j^2]
    \le \C(\delta,\tau,\kappa)/\sqrt p$.
    Equivalently,
    $\hbeta_j = \prox[\hat v^{-1}g_j](
        \frac{\hat r}{\hat v}Z_j 
        + (\pm\frac{\hat t}{\hat v})w_j
        +\frac{\hat r}{\hat v}\Rem_j )$ 
    by definition of the proximal operator.
    Since $x\mapsto\prox[h](x)$ is 1-Lipschitz for any convex, proper lower semi-continuous function $h$, the left-hand side of
    \Cref{thm:proximal-hbeta} is bounded from above by
    $\frac 1 p \sum_{j=1}^p\E[(\Rem_j)^2]$.
\end{proof}

For $\bm\Sigma =\frac 1 n \bm I_p$, \Cref{thm:proximal-hbeta} provides the proximal approximation
\begin{equation}
    \hbeta_j \approx 
    \prox\Bigl[\frac{1}{\hat v} g_j\Bigr]
    \Bigl(
        \frac{\hat r}{\hat v} Z_j +  \frac{\pm \hat t}{\hat v} w_j
    \Bigr) 
    \label{eq:informal-thm-hbeta-prox}
\end{equation}
for $\hbeta_j$, in an averaged sense over $j\in [p]$.
If $\bm \Sigma = \frac 1 p \bm I_p$ is used (instead of $\bm \Sigma = \frac 1 n \bm I_p$ in \Cref{thm:proximal-hbeta}) and the penalty is
$g(\bm b) = \frac 1 p \sum_{j=1}^p \tilde f(b_j)$, the same argument
yields \eqref{peek}, which is analogous to \eqref{informal-salehi} 
from \cite{salehi2019impact,gerbelot2020asymptotic} with the important difference
that the adjustments $(\hat r,\hat v,\hat t^2)$ are observable
and computed from the data, while the deterministic adjustments
in \eqref{informal-salehi} are not observable.
We note in passing that
\Cref{thm:confidence-intervals-Sigma} is more informative 
than \Cref{thm:proximal-hbeta} because the subgradient is explicit:
If $\bm\Sigma = \frac 1 n \bm I_p$ and $g(\bm b)=\frac 1 n \sum_{j=1}^p g_j(b_j)$,
\Cref{thm:confidence-intervals-Sigma} provides
$$
\hbeta_j + \frac{1}{\hat v} \bm e_j^T\bm X^T\bm \hpsi
\approx
\pm w_j\frac{\hat t}{\hat v}
+Z_j\frac{\hat r}{\hat v}
\qquad
\text{ with }
\qquad
\bm e_j^T\bm X^T\bm \hpsi \in\partial g_j(\bm \hbeta_j).
$$
On the other hand, the information that $\bm e_j^T\bm X^T\bm \hpsi$ is the subgradient
of $g_j$ at $\hbeta_j$ appearing in the KKT conditions of the proximal
operator is not visible from results such as \eqref{eq:informal-thm-hbeta-prox}.

In \Cref{thm:proximal-hbeta}
the index $\bm{w}$ is nonrandom.
If $g_j = g_0$ for all $j\in[p]$ and some function  $g_0:\R\to\R$,
and if
$\bm{w}$ is random independent of $\bm X$
with exchangeable entries then
$$
\max_{j=1,...,p}
\E\Bigl[
\frac{\hat v^2}{\hat r^2}
\Bigl|
    \hbeta_j 
    -
    \prox\Bigl[\frac{1}{\hat v} g_0\Bigr]
\Bigl(
    \frac{\hat r}{\hat v} Z_j +  \frac{\pm \hat t}{\hat v} w_j
\Bigr) 
\Bigr|^2
\Bigr] \le \frac{\C(\delta,\tau,\kappa)}{\sqrt p}
.
$$
Indeed, by exchangeability the expectation inside the maximum
is the same for all $j=1,...,p$, so that the maximum over $[p]$
is equal to the average over $[p]$.
The previous display is thus a consequence of \Cref{thm:proximal-hbeta}
conditionally on $\bm{w}$, followed by integrating with respect
to the probability measure of $\bm{w}$.

Previous studies on generalized linear models such as \citet{salehi2019impact,loureiro2021learning} discussed in the introduction
derived proximal representations such as 
\eqref{eq:limiting_salehi}-\eqref{informal-salehi} in logistic and single-index models, although the connection between
M-estimator $\bm\hbeta$
and the scalar
$\bar\sigma\bar\tau$ in $\prox[\bar\sigma\bar\tau \tilde f]$
of \eqref{eq:limiting_salehi}-\eqref{informal-salehi}
 has remained unclear.
 Results such as \eqref{eq:penalty-proximal-thm}-\eqref{eq:informal-thm-hbeta-prox} shed light
on this connection, showing that $1/\hat v$ plays the role of the
deterministic $\bar\sigma\bar\tau$ appearing in \eqref{eq:limiting_salehi}-\eqref{informal-salehi}.
In other words, the multiplicative coefficient of $\tilde f$ inside the proximal
operator in \eqref{eq:limiting_result_sur}-\eqref{informal-salehi}
has a simple expression, $1/\hat v$, in terms of the derivatives of
$\bm\hbeta(\bm y,\bm X)$ (cf. \eqref{eq:def-V-df-bound-ideal-case}
and \eqref{hat_v_r_t} for the definition of $\hat v$).
Such connection was previously only established in linear models, see
\cite[Lemma 3.1 and discussion following Proposition 2.7]{el_karoui2018impact}
for unregularized M-estimation,
\cite[Theorem 9]{celentano2020lasso}
for the Lasso
and \cite{bellec2021asymptotic} for penalized robust estimators.
The following subsection studies proximal representations of
the predicted values $\bm x_i^T\bm \hbeta$.

\subsection{Proximal mapping representation for predicted values}

The same techniques as \Cref{thm:confidence-intervals-Sigma}
provide a proximal representation for the predicted value
$\bm{x}_i^T\bm\hbeta$ for a fixed $i\in[n]$.

\begin{restatable}{theorem}{thmProxLoss}
    \label{thm:loss-proximal-representation}
    Let \Cref{assum,assumStrongConvex} be fulfilled.
    Define
    $\aaa_* = 
    \bm{w}^T\bm\Sigma\bm \hbeta$,
    $\sigma_*^2 = \|\bm\Sigma^{1/2} \bm{\hbeta}\|^2 - \aaa_*^2$
    and $\gamma_* = \trace[\bm\Sigma\bm{\hat{A}}]$. Then
    \begin{equation}
    \max_{i=1,...,n}
    \E\Bigl[\frac{1}{\hat r^2}\Bigl|
        \bm{x}_i^\top\bm \hbeta
        -
        \prox\bigl[\gamma_* \ell_{y_i}(\cdot)\bigr]
        \bigl( 
        \aaa_* U_i + \sigma_* Z_i
        \bigr)
    \Bigr|^2\Bigr]
    \le \frac{\C(\delta,\tau)}{ n }
    \label{eq:conclusion-proximal-loss}
    \end{equation}
    where $U_i =\bm{x}_i^T\bm w$ and $Z_i$ are independent $N(0,1)$ random variables.
\end{restatable}
For each $i\in[n]$, $Z_i$ above is obtained by considering a random vector
$\hbbeta^{(i)}$ constructed with the same loss and penalty as $\hbbeta$
constructed from observations $(y_l,\bm x_l)_{l\in[n]\setminus\{i\}}$
and $(y_i, \tilde{\bm{x}}_i)$ where $\tilde{\bm{x}}_i = \bm\Sigma \bm w  U_i + (\bm I_p - \bm\Sigma  \bm w \bm w^T)\bm g_i$ with $\bm g_i\sim N(\bm 0,\bm I_p)$ is independent
of everything else, so that
$\bm x_i$ and $\tilde{\bm{x}}_i$ are iid conditionally on $U_i$.
An alternative construction of $Z_i$ by leaving-out
observation $(y_i,\bm x_i)$; see \cite{bellec2025simultaneous}.

Inequality \eqref{eq:conclusion-proximal-loss} justifies the approximation
$\bm{x}_i^\top\bm \hbeta
\approx
\prox\bigl[\gamma_* \ell_{y_i}(\cdot)\bigr]
\bigl( 
\aaa_* U_i + \sigma_* Z_i
\bigr)$,
or equivalently by definition of the proximal operator,
\begin{align*}
\bm{x}_i^\top\bm \hbeta
+ \gamma_* \ell_{y_i}'(\bm x_i^T\bm\hbeta)
&\approx
\aaa_* U_i + \sigma_* Z_i.
\end{align*}
This provides a clear description of the predicted value $\bm x_i^T\bm\hbeta$,
although $(\aaa_*^2,\sigma_*^2)$ is not observable.
The topic of the next subsection is the estimation of these quantities by
$(\hat\aaa^2,\hat\sigma^2)$.

\subsection{Correlation estimation}

Recall the notation
$\aaa_* = 
\bm{w}^T\bm\Sigma\bm \hbeta$,
$\sigma_*^2 = \|\bm\Sigma^{1/2}\bm{\hbeta}\|^2 - \aaa_*^2$
and $\gamma_* = \trace[\bm\Sigma\bm{\hat{A}}]$
given in \Cref{thm:loss-proximal-representation},
and 
$t_*$ given in \Cref{thm:confidence-intervals-Sigma}.
While the adjustments $(\hat v,\hat r,\hat t)$ for the proximal representation
\eqref{eq:penalty-proximal-thm} are observable from the data,
the quantities $(\aaa_*,\sigma_*)$ in \eqref{eq:conclusion-proximal-loss}
are not. The star subscript in $(\aaa_*,\sigma_*)$ is meant
to emphasize that they are not observable.
Estimation of the quantity $\aaa_*$ is of interest in itself:
An estimate $\hat\aaa^2$ with $\hat\aaa^2\approx \aaa_*^2$ 
would allow the Statistician to estimate 
the correlation ${\aaa_*}\big/{\|\bm\Sigma^{1/2}\bm{\hbeta}\|}$
up to a sign,
or other performance metrics for $\bm\hbeta$.
The following result shows that 
$(\aaa_*^2, \sigma_*^2,\gamma_*, t_*^2)$ can 
can be estimated
by $(\hat \aaa^2,\hat\sigma^2,\hat\gamma, \hat t^2)$.

\begin{restatable}{theorem}{thmAdjustments}
    \label{thm:adjustments-approximation}
    Let \Cref{assum,assumStrongConvex} hold. Let $(\hat r,\hat v,
    \hat\gamma,\hat t^2, \hat \aaa^2,\hat\sigma^2)$
    be as in \eqref{hat_v_r_t},
    and let
    $(t_*,\aaa_*,\sigma_*)$ be as in 
    \Cref{thm:confidence-intervals-Sigma,thm:loss-proximal-representation}.
    Then
    \begin{align}
        \label{eq:consistency-hat-gamma}
        \E\bigl[ \big|\hat v(\hat \gamma - \gamma_*)\big| \bigr]
        &\le \C(\delta,\tau) n^{-1/2},
        \\
        \E\bigl[ \tfrac{1}{\hat r^2} \big|\hat t^2 - t_*^2\big| \bigr]
        &\le \C(\delta,\tau) n^{-1/2},
        \label{eq:consistency-hat-t}
        \\
        \E\bigl[\tfrac{|\hat v|}{\hat r^2}
        \big| \tfrac 1 n \|\bm{X}\bm \hbeta - \hat\gamma \bm \hpsi\|^2  
        - 
        \|\bm\Sigma^{1/2}\bm \hbeta\|^2
        \big|
        \bigr]
        &\le 
        \C(\delta,\tau) n^{-1/2},
        \label{eq:consistency-norm}
        \\
        \E\bigl[
            \tfrac{\hat v^2  \hat t^2  }{\hat r^4}
            (
            \big|
                \hat \aaa^2 -  \aaa_*^2
            \big|
            +
            \big|
                \hat \sigma^2 -  \sigma_*^2
            \big|
            )
        \bigr]
        &\le
        \C(\delta,\tau) n^{-1/2}.
        \label{eq:consistency-aaa}
    \end{align}
    If additionally \Cref{assumExtra} holds, then there exists an event $E$ with $\P(E)\to 1$ such that
    \begin{equation}
        \E\Bigl[ I_{E}
        \Bigl(\Big|\hat \gamma - \gamma_*\Big| 
        +
        \Big|\hat t^2 - t_*^2\Big| 
        +
        \Big| \tfrac 1 n \|\bm{X}\bm \hbeta - \hat\gamma \bm \hpsi\|^2  
        - 
        \|\bm\Sigma^{1/2}\bm \hbeta\|^2
        \Big|
        \Bigr)
        \Bigr]
        \le 
        \C(\delta,\tau,\ell) n^{-1/2}.
        \label{thm_adjustments_additional}
    \end{equation}
\end{restatable}

\Cref{thm:adjustments-approximation} justifies the approximations
$\hat \gamma\approx \gamma_*$,
$\hat t^2 \approx t_*^2$,
$\hat \aaa^2 \approx \aaa_*^2$
and $\hat \sigma^2 \approx \sigma_*^2$,
so that the quantities $(\gamma_*, \aaa_*^2,\sigma_*^2)$
appearing in the proximal representation for $\bm{x}_i^T\bm \hbeta$
in \Cref{thm:loss-proximal-representation} are estimable
by $(\hat\gamma, \hat\aaa^2, \hat\sigma^2)$.
We focus on estimation of $\aaa_*^2$ here instead of $\aaa_*$, because
the sign of $\aaa_*=\bm w^T\bm \Sigma\bm \hbeta$ is
unidentifiable in the single index model
\eqref{single-index} as any sign change can be absorbed into
$F(\cdot)$.
Disambiguating and estimating the sign of $\aaa_*$
    under additional assumptions is discussed in \Cref{sec:sign} below.
The accuracy of $\hat \aaa^2$ for estimation $\aaa_*^2$ is confirmed in simulations (\Cref{sec:simulations}) in
\Cref{table:LS,table:ridge} and \Cref{fig:L1}.

\subsection{Sign disambiguation: estimating $\sign(t_*)$ and $\sign(a_*)$}
\label{sec:sign}

\Cref{thm:confidence-intervals-Sigma} features the unknown sign
$\sign (t_*)$ where $t_*$ is the unknown scalar
defined in \Cref{thm:confidence-intervals-Sigma}.
The estimate $\hat t^2$ defined in \eqref{hat_v_r_t}
consistently estimates $t_*^2$ by \eqref{eq:consistency-hat-t},
but since $\hat t^2$ and $t_*$ are squared \eqref{eq:consistency-hat-t}
is not informative regarding $\sign(t_*)$.
Regarding the correlation with the index $\aaa_* = \bm w^T \bm\Sigma \hbbeta$,
the same can be said about $\hat a^2$ in \eqref{hat_v_r_t}:
by \eqref{eq:consistency-aaa} we have that $\hat a^2$ consistently
estimates $\aaa_*^2$ when $\hat v^2\hat t^2/\hat r^4$ is bounded away from 0,
but since both quantities are squared
\eqref{eq:consistency-aaa} is also non-informative regarding $\sign(\aaa_*)$.

In order to estimate $\sign(t_*)$ or $\sign(\aaa_*)$,
additional assumptions must be made to break symmetry,
otherwise $\bm w$ and $-\bm w$ are
unidentifiable in the single index model
\eqref{single-index} as any sign change can be absorbed into
$F(\cdot)$.
Consider a function $\varrho:\mathcal Y \to\R$ and the vector
$\varrho(\bm y)\in\R^n$ with components $\{\varrho(y_i)\}_{i\in[n]}$
(that is, $\varrho$ acts componentwise).
If $y_i$ is real valued, one may choose
the identity $\varrho(u)=u$.
We break the symmetry
between $\bm w$ and $-\bm w$ by
additionally assuming that $\varrho(\bm y)$ and $\bX\bm w$ are
positively correlated, i.e.,
\begin{equation}
    \E[ \varrho(\bm y)^T \bm X \bm w ] / n 
    = \E[\varrho(y_i) \bm x_i^T \bm w]
    \in (0,+\infty).
    \label{sign_assu}
\end{equation}
Only one of $\bm w$ and $-\bm w$ can satisfy this positive correlation
constraint, so assumption \eqref{sign_assu} disambiguates the sign of $\bm w$.
By the Law of Large Numbers, \eqref{sign_assu} ensure
$\varrho(\bm y)^T \bm X \bm w / n > 0$ with high-probability for
$n$ large enough, so $\sign(\varrho(\bm y)^T \bm X \bm w / n)$ is known.
This makes it possible to estimate $\sign(\aaa_*)$ due to the
approximation
$$
(\varrho(\bm y)^T \bm X \bm w) \aaa_*
\approx
 \varrho(\bm y)^T(\bX \hbbeta - \hat \gamma\bm\hpsi)
$$
formally proved in \Cref{prop:sign} below. The sign of $\aaa_*$ is then estimated
as the sign of $\varrho(\bm y)^T(\bX \hbbeta - \hat \gamma\bm\hpsi)$.
The sign of $t_*$ can be further estimated thanks to the approximation
$$
(\bm{\hpsi}^T \bm X \hbbeta
+ \hat v\|\bm X \hbbeta - \hat\gamma\bm \hpsi \|^2
)/(n\hat r^2)
- \hat \gamma
\approx \aaa_* t_* / \hat r^2,
$$
so that, since $\hat r^2=\|\bm \hpsi\|^2/n$,
we estimate the sign of $t_*$ by the product
\begin{equation}
    \sign\Bigl(\varrho(\bm y)^T(\bX \hbbeta - \hat \gamma\bm\hpsi)\Bigr)
    \cdot
    \sign\Bigl(\bm{\hpsi}^T \bm X \hbbeta
    + \hat v\|\bm{X} \hbbeta - \hat\gamma\bm \hpsi \|^2
    -\|\bm \hpsi\|^2 \hat \gamma\Bigr).
\end{equation}

\begin{proposition}
    \label{prop:sign}
    Let \Cref{assum,assumStrongConvex} be fulfilled. Then with $\bm P = \bm I_p - \bm\Sigma \bm w \bm w^T$,
    \begin{align*}
        \E\Bigl[\Big|
\frac{
 \varrho(\bm y)^T(\bX\bm P^T \hbbeta - \hat \gamma\bm\hpsi)
}{\|\varrho(\bm y)\| \|\bm \hpsi\|}
\Big|\Bigr]
=
        \E\Bigl[\Big|
            \frac{
 \varrho(\bm y)^T(\bX \hbbeta - \hat \gamma\bm\hpsi)
-
(\varrho(\bm y)^T \bm X \bm w) \aaa_*
}{\|\varrho(\bm y)\| \|\bm \hpsi\|}
\Big|\Bigr]
&\le \C(\delta, \tau)n^{-1/2},
        \\
        \E\Bigl[\Big|\hat v\Bigl(
    (\bm{\hpsi}^T \bm X \hbbeta
    + \hat v\|\bm{X} \hbbeta - \hat\gamma\bm \hpsi \|^2
    )/(n\hat r^2)
    - \hat \gamma
    - \aaa_* t_*/\hat r^2
    \Bigr)
\Big|\Bigr] &\le \C(\delta, \tau)n^{-1/2}.
    \end{align*}
\end{proposition}

\section{Main results for unregularized $M$-estimation}
\label{sec:unregularized}
Unregularized $M$-estimation refers to the special case of penalty $g$
in \eqref{hbeta} being identically 0.
In this case, \eqref{hbeta} includes Maximum Likelihood Estimator
if the negative log-likelihood is used for the loss functions
$\ell_{y_i}(\cdot)$ in \eqref{hbeta}.
Our first task is to justify the derivative formula \eqref{eq:derivatives-psi-hbeta} and to determine $\bm{\hat{A}}$ without the strong convexity
assumption made in \Cref{thm:derivatives}.
The following lemma
follows from the implicit function theorem and is proved in
\Cref{sec:proof:unreg}.

\begin{restatable}{lemma}{derivativeUnregularized}
    \label{lemma:derivative_unregularized}
    Let \Cref{assumLowDim,assumExtra} be fulfilled so that the penalty is $g=0$.
    Let $\bm y\in\mathcal Y^n$ and $\bar{\bm X}\in\R^{n\times p}$ be fixed.
    If a minimizer $\bm \hbeta$ exists at $(\bm y,\bar{\bm X})$
    and $\bar{\bm X}^T\bar{\bm X}$ is invertible,
    then there exists a neighborhood of $\bar{\bm X}$ such that
    $\bm\hbeta(\bm y,\bm X)$ exists in this neighborhood,
    the map $\bm X\mapsto \bm\hbeta(\bm y,\bm X)$ restricted to this
    neighborhood is continuously differentiable,
    and \eqref{eq:derivatives-psi-hbeta}
    holds with $\bm{\hat{A}} = \Aunreg$.
\end{restatable}

For unregularized $M$-estimation with $p<n$, the optimality conditions
of the optimization problem \eqref{hbeta} read $\bm{X}^T \bm \hpsi = \bm 0$.
With $\bm{X}^T \bm \hpsi = \bm 0$ and the explicit expression
$\bm{\hat{A}} = \Aunreg$ from the previous lemma,
the adjustments $(\df,\hat \gamma,\hat t^2,\hat \aaa^2,\hat\sigma^2)$ in \eqref{hat_v_r_t} reduce to the simpler forms
\begin{align}
    \df = p, 
    \quad
    \hat \gamma =  \frac{p/n}{\hat v},
    \quad
    \hat \aaa^2  = \frac{\|\bm{X} \bm \hbeta\|^2}{n}
    - \frac pn\bigl(1-\frac pn\bigr)\frac{ \hat r^2}{\hat v^2}
    ,
    \quad
    \hat t^2 = 
    \hat\aaa^2 \hat v^2,
    \quad
    \hat\sigma^2 
    = 
        \frac{p}{n}
    \Bigl(
        \frac{\hat r}{\hat v}
    \Bigr)^2.
    \label{adjustments_unregularized}
\end{align}
Here, the fact that $\df=p$ justifies the notation $\df$ for the quantity defined
in \Cref{thm:derivatives}: In unregularized $M$-estimation
$\df$ is the number of parameters, or degrees of freedom of the estimator.
A similar justification for the notation $\df$ is observed for L1
regularized $M$-estimation as discussed in \Cref{sec:L1} below.

The low-dimensional systems of equations (e.g., \eqref{karoui_system}, \eqref{sur_system}) characterizing the behavior of
unregularized $M$-estimation were obtained
in linear models \cite{el_karoui2013robust,donoho2016high,el_karoui2018impact},
in logistic regression models \citep{sur2018modern,zhao2020asymptotic},
other generalized linear models \cite{sawaya2023statistical},
single index models \cite{loureiro2021capturing} and
Gaussian mixture models \cite{mai2019large,loureiro2021learning}.
These results state that if the low-dimensional system corresponding
to the loss and the generative model for $y_i$ given $x_i$ admits
a unique solution, then this solution characterizes the limit in probability
of $\|\bm\Sigma^{1/2}\bm\hbeta\|^2$ as well as the limit in probability
of the correlation
$\aaa_*=\bm w^T\bm\Sigma\bm\hbeta$ with the index $\bm w$, and
limits of averages with respect to test functions $\phi$ as in
\eqref{eq:limiting_result_sur}. In logistic regression,
the
existence and unicity of solutions to 
the low-dimensional system \eqref{sur_system} is provably related to
the existence of the minimizer $\bm\hbeta$ and the linear 
separability of the data \cite{candes2020phase,sur2018modern}.
In the global-null setting of logistic regression, which
corresponds to the linear model as in \cite{el_karoui2013robust,donoho2016high},
existence of solutions to the low-dimensional system is proved
in the supplement of \cite{sur2019likelihood}.
If $\|\bm\Sigma^{1/2}\bm\hbeta\|^2$ and $\aaa_*$ admits non-zero limits
in probability as in these references, then $\P(\frac{1}{K}\le |\aaa_*|, \frac{1}{n}\|\bm X\bm\hbeta\|^2\le K)$ for some constant $K>$0.
The aforementioned results thus justify the following assumption that will be
used in the theorems of this section.

\begin{assumption}
    \label{assum:bounded_exists}
    Let $K>0$ be constant.
    Assume that \Cref{assum,assumLowDim,assumExtra} hold so that the penalty is $g=0$.
    Assume that with probability approaching one
    as $n,p\to\infty$, the estimator $\bm{\hbeta}$ exists,
    is bounded and has non-vanishing correlation with the index $\bm w$ 
    in the sense that
    $\P(\text{minimizer }\bm \hbeta \text{ in \eqref{hbeta} exists and }
    \frac 1 n \|\bm X\bm\hbeta\|^2 \le K
    \text{ and }\frac1K \le |\aaa_*|)\to1$.
\end{assumption}

The next result involves the square root of $\Omega_{jj}-w_j^2$.
This quantity is always non-negative since it equals
$\bm e_j^T\bm\Sigma^{-1}\bm e_j - \bm e_j^T\bm w \bm w^T \bm e_j
=\bm e_j^T\bm\Sigma^{-1/2}(\bm I_p - \bm v\bm v^T)\bm\Sigma^{-1/2}\bm e_j$
for $\bm v = \bm\Sigma^{1/2}\bm w$
and $\bm I_p-\bm v\bm v^T$ is positive semi-definite thanks to $\|\bm v\|=1$.

\begin{restatable}{theorem}{ThmUnregularized}
    \label{thm:unregularized}
    Let \Cref{assum,assumLowDim,assumExtra} be fulfilled so that the penalty is $g=0$,
    and let $\Omega_{jj}>0$ be defined in \eqref{Omega_jj}.
    Let $K>0$ and define the event
    $E = \{\text{the minimizer }\bm \hbeta \text{ in \eqref{hbeta} exists and }
    \frac 1 n \|\bm X\bm\hbeta\|^2 \le K) \}.$
    For each $j=1,...,p$
    such that $w_j^2\ne \Omega_{jj}$,
    there exists a standard normal $Z_j\sim N(0,1)$ satisfying
    \begin{equation}
    \E\Bigl[
    I_E
    \Bigl|
    \Bigl(\frac{n}{\Omega_{jj} - w_j^2}\Bigr)^{1/2}
    \Bigl(
    \hbeta_j
    -
    \aaa_*
    w_j
    \Bigr)
    -  Z_j
       \frac{\hat r}{\hat v}
    \Bigr|^2
    \Bigr]
    \le
    \C(\delta,K,\ell)
    \frac{1}{p}
    \label{eq:thm-Z_j-unregularized}
    \end{equation}
    for $\aaa_*=\bm\hbeta^T\bm \Sigma \bm w$ and
    the observables $(\hat \gamma, \hat v, \hat r ,\hat \aaa^2)$ in
    \eqref{adjustments_unregularized}.
    For $\sigma_*^2 = \|\bm\Sigma^{1/2}\bm\hbeta\|^2 - \aaa_*^2$, it holds
    \begin{equation}
    \E\bigl[
    I_E
    \bigl(
    \big|\aaa_*^2 - \hat \aaa^2\big|
    +
    \big|\sigma_*^2 - \hat \sigma^2\big|
    \bigr)
    \bigr]
    \le \C(\delta,K,\ell) / \sqrt p 
    \label{eq:bound-correlation-unreg}
    \end{equation}
    and $\E[I_E\max\{\hat r, \frac{1}{\hat r}, \hat v, \frac{1}{\hat v} \}]\le \C(K,\delta,\ell)$.
    If additionally \Cref{assum:bounded_exists} holds,
    then $\big|\hat \aaa - |\aaa_*|\big|= O_\P(n^{-1/2})$
    for $\hat\aaa = \max(0,\hat a^2)^{1/2}$
    where $O_\P(\cdot)$ hides constants depending only on $(\delta,K,\ell)$.
\end{restatable}
\Cref{thm:unregularized} is proved in
\Cref{sec:proof_thm_unregularized}.
In logistic regression \cite{candes2020phase,sur2018modern} for instance,
the event that $\bm\hbeta$ does not exist has positive probability
for any $n,p$, although this probability may be exponentially small.
This makes the indicator function $I_E$
(of an event $E$ on which $\hbbeta$ exists) unavoidable inside the expectations
in \eqref{eq:thm-Z_j-unregularized}-\eqref{eq:bound-correlation-unreg},
since we cannot make any statement about $\bm\hbeta$ in the event that
it does not exist.

When $\P(E)\to 1$ as in \Cref{assum:bounded_exists} and $\Omega_{jj}>w_j^2$,
\eqref{eq:thm-Z_j-unregularized} implies the
approximation
\begin{equation}
    \textstyle
    \sqrt{\frac{n}{\Omega_{jj}-w_j^2}}
    \bigl(\hbeta_j - \aaa_* w_j\bigr) 
    - Z_j \frac{\hat r}{\hat v}
    =
    \sqrt{\frac{n}{\Omega_{jj}-w_j^2}}
    \bigl(\hbeta_j - \aaa_* w_j\bigr) 
    - Z_j\hat \sigma \sqrt{\frac np}
    = O_\P(\frac 1p)  
\label{eq:asymp-normality-O_p1}
\end{equation}
with $Z_j\sim N(0,1)$ having pivotal distribution,
and the remainder $O_\P(1/p)$
converges to 0 in probability.
In a logistic model where $\bm\hbeta$ is the logistic MLE,
the stochastic representation result
in \cite[Lemma 2.1]{zhao2020asymptotic} yields
the normal approximation
$\sqrt{p}(\Omega_{jj}-w_j^2)^{-1/2}(\hbeta_j-\aaa_*w_j)/\sigma_*\to^d N(0,1)$
which is comparable to the previous display with $\hat \sigma$ replaced
by $\sigma_*$; the advantage of the asymptotic normality involving
$\hat\sigma$ is that the asymptotic variance is already estimated.

It remains to explain how the unknown $\aaa_*$ in \eqref{eq:asymp-normality-O_p1}
can be estimated.
One important consequence of \Cref{thm:unregularized} resides in the rate of convergence for $|\aaa_*^2-\hat\aaa^2|$ in \eqref{eq:bound-correlation-unreg}
and for $\big|\hat a - |\aaa_*|\big|$ under \Cref{assum:bounded_exists}.
From \eqref{eq:asymp-normality-O_p1}, in order to relate $\hbeta_j$
to the unknown parameter $w_j$ of interest through data-driven
quantities, one needs to replace the unknown correlation $\aaa_*=\bm\hbeta^T\bm\Sigma \bm w$ by some known scalar $a$, leading to the approximation
\begin{equation}
    \textstyle
    \sqrt{\frac{n}{\Omega_{jj}-w_j^2}}
    \bigl(\hbeta_j - a w_j\bigr) 
    - Z_j \frac{\hat r}{\hat v}
    = O_\P(\frac 1p) +  
    \frac{|w_j|}{\sqrt{\Omega_{jj}-w_j^2}}
    \sqrt n|a-a_*|
\label{eq:asymp-normality-O_p3}
\end{equation}
In order to obtain an asymptotic normality result from the
previous display,
the scalar $a$ needs to estimate $\aaa_*$ fast enough to ensure
that the right-hand side converges 
to 0 in probability;
similar observations were already made in \cite[Section 3.2.2]{zhao2020asymptotic}.
\citet{zhao2020asymptotic} further conjectured that
$|\bar a-\aaa_*|=O_\P(n^{-1/2})$ holds if $\bar a$ is the deterministic 
limit in probability
of $\aaa_*=\bm\hbeta^T\bm\Sigma \bm w$
(that is, $\aaa_*\to^\P \bar a$),
obtained by solving the nonlinear system of three unknowns \eqref{sur_system}
studied in \cite{sur2018modern}
under the assumptions in \cite{sur2018modern,zhao2020asymptotic}.
Instead of considering deterministic scalars $a$ in \eqref{eq:asymp-normality-O_p3}, our proposal is to use $a=\pm\hat \aaa$ for $\hat \aaa=\max(0,\hat a^2)^{1/2}$
and $\pm=\sign(\aaa_*)$, where $\hat a^2$ is 
the observable adjustment defined in
\eqref{adjustments_unregularized}.
Disambiguating the sign $\pm = \sign(a_*)$
under additional assumptions is discussed in \Cref{sec:sign},
and the estimate proposed there for $\sign(a_*)$ is still valid
in the unregularized case.
The last part of \Cref{thm:unregularized} shows that 
$\hat a$ is a $n^{-1/2}$-consistent estimate of $|\aaa_*|$.
On the other hand, as discussed in \cite[3.2.2]{zhao2020asymptotic} and
\cite[Remark 1]{yadlowsky2021sloe}, obtaining $O_\P(n^{-1/2})$ error bounds
on $|\bar a - \aaa_*|$ appear out of reach of current techniques.

Developing $n^{-1/2}$-consistent estimates of $|\aaa_*|$ is important because
\eqref{eq:asymp-normality-O_p1}
and \eqref{eq:asymp-normality-O_p3} with $a=\hat a$ imply
\begin{equation}
    \textstyle
    \sqrt{\tfrac{n}{\Omega_{jj}-w_j^2}}
    \bigl(\hbeta_j - \pm \hat\aaa w_j\bigr) 
    - Z_j \frac{\hat r}{\hat v} 
    =
    o_\P(1) + 
    O_\P(1)
    \tfrac{|w_j|}{\sqrt{\Omega_{jj}-w_j^2}}
    \label{eq:asymp-normality-O_p2}
\end{equation}
if
$\big|\hat\aaa - |a_*| \big| = O_\P(n^{-1/2})$
as in the discussion of the previous paragraph.
Then the right-hand side converges to 0 in probability for any
covariate $j\in[p]$ such that $w_j^2/\Omega_{jj} \to 0$.
Since $\max\{\hat r,\frac{1}{\hat r},\hat v, \frac{1}{\hat v} \} = O_\P(1)$
by \Cref{thm:unregularized}, $w_j^2/\Omega_{jj} \to 0$ implies 
the normal approximation
\begin{equation}
    \label{eq:unreg_normal-approx}
    \textstyle
    \sqrt{\tfrac{n}{\Omega_{jj}-w_j^2}}
    \bigl(\hat v / \hat r\bigr)
    \bigl(\hbeta_j - \pm \hat\aaa w_j\bigr) 
    \to^d N(0,1)
    \qquad
    \text{ where }
    \pm = \sign(\aaa_*), ~\hat \aaa = \max(0,\hat a^2),
\end{equation}
which relates $\hbeta_j$ and the unknown parameter of interest
$w_j$ to the pivotal
standard normal distribution
via the observable adjustments $\hat v,\hat r, \hat\aaa$.
We summarize this discussion in the next corollary.

\begin{corollary}
    \label{cor:unregularized}
    Let \Cref{assum:bounded_exists}
    be fulfilled.
    For any sequence $j=j_n$,
    if
    \begin{equation}
        \label{eq:assum_w_j}
        w_j^2/\Omega_{jj}\to 0 
        \quad
        \text{ as }
        \quad
        n,p\to+\infty,
    \end{equation}
    then the normal approximation \eqref{eq:unreg_normal-approx} holds.
\end{corollary}
\begin{proof}[Proof of \Cref{cor:unregularized}]
Since  $\P(E)\to 1$ for the event $E$ in \Cref{thm:unregularized}
under \Cref{assum:bounded_exists},
the bounds $\big|\hat \aaa - |\aaa_*|\big|=\big|\hat \aaa - \pm\aaa_*\big| = O_\P(n^{-1/2})$ and $\max\{\frac{\hat r}{\hat v}, \frac{\hat v}{\hat r}\}= O_\P(1)$ hold  by the last part of \Cref{thm:unregularized}.
Consequently the right-hand side in
\eqref{eq:asymp-normality-O_p2} converges to 0 in probability 
if \eqref{eq:assum_w_j} holds.
\end{proof}

If the condition number $\kappa$ of $\bm\Sigma$ is bounded from above by $\kappa$ we have $
1/\kappa \le \Omega_{jj}/\|\bm w\|^2 \le \kappa$
thanks to $\|\bm \Sigma^{1/2}\bm w\|=1$, so that
$$
\frac1{\kappa}\frac{w_j^2}{\|\bm w\|^2}
\le
\frac{w_j^2}{\Omega_{jj}} 
\le \kappa \frac{w_j^2}{\|\bm w\|^2}.
$$
As long as the condition number $\kappa$ is a constant independent of $n,p$,
the quantity $w_j^2/\Omega_{jj}\to 0$ appearing in the right-hand side
of 
\eqref{eq:thm-Z_j-unregularized},
\eqref{eq:asymp-normality-O_p1} and 
of \eqref{eq:asymp-normality-O_p2} is of the same order as
$w_j^2/\|\bm w\|^2$, and the condition $w_j^2/\Omega_{jj}\to 0$
is equivalent to $w_j^2$ being
negligible compared to the full squared norm $\|\bm w\|^2$ of the index.
In other words, in typical cases where the unregularized M-estimator
$\bm\hbeta$ exists, is bounded and has non-vanishing correlation $\aaa_*$
with high-probability as in
\Cref{assum:bounded_exists},
if the condition number of $\bm\Sigma$ is bounded 
and $w_j^2/\|\bm w\|^2\to 0$ then
the normal approximation
\eqref{eq:unreg_normal-approx} holds.

In a well-specified binary logistic regression model where
$\bm\hbeta$ is the MLE and $\bm\beta^*$ is the true logistic
regression vector
normalized such that $\E[(\bm x_i^T\bm\beta^*)^2]=(\bm\beta^*)^T\bm\Sigma\bm\beta^*$ is a fixed constant $\gamma^2$,
\cite[Theorem 3.1]{zhao2020asymptotic}
established asymptotic normality of $\hbeta_j$ under the assumption
that $\sqrt n \tau_j \beta_j^* = O(1)$ where $\tau_j^2=(\bm e_j^T\bm\Sigma^{-1}\bm e_j)^{-1}$ in the notation of \cite{zhao2020asymptotic}.
\Cref{cor:unregularized} obtains asymptotic normality results
under the relaxed assumption \eqref{eq:assum_w_j}, which is equivalent
to $\tau_j\beta_j^* \to 0$ as $n,p\to+\infty$
since $\tau_j\beta_j^* = \gamma w_j/\sqrt{\Omega_{jj}}$.
In other words, the amplitude of $\tau_j\beta_j^*$ for which asymptotic
normality provably holds in \Cref{cor:unregularized} is almost $\sqrt n$
greater than allowed in previous studies.

\subsection{Non-separable loss function}
We use this section to explain 
the presence of the indicator function
$I_E$ in \Cref{thm:unregularized}
and to explain 
the argument behind its proof.
Throughout this subsection, assume $p<n$ and $\bm X^T\bm X$ positive definite,
which holds with probability one under \Cref{assum,assumLowDim}.
The situation where $\bm\hbeta$ does not exist
stems from the lack of coercivity of the loss function:
in the minimization problem
$\min_{\bm b\in\R^p}\sum_{i=1}^n \ell_{y_i}(\bm x_i^T\bm b)$,
one can find a sequence $(\bm b^{(t)})_{t\ge 1}$ with $\|\bm X\bm b^{(t)}\|\to+\infty$ as $t\to+\infty$ such that
$\ell_{y_i}(\bm x_i^T\bm b^{(t)})\to 
\inf_{\bm b\in\R^p}\sum_{i=1}^n \ell_{y_i}(\bm x_i^T\bm b)$,
although the infimum is not attained at any $\bm b\in\R^p$.
This is for instance 
the case in logistic regression when the data is separable
(see, e.g., \cite{candes2020phase} and the references therein).

This situation can be avoided by modifying the loss function to make it
coercive. We prove \Cref{thm:unregularized} for loss functions 
satisfying \Cref{assumExtra} by introducing the modified optimization
problem
\begin{equation}
    \label{modified}
\min_{\bm b\in\R^p} \sum_{i=1}^n\ell_{y_i}(\bm x_i^T\bm b)
+ n H\Bigl(\frac{1}{2}(\frac1n\|\bm X\bm b\|^2 - K)\Bigr)
\end{equation}
where $K>0$ is a fixed constant and $H:\R\to\R$ is a convex smooth 
function with derivative $h=H'$ satisfying $h(t)=0$ for $t<0$, $h(t)=1$ for $t>1$. An example of such function $H$ is given in \Cref{thm:unregularized_with_H} below. This modified optimization maintains desirable properties 
of $\bm\hbeta$ if $\bm\hbeta$ exists and satisfies $\frac1n\|\bm X\bm\hbeta\|^2\le K$, since in this case $\bm\hbeta$ is also a minimizer of \eqref{modified}.
In other words, in the event $E$ of \Cref{thm:unregularized},
$\bm\hbeta$ is also a minimizer of \eqref{modified} so we may as well
study the optimization problem \eqref{modified} to bound from above
the expectations in \eqref{eq:thm-Z_j-unregularized}-\eqref{eq:bound-correlation-unreg}.
The second term in \eqref{modified} could also be useful
in practice if
$\inf_{\bm b\in\R^p}\sum_{i=1}^n \ell_{y_i}(\bm x_i^T\bm b)$ is not attained,
as the coercivity of the second term in \eqref{modified} ensures
that a minimizer always exists.
The modified objective function in \eqref{modified} is not a separable
function of $(\bm x_i^T\bm b)_{i\in[n]}$. The next result
provides general bounds applicable to such non-separable loss function.

\begin{restatable}{theorem}{ThmNonSeparable}
    \label{thm:nonseparable-mathcalL}
    Assume that $1 < \delta \le n/p \le 2\delta$, that $\bm X$ has iid $N(\bm 0,\bm \Sigma)$ rows
    and $\bm w\in\R^p$ satisfies $\bm w^T\bm \Sigma \bm w = 1$.
    Let $U$ be a latent random variable independent of $\bm X$
    and assume that $\mathcal L:\R^n\to\R$ is a random loss function
    of the form
    $\mathcal L(\bm v)= F(\bm v ,U,\bm X \bm w)$ for all $\bm v\in\R^n$ 
    for some deterministic measurable function $F$.
    Assume that with probability one with respect to $(U,\bm X\bm w)$,
    $\mathcal L$ is
    convex, coercive, twice
    differentiable with positive definite Hessian everywhere.
    Let $\bm\hbeta = \argmin_{\bm b\in\R^p} \mathcal L(\bm X \bm b)$
    and assume that $\P(\|\nabla \mathcal L(\bm X\bm\hbeta)\|\le \sqrt n)=1$.
    Extend the notation $\bm D,\hbA,\bm V$ to this non-separable
    setting by
    $$
    \bm \hpsi(\bm y,\bm X) = - \nabla \mathcal L(\bm X\bm\hbeta)\in\R^n,
    \quad
    \bm D = \nabla^2 \mathcal L(\bm X \bm \hbeta)\in\R^{n\times n},
    \quad
    \hbA = (\bm X^T\bm D \bm X)^{-1}\in\R^{p\times p}
    $$
    and $\bm V = \bm D - \bm D \bm X\hbA\bm X^T\bm D$.
    Define $\hat r^2=\|\bm\hpsi\|^2/n$ as well as $\hat v=\trace[\bm V]/n$.
    Let $(\hat a^2,\hat\sigma^2)$ be as in \eqref{adjustments_unregularized}
    and define $\aaa_*=\bm\hbeta^T\bm\Sigma \bm w$ as well as
    $\sigma_*^2=\|\bm\Sigma^{1/2}\bm\hbeta\|^2-\aaa_*^2$.
    Then
    \begin{multline}
    \E[\hat v^2|\hat\sigma^2-\sigma_*^2|]
    +
    \E[\hat v^2|\hat\aaa^2-\aaa_*^2|]
    +
    \E[
    (\sqrt n \hat v \bm\hbeta^T\bm u - \hat r Z )^2
    ]^{1/2}
    \\\le 
    \tfrac{\C(\delta)}{\sqrt n}
    \E\bigl[
        \bigl(
        1
        \vee
        \| n \bm\Sigma^{1/2}\hbA\bm\Sigma^{1/2}\|_{op}
        \vee 
        \|\tfrac{1}{\sqrt n}\bm D\|_F
        \vee
        \|\bm\Sigma^{1/2}\bm\hbeta\|
        \vee
        \|n^{-1/2}\bm X\bm\Sigma^{-1/2}\|_{op}
        \bigr)^8
    \bigr]
        \label{eq:conclusion_nonseparable}
    \end{multline}
    for some $Z\sim N(0,1)$
    for any deterministic $\bm u\in\R^p$ with $\|\bm\Sigma^{-1/2}\bm u\|=1$
    such that $\bm w^T\bm u =0$.
\end{restatable}

\Cref{thm:nonseparable-mathcalL} is proved in
\Cref{sec:thm:nonseparable-mathcalL}.
To obtain a desired $n^{-1/2}$ upper bound on the first line of \eqref{eq:conclusion_nonseparable}, we only need to show that the expectation
on the right-hand side of \eqref{eq:conclusion_nonseparable} is bounded
from above by a constant, which is done for the loss \eqref{modified}
in the next result.

\begin{restatable}{theorem}{ThmUnregularizedExtra}
    \label{thm:unregularized_with_H}
    Let \Cref{assum,assumLowDim,assumExtra} be fulfilled so that the penalty is $g=0$.
    Let $K>0$. Define $h(t)= 0$ for $t<0$, $h(t)=1$ for $t>1$ and $h(t)=3t^2-2t^3$ for $t\in[0,1]$, set $H(t)=\int_0^t h(u)du$ and note that $H$ is convex, twice continuously differentiable and nondecreasing. 
    For any $\bm v\in\R^n$, define
    \begin{equation}
    \mathcal L(\bm v)=
    \frac{1}{1+\sqrt{K+2}}
    \Bigl[
        \sum_{i=1}^n
    \ell_{y_i}(v_i)
    + n H\Bigl(\tfrac12(\tfrac{1}{n} \|\bm v\|^2 - K)\Bigr)
    \Bigr].
    \end{equation}
    Then the minimizer $\bm{\hat{b}}=\argmin_{\bm b\in\R^p}\mathcal L(\bm X\bm b)$ satisfies $\mathbb P(\|\nabla \mathcal L(\bm X \bm{\hat b})\|\le \sqrt n)=1$, the assumptions of \Cref{thm:nonseparable-mathcalL} are satisfied
    and the right-hand side of \eqref{eq:conclusion_nonseparable}
    is bounded from above by $\C(\delta,K,\ell) n^{-1/2}$.
    In the event $E$ of \Cref{thm:unregularized},
    the minimizer
    $\bm{\hbeta}=\argmin_{\bm b\in\R^p}\sum_{i=1}^n\ell_{y_i}(\bm x_i^T\bm b)$
    is equal to $\bm{\hat b}$
    and $\max\{\hat r, \frac{1}{\hat r}, \hat v, \frac{1}{\hat v} \}\le \C(K,\delta,\ell)$.
\end{restatable}
\Cref{thm:unregularized_with_H} is proved in 
\Cref{sec:proof:thm:unregularized_with_H}.
\Cref{thm:unregularized} is finally obtained as a consequence
of \Cref{thm:nonseparable-mathcalL,thm:unregularized_with_H} thanks to
$\E[I_E|\cdot|]\le \E[|\cdot|]$ and using that $\bm\hbeta$
from \Cref{thm:unregularized} equals $\bm{\hat{b}}$ from \Cref{thm:unregularized_with_H} in $E$, where $E$ is the event in \Cref{thm:unregularized}.

\section{Examples and simulations}
\label{sec:simulations}

\subsection{Linear models: Square loss and Huber loss}

It is first instructive to specialize \Cref{thm:confidence-intervals-Sigma}
and the adjustments \eqref{hat_v_r_t}
to loss functions usually used in linear models.
For the square loss $\ell_{y_i}(u_i)=\frac 1 2 (y_i - u_i)^2$
and convex penalty $g$ in \eqref{hbeta} we have
$\ell_{y_0}''(u)=1$ for all $y_0\in\R$ and 
\begin{equation}
    \label{v-gamma-r-squareloss}
    \hat v = 1- \frac{\df}{n},
    \qquad
    \hat \gamma = \frac{\df}{n - \df} =\frac{1}{\hat v} - 1,
    \qquad
    \hat r^2 
    = \frac{\|\bm y - \bm X\bm \hbeta\|^2}{n}
\end{equation}
where $\df$ is defined in \eqref{eq:def-V-df-bound-ideal-case},
and the quantity $\df$
equals $\trace[\frac{\partial}{\partial \bm y} \bm X \bm\hbeta(\bm y,\bm X)]$
thanks to \cite[Theorem 2.1]{bellec2021derivatives} which relates
the matrix ${\smash{\bm{\hat{A}}}}$ of \Cref{thm:derivatives} to the derivatives
with respect to $\bm y$.
Thus, $\df$ is the usual notion of degrees-of-freedom of the estimator
$\bm\hbeta$ in linear models as introduced in \citet{stein1981estimation},
and $\hat v = 1 - \df/n$ captures the difference
between sample size and degrees-of-freedom.
On the other hand, $n \hat r^2$ is the usual residual sum of squares.

If the Huber loss 
$H(u)=\int_0^{|u|}\min(1,v)dv$ is used and 
$\bm \hbeta$ is the estimate \eqref{hbeta} with
$\ell_{y_i}(\bm x_i^T\bm b) = H(y_i - \bm x_i^Tb)$ and convex penalty $g$, 
then
\begin{equation}
      \hat v =\frac{\hat n - \df}{n},
      \qquad
      \hat \gamma = \frac{\df}{\hat n - \df},
      \qquad
      \hat r^2 = \frac{\|H'(\bm y - \bm X\bm \hbeta)\|^2}{n}
\end{equation}
where $\hat I \defas \{i=1,...,n: |y_i-\bm x_i^T\bm\hbeta|\le 1\}$ and
$\hat n = |\hat I|$ denotes the number of residuals that fall within the interval
$[-1,1]$ where the loss $H(\cdot)$ is quadratic.
The estimate $\hat r^2$ is the average of the squared
residuals clipped to $[-1,1]$, since here the derivative $H'$ of the Huber
loss is $H'(u)=\max(-1, \min (1, u)))$.
The integer $\hat n=|\hat I|$ represents the effective sample size,
since observations $i\notin \hat I$ do not participate in the fit
of \eqref{hbeta} in the sense that $\frac{\partial}{\partial y_i}\bm\hbeta(\bm y,\bm X)=0$ for $i\notin \hat I$ \cite{bellec2021derivatives}.
Similarly to the square loss case, $n\hat v=\hat n - \df$ captures
the effective sample size left after subtracting the degrees-of-freedom of the estimator $\bm\hbeta$.

For the square loss,
both $(n - \df)$ and $\|\bm y - \bm X \bm \hbeta\|^2$
are expected to appear in the confidence interval about $\beta_j^*$
for regularized least-squares \cite{bellec_zhang2019second_poincare},
while $(\hat n - \df)$ and $\|H'(\bm y - \bm X \bm \hbeta)\|^2$
are expected to appear in confidence intervals about $\beta_j^*$
for the Huber loss \cite{bellec2021asymptotic}.
In \Cref{thm:confidence-intervals-Sigma} on the other hand,
the confidence interval is about the component $w_j$ of the normalized
index $\bm w$. This is where $\hat t$
enters the picture:
the role of $\hat t$ is to bring the index $\bm w$ (which is normalized with $\|\bm \Sigma^{1/2}\bm w\|=1$) on the same scale as $\hat v \bm\hbeta$.
The following proposition makes this precise.
\begin{restatable}{proposition}{PropHatTLinearModel}
    \label{prop:hat-t-linear-model}
    Let \Cref{assum,assumStrongConvex} be fulfilled and additionally assume a linear model
    where observations $y_i = \bm x_i^T\bm\beta^* + \eps_i$ are iid with
    additive noise
    $\eps_i$ independent of $\bm x_i$.
    Set $\bm w = \bm \beta^* / \|\bm \Sigma^{1/2}\bm \beta^*\|$,
    and assume that $\|\bm \Sigma^{1/2}\bm \beta^* \|$ equals a
    constant independent of $n,p$.
    Then
    $$\hat t = \hat v\|\bm \Sigma^{1/2}\bm \beta^* \| 
     + n^{-1/2}O_\P(1)\bigl[
            \hat r  + \|\bm \Sigma^{1/2}\bm \beta^*\|
     \bigr].
    $$
    Above, $O_\P(1)$ denotes a random variable $W$ such that
    for any $\eta>0$ there exists a constant $K$ depending on $(\eta,\delta,\tau,\|\bm \Sigma^{1/2}\bm\beta^*\|)$ only such that
    $\P(|W|>K)\le \eta$.
\end{restatable}

\Cref{prop:hat-t-linear-model} justifies the approximation
$\hat t \approx \hat v \|\bm \Sigma^{1/2}\bm\beta^*\|$, and combined
with \eqref{eq:informal-approximation-theorem31},
$
\sqrt n \hat v (\debias_j - \beta_j^*)
\approx
\hat r ~ \Omega_{jj}^{1/2}Z_j
$
with $Z_j\sim N(0,1)$.
This recovers the asymptotic normality result 
proved for regularized least-squares in \cite{bellec_zhang2019second_poincare}
and for some robust loss functions in \cite{bellec2021asymptotic}.
\Cref{prop:hat-t-linear-model} illustrates that $\hat t$ brings the normalized index $\bm w$
on the same scale as $\hat v \bm\beta^*$ and $\hat v \bm\debias$
in this linear model setting.

We note in passing that \Cref{prop:hat-t-linear-model}
justifies the use of $\hat t /\hat v$ to estimate the signal strength $\|\bm \Sigma^{1/2}\bm\beta^*\|$, when an initial M-estimator $\bm\hbeta$ is provided
to estimate the high-dimensional parameter $\bm\beta^*$.
For $\bm\hbeta = 0$ (which can be seen as a special case of \eqref{hbeta} with penalty satisfying $g(\bm 0)=0$ and $g(\bm b) = +\infty$ for $\bm b\ne \bm 0$), the quantity $\hat t^2/\hat v^2$ reduces to the
estimator of the signal strength in \cite{dicker2014variance}.
\Cref{table:signal-strength} reports experiments demonstrating
the accuracy of \Cref{prop:hat-t-linear-model}.
The approximate normality of $\hat t/\hat v - \|\bm\Sigma^{1/2}\bm\beta^*\|$ 
observed in the QQ-plot of \Cref{table:signal-strength} is not currently proved theoretically---a proof of this observed normality
remains an open problem.

\begin{table}[ht]
    \small
\begin{tabular}{@{}|l|l|l|l|l|@{}}
\toprule
Estimate & Signal strength \\ \midrule
$\hat t / \hat v$ & $\|\bm \Sigma^{1/2}\bm\beta^*\|$  \\ \midrule
$2.071 \pm 0.122$ (average $\pm$ std) & 2.000  \\ \midrule
\includegraphics[width=3in]{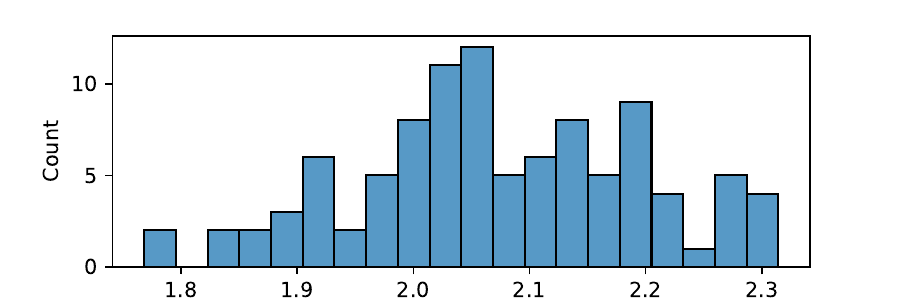} & \\ \midrule
\includegraphics[width=3in]{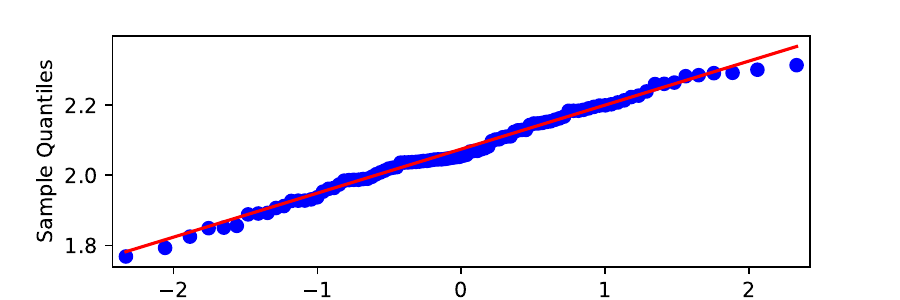} & \\ \midrule
\end{tabular}
\caption{
    \label{table:signal-strength}
    \small
    Estimate $\hat v / \hat v$ of the signal strength $\|\bm \Sigma^{1/2}\bm \beta^*\|$
    in a linear model $y_i+\bm x_i^T\bm\beta^* +\eps_i$ with
    standard Cauchy noise $\eps_i$ independent of $\bm x_i\sim N(\bm 0,\bm \Sigma)$ with $\bm \Sigma = \bm I_p$,
    with $n=1500, p=1501$ and $\|\bm\beta^*\|_0=200$ with all non-zero
    coordinates equal to the same value. The M-estimator
    $\bm \hbeta$ is chosen with
    $\ell_{y_i}(u) = H(u-y_i)$ for the Huber loss $H(u)=\int_0^{|u|}\min(1,v)dv$ and Elastic-Net penalty $g(\bm b) = n^{-1/2} \|\bm b\|_1 + 0.05\|\bm b\|_2^2$.
    Average, standard deviation, histogram and QQ-plot of $\hat t/\hat v$ were computed
    over 100 independent realizations of the dataset.
    Over the 100 repetitions, $\bm\hbeta$
    has False Negatives $47.4\pm7.3$, False Positives $271.36\pm13.4$ and
    True Positives $152.6\pm7.3$.
}
\end{table}

In summary, for the square loss and Huber loss,
\begin{itemize}
    \item$\hat r^2$ is a generalization of the residual sum of squares,
    \item $\hat v$ is the difference of an effective sample size minus the degrees-of-freedom of $\bm \hbeta$,
and
\item
    $\hat t/\hat v$ brings the normalized index $\bm w$
on the same scale as $\bm\hbeta$.
\end{itemize}

\subsection{Least-squares with nonlinear response}
Let $p/n\le \gamma < 1$.
We now focus on the square loss, $\ell_{y_i}(u) = (y_i - u)^2$
with no penalty ($g=0$ in \eqref{hbeta}),
so that $\bm\hbeta$ is the least-squares estimate
$\bm\hbeta = (\bm X^T\bm X)^{-1}\bm X^T \bm y$.
We emphasize that here, $\bm y$ does not follow a linear model:
$y_i$ is allowed to depend non-linearly on $\bm x_i^T \bm w$
as in the single index model \eqref{single-index}.
In this setting, 
$\hat r^2=\frac 1 n \|\bm y - \bm X \bm \hbeta\|^2$ is the residual sum of squares
as in \eqref{v-gamma-r-squareloss},
$\hat v = 1-\frac p n$ since the degrees-of-freedom $\df$ equals $p$,
and 
$(\hat \aaa,\hat t)$ defined in \eqref{adjustments_unregularized}
satisfy
\begin{equation}
\frac{\hat t^2}{\hat r^2\hat v^2} 
=
\frac{\hat \aaa^2}{\hat r^2} = 
\frac{\|\bm X\bm \hbeta\|^2}{\|\bm X\bm \hbeta - \bm y\|^2}- \frac{p}{n-p}.
\label{square_loss_adjustments}
\end{equation}
Assuming $p/n\le \gamma < 1$,
\Cref{thm:unregularized} yields the approximation
\begin{equation}
    \label{eq:LS}
\frac{n-p}{\Omega_{jj}^{1/2}}
\Bigl[\frac{\hbeta_j}{\|\bm y-\bm X\bm \hbeta\|}
-
\frac{\pm w_j}{\sqrt n}
    \Bigl(
\frac{\|\bm X\bm \hbeta\|^2}{\|\bm X\bm \hbeta - \bm y\|^2}- \frac{p}{n-p}
\Bigr)_+^{1/2}
\Bigr]\approx N(0,1)
\end{equation}
where $\pm$ is the sign of $\bm w^T\bm\Sigma \bm\hbeta$.
If $\bm\Sigma$ is unknown, the quantity $\Omega_{jj}$ is linked
to the noise variance in the linear model of regressing $\bm X\bm e_j$
onto $\bm X_{-j}$:
$\Omega_{jj}$ can be estimated using
$$\Omega_{jj}\big\|\big[\bm I_n - \bm X_{-j}(\bm X_{-j}^T\bm X_{-j})^{-1}\bm X_{-j}^T\big]\bm X\bm e_j\big\|^2 \sim \chi^2_{n-p+1}$$
where $\bm X_{-j}\in\R^{n\times (p-1)}$ has the $j$-th column removed.
While \eqref{eq:LS} does not formally follow from \Cref{thm:unregularized}
because the square loss fails to satisfy \Cref{assumExtra}, 
the argument of the proof of \Cref{thm:unregularized} only requires minor modifications to obtain \eqref{eq:thm-Z_j-unregularized} and \eqref{eq:LS}
for the square loss (thanks to $\ell_{y_0}''=1$, the proof for the square loss is actually much simpler than for the logistic loss and other loss functions covered by \Cref{thm:unregularized}).
Alternatively, \eqref{eq:LS} follows from \Cref{thm:nonseparable-mathcalL}
with $\mathcal L(\bm v) = \frac12\|\bm v - \frac{1}{\|\bm y\|}\bm y\|^2$.

For fixed $p$ and $n\to+\infty$, asymptotic normality and confidence intervals
for the least-squares
in single index models with Gaussian covariates dates back to at least 
\citet{brillinger1982gaussian}. High-dimensional estimation performance of the
least-squares and penalized least-squares is studied in 
\cite{yang2017high,plan2016high,plan2016generalized}.
With $\bm\Sigma=\bm I_p$ to simplify comparison,
asymptotic normality in \cite{brillinger1982gaussian} concerns the random
variable $\bm\hbeta - \mu\bm w$ and the estimation bounds in
\cite{yang2017high,plan2016high,plan2016generalized} bounds the estimation
error of $\|\bm\hbeta - \mu\bm w\|$ where $\mu=\E_{g\sim N(0,1)}[gF(g,U_i)]$
where $F$ is the function defining the single index model in \eqref{single-index}.
The scaled vector $\mu\bm w$ appears here because it is the minimizer
of the population minimization problem $\min_{\bm b\in\R^p} \E[(\bm x_i^T\bm b - y_i)^2]$. The constant $\mu$ is typically unknown. A major difference
with these previous results featuring the unknown $\mu$ is that the multiplicative coefficient of $w_j$
in \eqref{eq:LS} is an estimate from the data.

We illustrate the normal approximation \eqref{eq:LS}
in \Cref{table:LS} with
$n=3000, p=2400$ and four different models:
linear, logistic, Poisson, and 1-Bit compressed sensing with
a 20\% probability of flipped bits ($\P(u_i=-1)=0.2=1-\P(u_i=1)$ with $u_i$ 
independent of $\bm x_i$).

\begin{table}[t]
    \small
\begin{tabular}{@{}|l|l|l|l|l|@{}}
\toprule
& Linear & Logistic $y_i\in\{0,1\}$ & 1-bit $y_i\in\{\pm1\}$ & Poisson
\\\midrule $y_i|\bm x_i$ & $y_i\sim N(\bm x_i^T\bm w,0.5)$
                        & $\E[y_i|\bm x_i]=\frac{e^{\bm x_i^T\bm w}}{1+e^{\bm x_i^T\bm w}}$ 
                        & $y_i = u_i \sign(\bm x_i^T\bm w)$
                        & $y_i|\bm x_i\sim\text{Poisson}(e^{\bm x_i^T\bm w})$
\\\midrule$\hat \aaa\pm$std & $.999\pm.021$
                     & $.407\pm.072$ 
                     & $.475\pm.05$
                     & $1.629\pm.163$ 
\\\midrule $\aaa_*\pm$std & $.999\pm.027$  
                   & $-.413\pm.033$  
                   & $.483\pm.037$
                   & $1.637\pm.141$
\\\midrule QQplot & \includegraphics[width=1in]{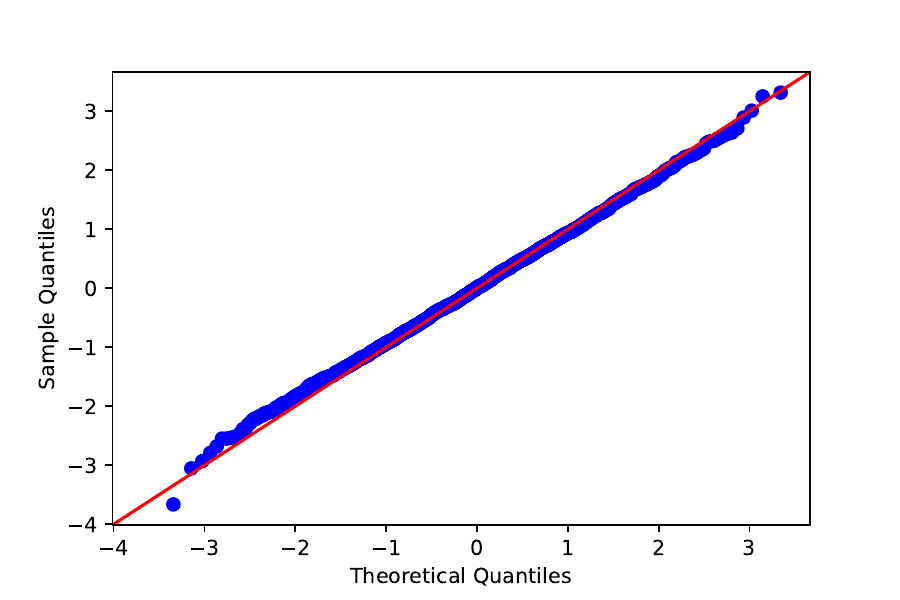}  
                  & \includegraphics[width=1in]{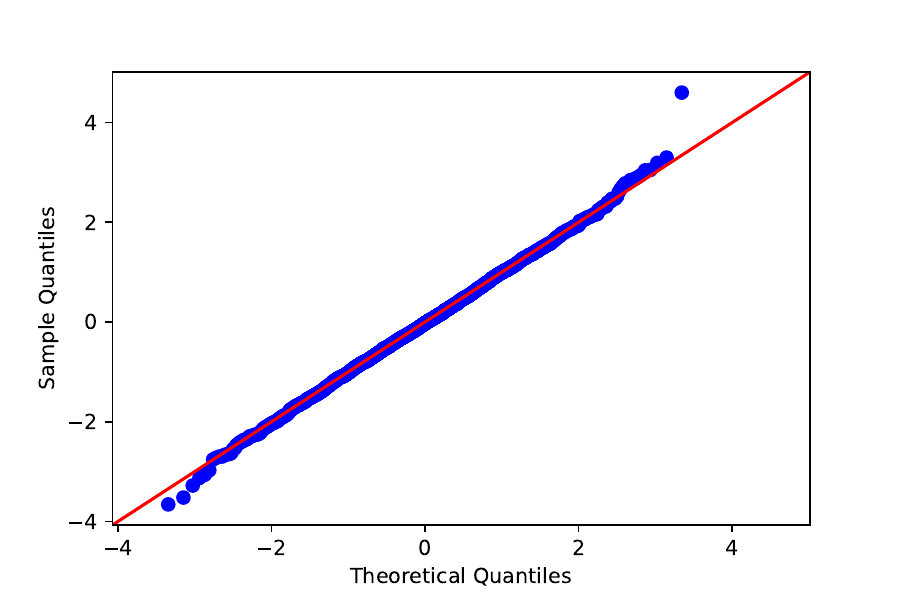}
                  & \includegraphics[width=1in]{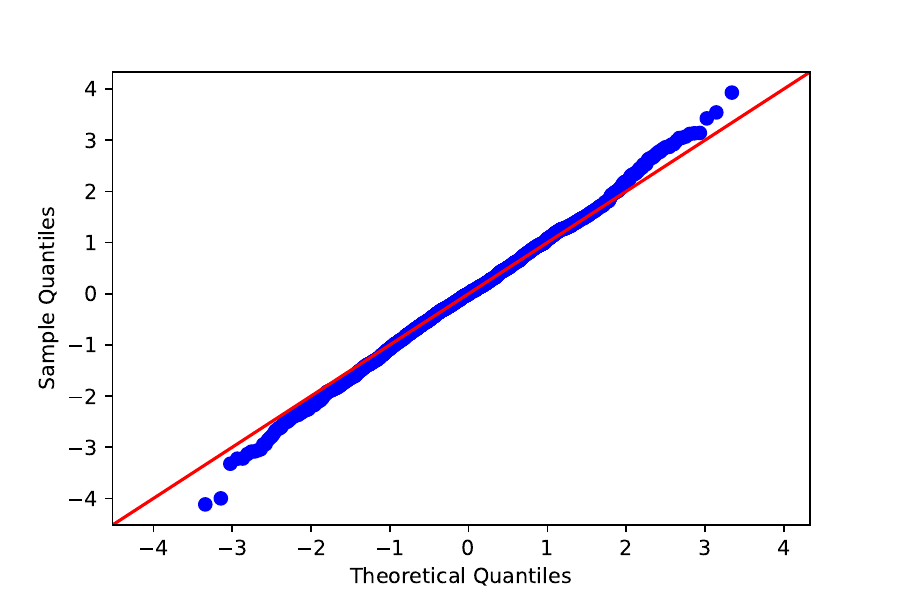}
                  & \includegraphics[width=1in]{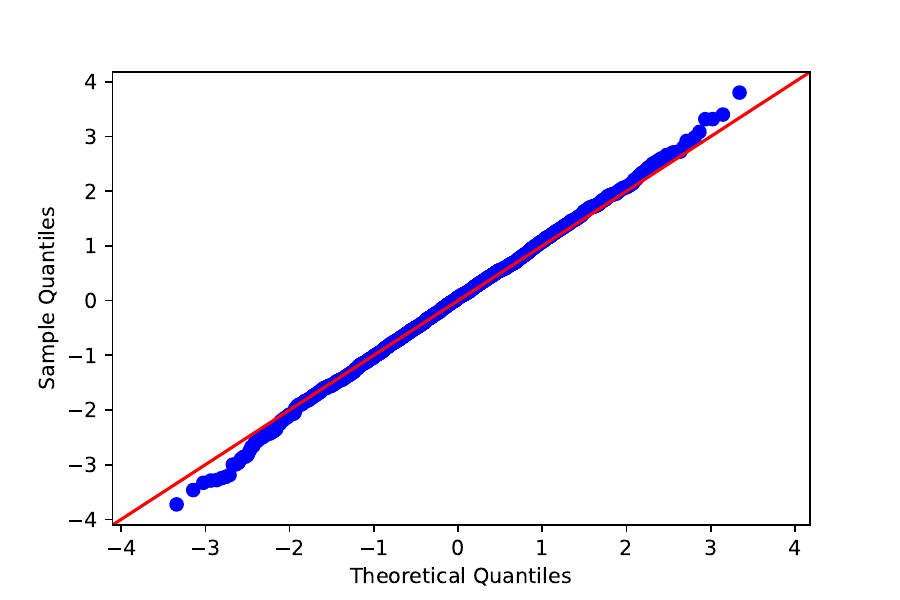} 
\\ \midrule
\end{tabular}
\caption{
    \label{table:LS}
    \small
    Quantile-quantile plots
    of the left-hand side of \eqref{eq:LS} for $n=3000, p=2400$,
    $\bm \Sigma = \bm I_p$
    and $\bm \hbeta = (\bm X^T\bm X)^{-1}\bm X^T\bm y$.
    For each model, the QQ-plot features the quantities \eqref{eq:LS}
    for all $j=1,...,p$ for a single realization of $(\bm X,\bm y)$.
    For the 1-Bit compressed sensing model, $\P(u_i=-1)=0.2=1-\P(u_i=1)$
    and $u_i$ is independent of $\bm x_i$.
    The quantity $\aaa_* = \bm w^T\bm \Sigma \bm\hbeta$
and its estimate $\hat \aaa$  are computed over 100 independent
    realizations of the data $(\bm X,\bm y)$; the corresponding
    columns in the table show the average and standard deviation
    over these 100 repetitions.
}
\end{table}

\subsection{Ridge regularized $M$-estimation}

Consider now an isotropic design with $\bm{\Sigma} = \frac1p \bm I_p$
and the Ridge penalty $g(\bm{b}) = \lambda \|\bm b\|^2/(2p)$
as for the nonlinear system \eqref{ridge_system} of \cite{salehi2019impact}.
The optimality conditions of the optimization problem \eqref{hbeta}
are 
\begin{equation}
    \label{eq:actualKKT_ridge}
\bm{X}^T \bm \hpsi = (n/p) \lambda \bm \hbeta
\end{equation}
 so that
the terms $\bm{\hpsi}^T \bm X \bm \hbeta /n$
and $\|\bm{\Sigma}^{-1/2}\bm X^T\bm \hpsi\|^2/n^2$
in \eqref{hat_v_r_t} and \eqref{eq:tilde-t-aaa-sigma} reduce to
\begin{equation}
    \label{KKT-ridge}
\bm{\hpsi}^T \bm X \bm \hbeta /n =
\lambda \|\bm{\hbeta}\|^2/p,
\qquad
\|\bm{\Sigma}^{-1/2}\bm X^T\bm \hpsi\|^2/n^2
= \lambda^2 \|\bm{\hbeta}\|^2/p.
\end{equation}
Computing explicitly 
$(\partial/\partial x_{ij})\bm\hbeta(\bm{y},\bm X)$
in \Cref{thm:derivatives} for this Ridge penalty yields
$\bm{\hat{A}} = (\bm X^T\bm D\bm X + \lambda \frac n p \bm I_p)^{-1}$.
This implies that
$\gamma_* \defas \trace[\bm{\Sigma} \bm{\hat{A}}]$
satisfies
$$n\lambda\gamma_* + \df 
= \trace[\lambda \tfrac n p \bm{\hat{A}} + \bm X^T\bm D \bm X\bm{\hat{A}}]
= \trace[\bm I_p]
= p$$
by definition of $\df$ in \Cref{thm:derivatives}.
Next, \eqref{eq:consistency-hat-gamma}
provides $\df/n = \hat v \gamma_* + O_\P(n^{-1/2})$
so that
$(\lambda + \hat v)\gamma_* = \frac p n + O_\P(n^{-1/2})$.
This justifies replacing $\gamma_*$ by $\frac p n (\lambda+\hat v)^{-1}$
in the expression of $\tilde \aaa^2$, which provides the approximations
\begin{equation}
    \left\{
    \begin{aligned}
        \tilde t^2 &= (\hat v + \lambda)^2 \tfrac 1 p\|\bm\hbeta\|^2 - (p/n)\hat r^2,
        \\\tilde \aaa^2 &\approx
         \tfrac 1 p \|\bm\hbeta\|^2 - \tfrac{p/n}{(\hat v + \lambda)^2}\hat r^2
        = \tfrac{ \hat t^2 }{(\hat v + \lambda)^2}
    ,
    \\
    \tilde\sigma^2 &= \tfrac 1 p \|\bm{\hbeta}\|^2
    -\tilde \aaa^2
    \approx
    \tfrac{p/n}{(\hat v + \lambda)^2} \hat r^2
    \approx \tfrac{n}{p} \gamma_*^2 \hat r^2
    .
    \end{aligned}
    \right.
    \label{ridge_adjustments}
\end{equation}
In this setting,
the normal approximation \eqref{eq:informal-approximation-theorem31}
from \Cref{thm:confidence-intervals-Sigma} becomes
\begin{equation}
\label{Z_j-approx-ridge}
\frac{(n/p)^{1/2}}{\hat r}\Bigl[(\hat v+\lambda)\hat\beta_j - \pm \hat t w_j \Bigr]\approx Z_j \quad \text{ with } Z_j\sim N(0,1).
\end{equation}
We illustrate these results for Ridge regularized M-estimation
with the simulation study in \Cref{table:ridge} for the logistic loss
in the logistic model.

\begin{table}[t]
    \small
\begin{tabular}{|l|r|r|r|}
\toprule $\lambda$ &      0.01 &      0.10 &      1.00 \\
\midrule
$\frac 1 n \|\bm X\bm\hbeta-\hat\gamma\bm\hpsi\|^2 - \frac{p/n}{(\hat v+\lambda)^2}\hat r^2$
       &  0.621942$\pm$0.160021 &  0.168159$\pm$0.043360 &  0.016483$\pm$0.004789 \\
\midrule
$\frac 1 p \|\bm\hbeta\|^2 - \frac{p/n}{(\hat v+\lambda)^2}\hat r^2$
  &  0.630087$\pm$0.167536  &  0.170862$\pm$0.039237 &  0.016765$\pm$0.003354 \\
\midrule
$\aaa_*^2 = \frac{1}{p}\frac{(\bm w^T\bm\hbeta)^2}{\|\bm w\|^2}$       &  0.610240$\pm$0.039087 &  0.164714$\pm$0.009765 &  0.016184$\pm$0.000914 \\
\midrule
\eqref{Z_j-approx-ridge} for $j:w_j=0$ &
\includegraphics[width=1.3in]{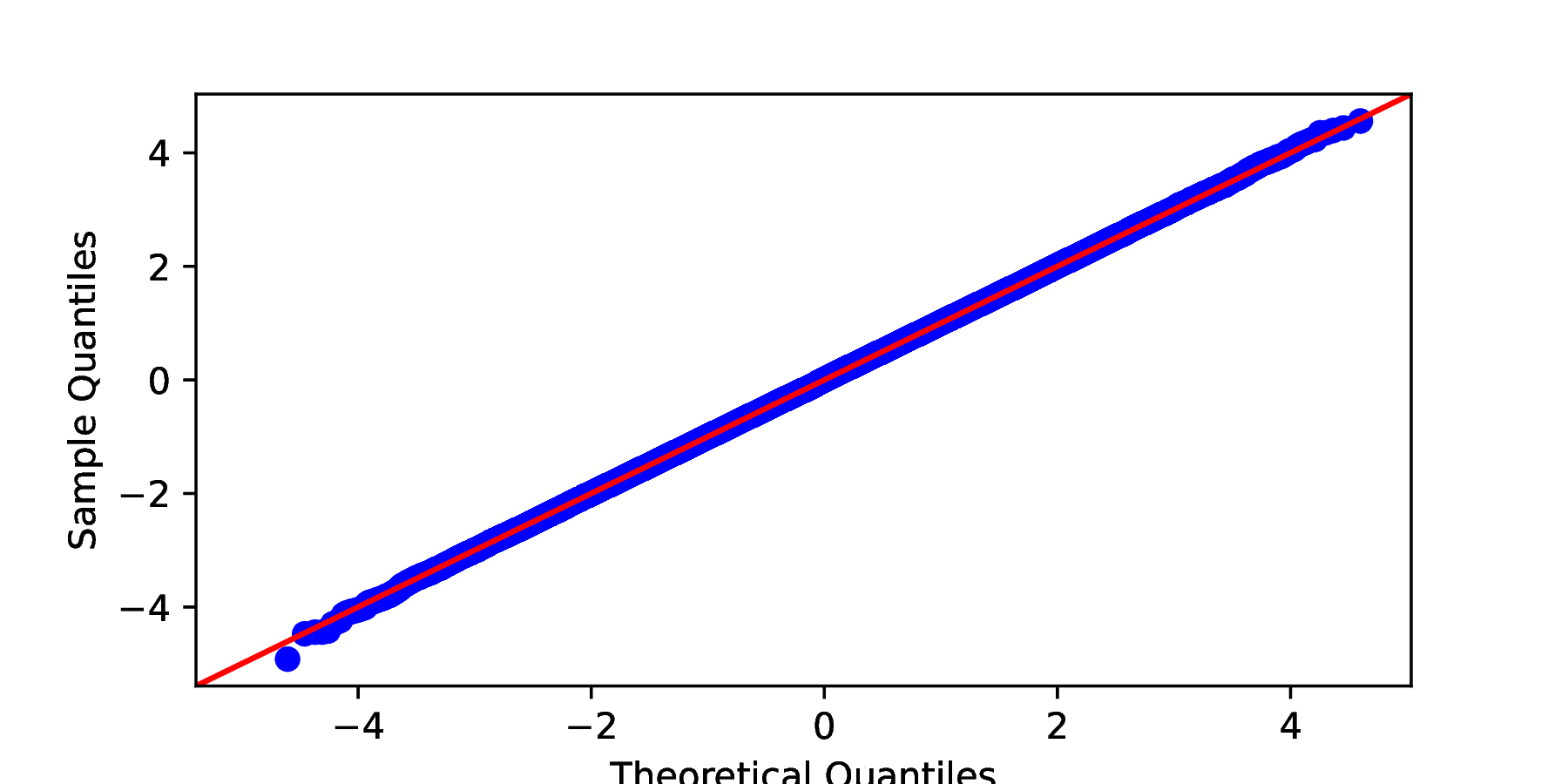}
&\includegraphics[width=1.3in]{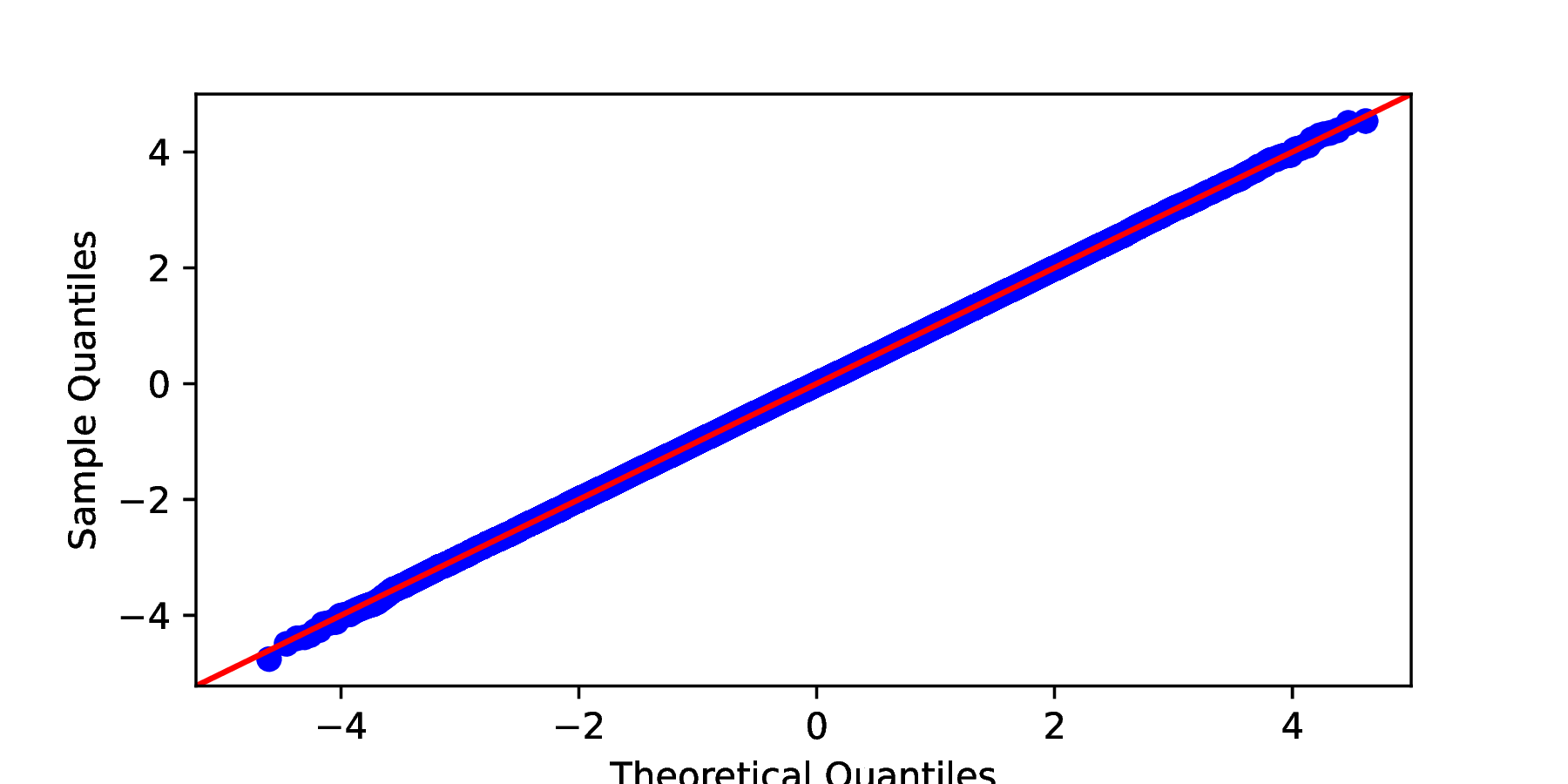}
&\includegraphics[width=1.3in]{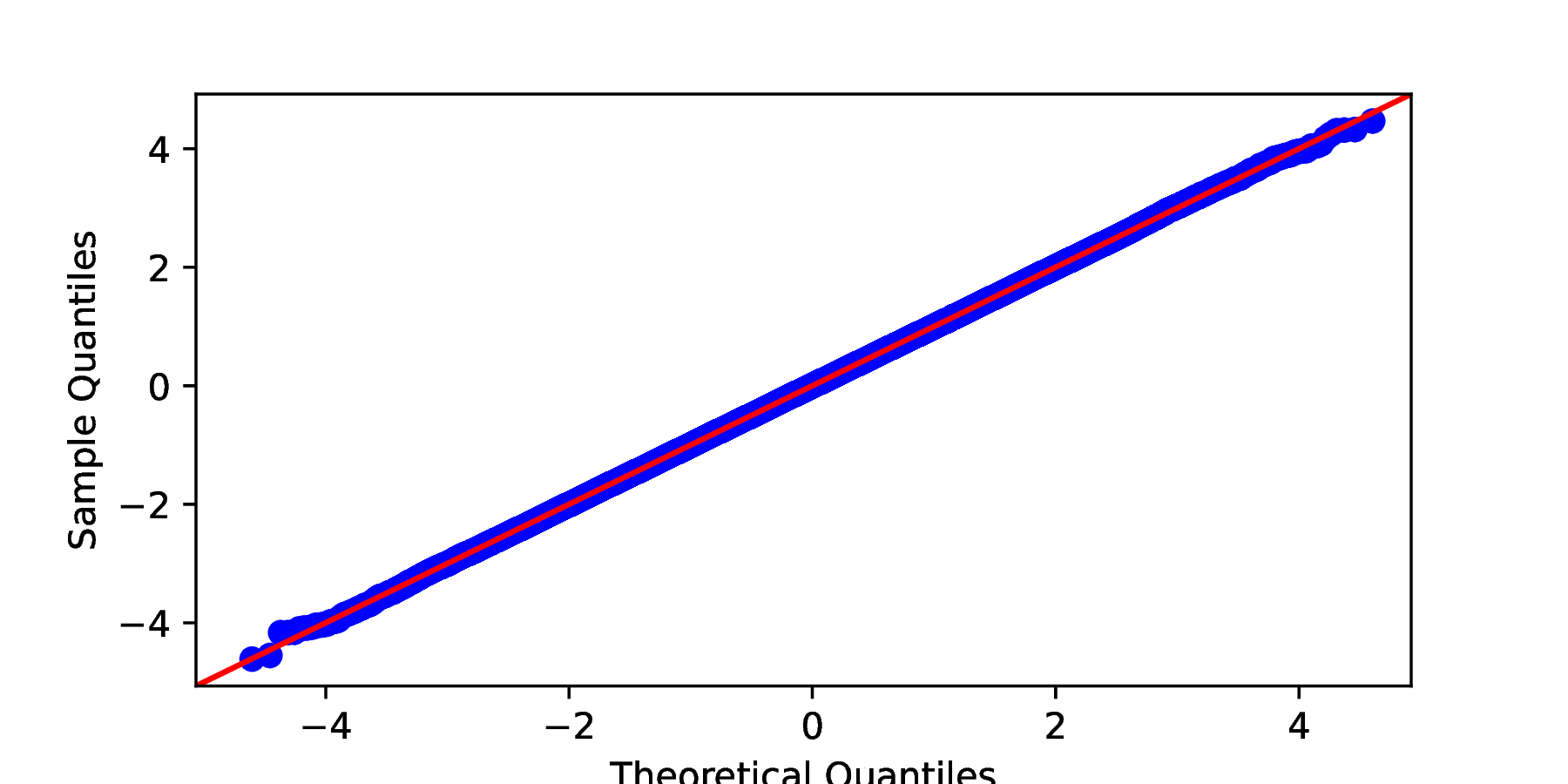}
\\
\midrule
\eqref{Z_j-approx-ridge} for $j:w_j\ne 0$ 
& \includegraphics[width=1.3in]{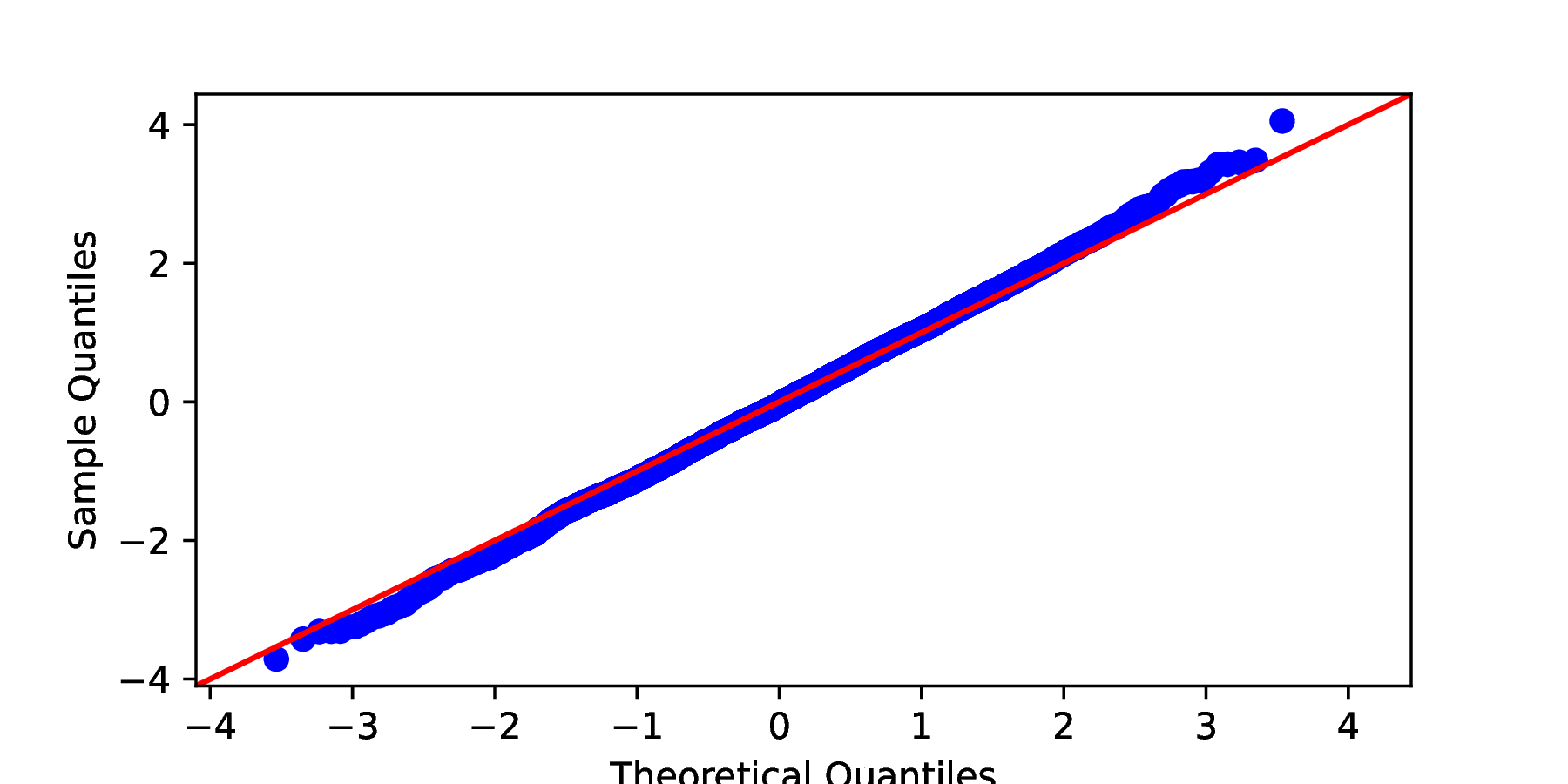}
& \includegraphics[width=1.3in]{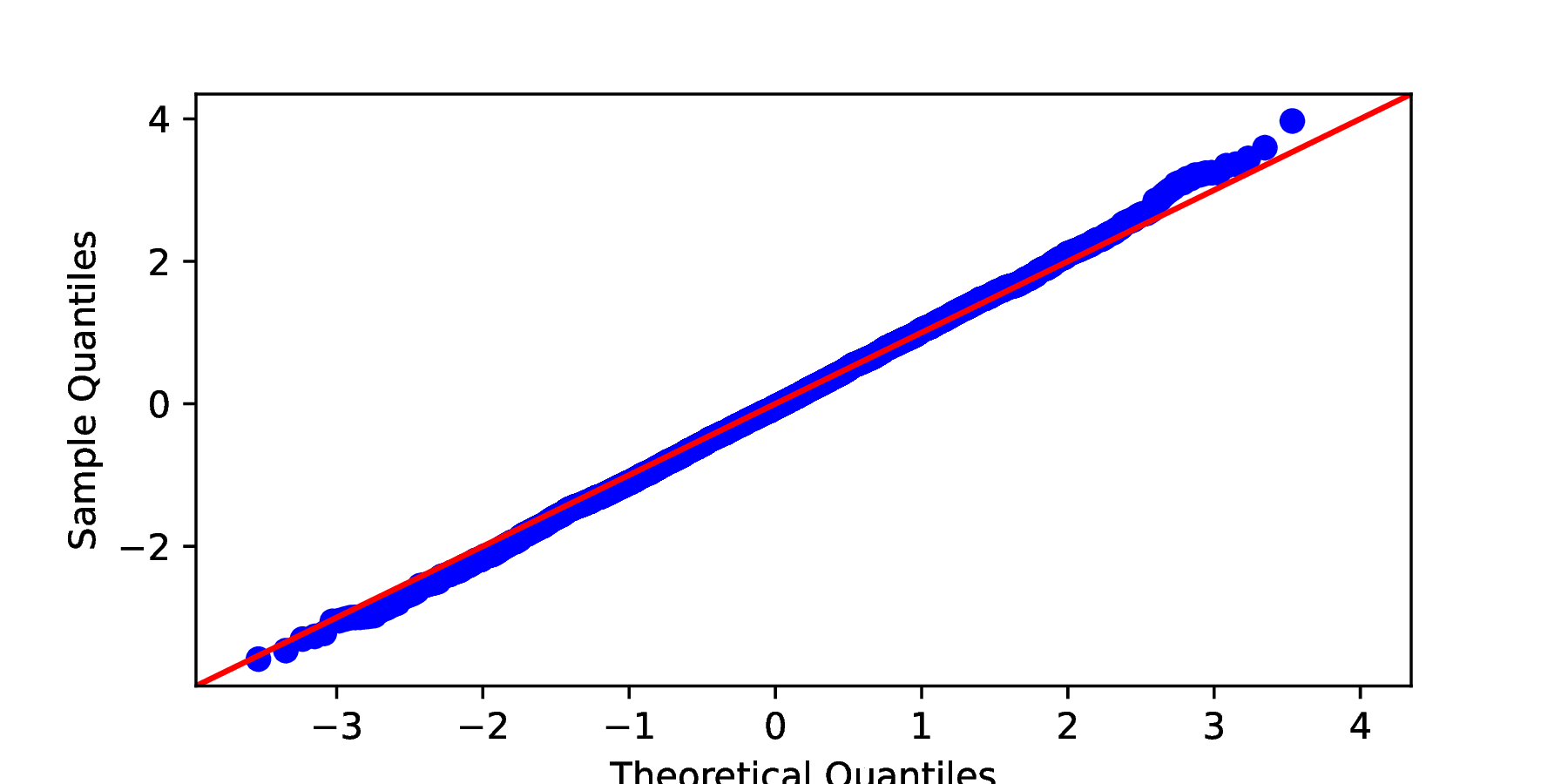}
& \includegraphics[width=1.3in]{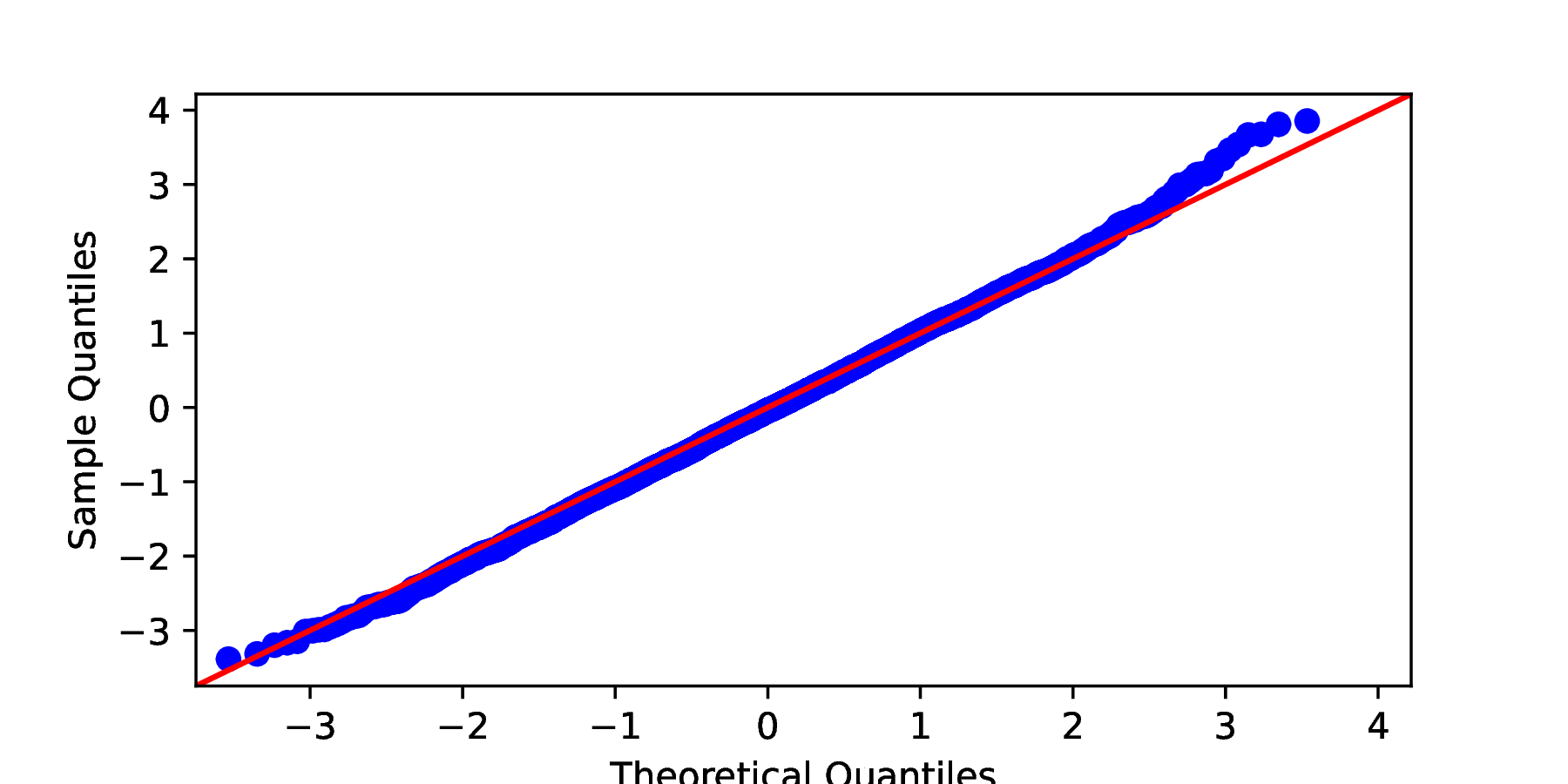}
\\
\midrule
$\frac{p/n}{(\hat v + \lambda)^2}\hat r^2$ &  8.540738$\pm$0.265670   &  2.105677$\pm$0.047351   &  0.185784$\pm$0.001738 \\
\midrule
$\sigma_*^2 =\frac 1 p \|(\bm I_p-\frac{\bm w\bm w^T}{\|\bm w\|^2})\bm\hbeta\|^2$   &  8.560584$\pm$0.110858   &  2.111824$\pm$0.014522   &  0.186364$\pm$0.001680 \\
\midrule
\end{tabular}
\caption{
    \label{table:ridge}
    In the logistic
    model $y_i\in\{0,1\}$, $\P(y_i=1) = 1/(1+\exp(-\bm x_i^T\bm w))$
    with
    $\bm x_i\sim N(\bm 0,\bm\Sigma)$ and isotropic covariance 
    $\bm\Sigma=\frac 1 p \bm I_p$,
    the M-estimator $\bm\hbeta$ is constructed with logistic loss
    $\ell_{y_i}(u)=\log(1+e^{u})-y_iu$
    and Ridge penalty $g(\bm b)=\lambda \|\bm b\|^2/p$.
    Dimension and sample size are $(n,p)=(5000,10000)$ and
    the three tuning parameters $\lambda \in\{0.01, 0.1, 1.0\}$ are used.
    For $\aaa_*^2$, its estimates
    $\frac 1 n \|\bm X\bm\hbeta-\hat\gamma\bm\hpsi\|^2 - \frac{p/n}{(\hat v+\lambda)^2}\hat r^2$
    and
    $\frac 1 p \|\bm\hbeta\|^2 - \frac{p/n}{(\hat v+\lambda)^2}\hat r^2$,
    as well as
    $\sigma_*^2$ and its estimate
    $\frac{p/n}{(\hat v+\lambda)^2}\hat r^2$,
    we report
    the average$\pm$standard deviation over 50 independent 
    repetitions of the dataset. The index $\bm w$ has $s=100$ entries equal to $\sqrt{p/s}$ and
    $p-s$ entries equal to 0.
    Middle rows show standard normal QQ-plots of
    $\{\hat r^{-1}(n/p)^{1/2}[(\hat v + \lambda)\hbeta_j - \pm\hat t w_j]\}_{j=1,...,p}$ (as in \eqref{Z_j-approx-ridge})
    for the null covariates $(j\in[p]: w_j=0$) collected over the 50 repetitions (the QQ-plot thus featuring $50(p-s)$ points),
    and for the non-nulls $(j\in[p]: w_j \ne0)$
    (the QQ-plot thus featuring $50s$ points).
}
\end{table}

Because the simple structure of the matrix $\bm{\hat{A}}$ and the KKT conditions \eqref{eq:actualKKT_ridge}, for isotropic design the estimation error
of $\aaa_*^2 = (\bm w^T\bm\Sigma\bm\hbeta)^2$ and of $\sigma_*^2 = \|\bm\Sigma^{1/2}\bm\hbeta\|^2 - \aaa_*^2$
can be slightly improved compared to the estimation error in
\eqref{eq:consistency-aaa}, and the proof is significanly simpler.
With $\bm\Sigma = \frac 1 p \bm I_p$ as in the present subsection,
we have $\gamma_* = \frac 1 p \trace[\bm{\hat{A}}]$,
$\aaa_* = \frac{1}{p}\bm w^T\bm\hbeta$
and
$\sigma_*^2 = \frac 1 p\|\bm\hbeta\|^2 - \aaa_*^2$.

\begin{restatable}{proposition}{easyRidgeProposition}
    \label{prop:improved-ridge}
    Let \Cref{assum} be fulfilled
    with $\bm\Sigma = \frac 1 p \bm I_p$. Set 
    $g(\bm b) = \lambda \|\bm b\|_2^2/p$ as the penalty in \eqref{hbeta}.
    If additionally $\sup_{y_0\in\mathcal Y}|\ell_{y_0}'| \le 1$ then
    \begin{equation}
        \label{eq:conclusion-improved-ridge}
    \E\Bigl[
    \Bigl(
        \sigma_*^2
        - \frac{\gamma_* \hat r^2}{\lambda+\hat v}
    \Bigr)^2
    \Bigr]
    \le \frac{\C(\delta,\lambda)}{p},
    \quad
    \E\Bigl[
    \Bigl(
        \aaa_*^2
        -
        \Bigl\{
        \frac{1}{p} \|\bm\hbeta\|^2
        -\frac{\gamma_* \hat r^2}{\lambda+\hat v}
        \Bigr\}
    \Bigr)^2
    \Bigr]
    \le \frac{\C(\delta,\lambda)}{p}.
    \end{equation}     
\end{restatable}

The proof of \Cref{thm:adjustments-approximation} for general penalty functions
encompasses \eqref{eq:conclusion-improved-ridge}
up to multiplicative factor $\hat v^2\hat t^2/\hat r^4$ in
\eqref{eq:consistency-aaa}. The bound \eqref{eq:conclusion-improved-ridge} avoids such multiplicative factor.
We give a short proof of \eqref{eq:conclusion-improved-ridge}
in \Cref{sec:proof_easy_ridge},
in order to
present a concise overview of some of the techniques 
behind the proof of more general result in
\Cref{thm:adjustments-approximation}.

%


\subsection{L1 regularized $M$-estimation}
\label{sec:L1}

The last example concerns L1 regularized M-estimation,
with penalty $g(\bm b) = \lambda_{n,p}\|\bm b\|_1$
where $\lambda_{n,p}>0$ is a tuning parameter.
Throughout \Cref{sec:L1}, let $\hat S =\{j\in[p]: \hbeta_j\ne 0\}$ be the set of active covariates
of the corresponding L1 regularized M-estimator in \eqref{hbeta}.
Then, by now well-understood arguments for the Lasso
(cf. \cite{tibshirani2012,tibshirani2013lasso}\cite[Proposition 3.10]{bellec_zhang2018second_stein},\cite[Proposition 2.4]{bellec2020out_of_sample}), the KKT conditions
of \eqref{hbeta} hold strictly with probability one, and the formulae
\eqref{eq:derivatives-psi-hbeta} hold true with $\bm{\hat{A}}$ symmetric and diagonal by block with
$\bm{\hat{A}}_{\hat S,\hat S} = (\bm X_{\hat S}^T\bm D \bm X_{\hat S})^{-1}$
and $\bm{\hat{A}}_{\hat S^c, \hat S^c} = \bm 0$.
This implies that $\df$ defined in \eqref{eq:def-V-df-bound-ideal-case}
is simply
\begin{equation}\df = \trace[\bm{\hat{A}}\bm X^T\bm D\bm X ] =
\trace[
\bm{\hat{A}}_{\hat S,\hat S}
(\bm{\hat{A}}_{\hat S,\hat S})^{-1}
]
=
|\hat S|
\label{df-L1}
\end{equation}
for almost every $\bm X$
if the diagonal matrix $\bm D$ has at least $|\hat S|$ positive entries.
This motivates
the notation $\df$ for the effective degrees-of-freedom
of $\bm\hbeta$, following tradition on the Lasso in linear regression 
\cite{zou2007degrees,tibshirani2012} in the context of
Stein's Unbiased Risk Estimate \cite{stein1981estimation}.

In previous sections, the algebraic nature of the loss
and/or the penalty allowed some major simplifications
in the expression of $\hat\aaa^2, \hat t^2$ and other adjustments
in \eqref{hat_v_r_t}, 
see for instance
\eqref{square_loss_adjustments}
for the square loss 
and \eqref{ridge_adjustments}
for Ridge regularized estimates
under isotropic design.
In this case of L1-regularized M-estimation, no such simplification
other than $\df=|\hat S|$ 
is available and the estimator
$\hat\aaa^2$ in \eqref{hat_v_r_t} of the squared correlation
$(\bm \hbeta^T\bm\Sigma\bm w)^2$
has expression
\begin{equation}
    \hat \aaa^2 =
    \frac{ \big(
    \tfrac{\hat v}{n}\|\bm{X} \bm \hbeta - \hat\gamma \bm \hpsi\|^2
    + \tfrac{1}{n}\bm{\hpsi}^\top\bm X\bm \hbeta
    - \hat\gamma \hat r^2
\big)^2
}{
        \tfrac{1}{n^2}\|\bm{\Sigma}^{-1/2}\bm{X}^T\bm \hpsi\|^2
        +  \tfrac{2 \hat v}{n} \bm{\hpsi}^T\bm X\bm \hbeta
        + \tfrac{\hat v^2}{n}
        \|\bm{X} \bm \hbeta - \hat\gamma \bm \hpsi\|^2
        -
        \tfrac{p}{n}\hat r^2
}.
\label{hat_a_L1}
\end{equation}
\begin{figure}[t]
    \includegraphics[height=2.2in]{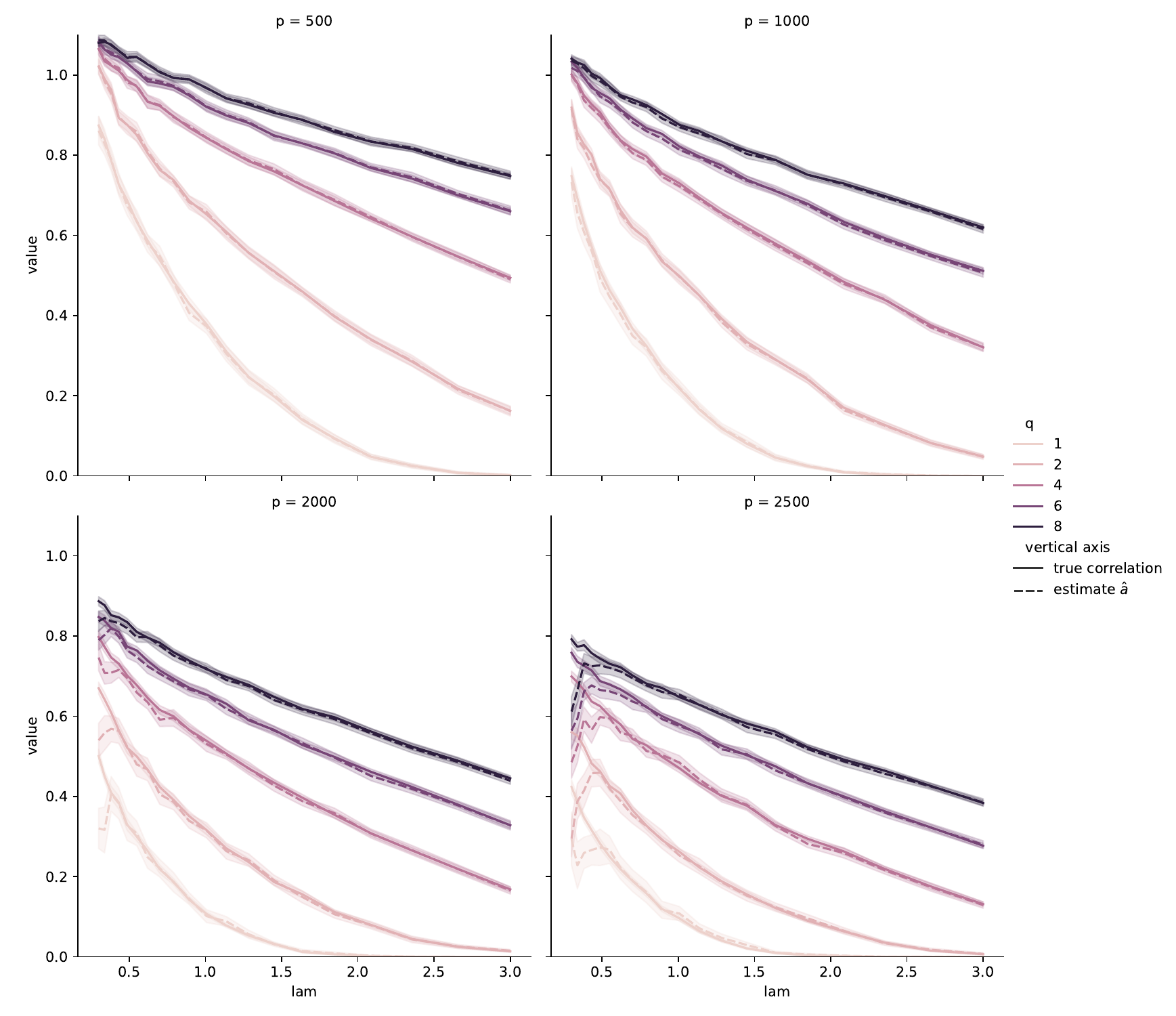}
    \includegraphics[height=2.2in]{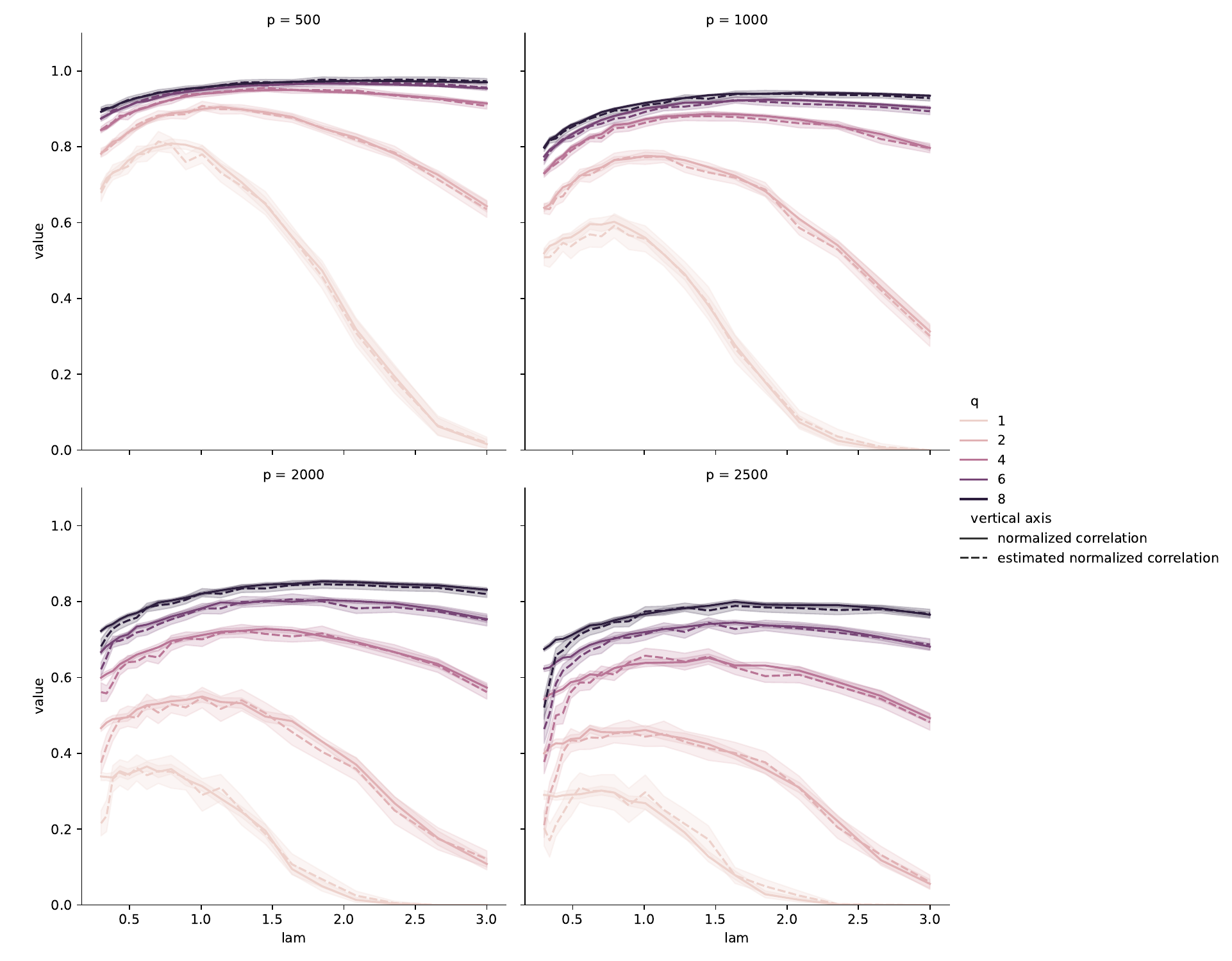}
    \caption{Left:
        Correlation
        $\aaa_*=\bm w^T\bm\Sigma\bm\hbeta$ and its 
        estimate $\sqrt{\hat a^2}$ for $\hat a^2$ in \eqref{hat_a_L1}
        for $n=1000$, in the binomial model $y_i|\bm x_i\sim\text{Binomial}(q,\rho'(\bm x_i^T\bm\beta^*))$ and the L1-penalized M-estimator
        \eqref{hbeta-L1}. The precise simulation setting is described
        in \Cref{sec:L1}.
        The x-axis represents the tuning parameter $\lambda$ in
        \eqref{hbeta-L1}.
    }
    \label{fig:L1}
\end{figure}
We could not think of
any intuition to suggest
the validity this expression, except from the algebra
and the probabilistic inequalities used in the proof
of \eqref{eq:consistency-aaa}.
We verify in \Cref{fig:L1} the validity of this estimate
under the following simulation setting.
Let $\bm\Sigma=\bm I_p$, $n=1000$ and $p\in\{500, 1000, 2000, 2500\}$.
For each $q\in\{1,2,4,6,8\}$, the response $y_i$ is generated from the
model $y_i|\bm x_i \sim \text{Binomial}(q, \rho'(\bm x_i^T\bm\beta^*))$
where $\rho'(t)=1/(1+e^{-t})$ is the sigmoid. The unknown vector
$\bm\beta^*$ has $\frac{p}{20}$ nonzero coefficients and is 
deterministically generated as follows:
first generate $\bm\beta^0$ with the first $\frac{p}{20}$ coefficients
being equispaced in $[0.5,4]$ and with remaining coefficients equal to 0,
second set $\bm\beta^* = 1.1 \bm\beta^0 / \|\bm\Sigma^{1/2}\bm\beta^0\|$.
For each dataset, we compute the regularized M-estimate
\begin{equation}
\bm\hbeta = \argmin_{\bm b\in\R^p}
\frac1n
\sum_{i=1}^n 
\Bigl[
q \log(1+e^{\bm x_i^T\bm b})
- y_i \bm x_i^T\bm b
\Bigr]
+ \frac{\lambda}{\sqrt n}\|\bm b\|_1
\label{hbeta-L1}
\end{equation}
for each tuning parameter $\lambda$ in a logarithmic grid
of cardinality 20 from $0.3$ to $3.0$.
\Cref{fig:L1} reports averages (as well as error bands depicting 
standard errors times 1.96),
over 50 independent copy of the dataset $(\bm X, \bm y)$, of
$\sqrt{\hat a^2}$ for $\hat a^2$ in \eqref{hat_a_L1} versus $\aaa_*=\bm w^T\bm\Sigma\bm\hbeta$
in the left plot of \Cref{fig:L1},
and of the normalized correlation 
$\sqrt{\hat a^2}/(0.01+\|\bm\Sigma^{1/2}\bm\hbeta\|^2)^{1/2}$ versus $\aaa_*/(0.01+\|\bm\Sigma^{1/2}\bm\hbeta\|^2)^{1/2}$.
We add 0.01 to the denominator of the normalized correlation
to avoid numerical instability
for large tuning parameters for which $\bm\hbeta=\bm 0_p$ holds in
some repetitions.
\Cref{fig:L1} shows that \eqref{hat_a_L1} is accurate across
different dimensions, different tuning parameter $\lambda$ and
different binomial parameter $q$. Estimation inaccuracies start to appear 
at small tuning parameter $\lambda$
and for the largest values of the dimension $p$.


\renewcommand{\doi}{}
\bibliographystyle{plainnat}
\bibliography{../../bibliography/db}

\newpage
\appendix

\subsection*{SUPPLEMENT}
\section{Derivatives}
\label{sec:proof_derivatives}

\thmDerivatives*

\begin{proof}[Proof of \Cref{thm:derivatives}]
    Throughout, let $\bm{y}\in\R^n$ be fixed.
    Let $\bm{X},\bm{\tilde X}\in\R^{n\times p} $ with corresponding minimizers $\bm \hbeta=\bm \hbeta(\bm y,\bm X),\bm \tbeta=\bm\hbeta(\bm y,\bm {\tilde X})$,
    for the same response vector $\bm{y}\in\R^n$.
    Let also $\bm{\tpsi} = \bm{\hpsi}(\bm{y},\bm{\tilde X}\bm{\tbeta})$
    be the counterpart of $\bm{\hpsi}=\bm \hpsi(\bm y,\bm X)$ for $\bm{\tilde X}$.
    The KKT conditions read
    $\bm{X}^T\bm{\hpsi} \in n \partial g(\bm \hbeta)$
    and
    $\bm{ \tilde X}^T\bm{\tpsi}  \in n \partial g(\bm \tbeta)$.
    Multiplying by $\bm{\hbeta} - \bm \tbeta$ and taking the difference we find
    \begin{align}
        \label{eq:fundamental-KKT}
        \qquad 
        & n(\bm{\hbeta}-\bm \tbeta)^T(\partial g(\bm \hbeta) - \partial g(\bm \tbeta))  
        +
        (\bm{X}\bm \hbeta - \bm{\tilde X}\bm \tbeta)^T(\ell_{\bm y}'(\bm X\bm \hbeta) - \ell_{\bm y}'(\bm{\tilde X}\bm \tbeta))
        \nonumber
        \\ = & (\bm{\hbeta} - \bm\tbeta)^T
        [\bm{X}^T\bm  \hpsi - \bm{ \tilde X}^T \bm \tpsi]
        +
        (\bm{X}\bm \hbeta - \bm{ \tilde X }\bm \tbeta)^T(\bm \tpsi - \bm \hpsi)
        \\=&
        (\bm{\hbeta}-\bm \tbeta)^T(\bm X-\bm{ \tilde X})^T\bm \hpsi
        +
        \bm{\hbeta}^T(\bm{\tilde X} - \bm X)^T(\bm \hpsi - \bm \tpsi).
        \nonumber
    \end{align}
    Note that by convexity of $g$ and of $\ell$, the two terms
    in the first line are non-negative.
    If $g$ is strongly convex ($\tau>0$ from \Cref{assumStrongConvex}), the first term on the first line is
    bounded from below as follows: 
    $\mu\|\bm{\hbeta}-\bm \tbeta\|^2 \le 
    n (\bm{\hbeta} - \bm \tbeta)^T(\partial g(\bm \hbeta) - \partial g(\bm \tbeta))$
    for some constant $\mu>0$
    (e.g, $\mu=\phi_{\min}(\bm{\Sigma}) \tau n$ works).
    For the second term in the first line of \eqref{eq:fundamental-KKT},
    \begin{equation}
        \label{eq:lower-bound-1-Lipschitz-ell}
    \|\bm{\hpsi} - \bm \tpsi\|^2
    =
    \|\ell_{\bm{y}}'(\bm X\bm \hbeta) - \ell_{\bm y}'(\bm{ \tilde X}\bm \tbeta)\|^2
    \le
    (\bm{X}\bm \hbeta - \bm{\tilde X}\bm \tbeta)^T(\ell_{\bm y}'(\bm X\bm \hbeta) - \ell_{\bm y}'(\bm{\tilde X}\bm \tbeta))
    \end{equation}
    since $\ell_{y_0}'$ is increasing and 1-Lipschitz 
    for all $y_0\in\R$.
    Using
    $
    (\bm{\hbeta}-\bm \tbeta)^T(\bm X-\bm{\tilde X})^T\bm \hpsi
        +
        \bm{\hbeta}^T(\bm{\tilde X} -\bm  X)^T(\bm \hpsi - \bm \tpsi)
        \le \|\bm{X}-\bm{\tilde X}\|_{op} 
    (\|\bm{\hpsi}\|
\|\bm{\hbeta}-\bm \tbeta\| 
    + \|\bm{\hbeta}\| \|\bm \hpsi-\bm \tpsi\|)$
    we find
    $$
    \mu\|\bm{\hbeta}-\bm \tbeta\|^2 + \|\bm \hpsi -\bm \tpsi\|^2
    \le \|\bm{\hbeta}-\bm \tbeta\| \|\bm X-\bm{\tilde X}\|_{op} \|\bm{\bm \hpsi}\| + \|(\bm X-\bm{\tilde X})\bm \hbeta\|^2
    \|\bm \hpsi - \bm \tpsi\|.
    $$
    Note that $\bm{X}\mapsto \bm{\hbeta}(\bm{y},\bm{X})$ is continuous
    as $\bm{\hbeta}(\bm{y},\bm{X})$ is the unique minimizer, by strong convexity,
    of the continuous objective function $L_{\bm{y}}(\bm{X},\bm{b})=\frac 1 n \sum_{i=1}^n\ell_{y_i}(\bm{x}_i^T\bm{b}) + g(\bm{b})$. 
    By continuity
    $\sup_{\bm{X}\in K}\|\bm{\hpsi}(\bm{y},\bm{X})\| + \|\bm{\hbeta}(\bm{y},\bm{X})\|$ is bounded
    for any compact $K$.
    This proves that $\bm{X} \mapsto \bm{\hbeta}(\bm{y},\bm{X})$ is Lipschitz in 
    any compact
    and differentiable almost everywhere in $\R^{n\times p}$
    by Rademacher's theorem.

    Recall that $\bm y$ is fixed throughout this proof. Assume that $\bm{X}\mapsto \bm{\hbeta}(\bm{y},\bm{X})$ is differentiable at $\bm X$. This means that for some linear map $\bm{\dbeta}_{\bm X}:\R^{n\times p}\to \R^p$ we have for $\bm{\dX}\in\R^{n\times p}$
    $$
    \bm{\hbeta}(\bm y,\bm X + \bm{\dX})
    -
    \bm{\hbeta}(\bm y,\bm X)
    =
    \bm{\dbeta}_{\bm X}(\bm{\dX})
    + o(\|\bm{\dX}\|_F) 
    $$
    and
    $\frac{\partial}{\partial x_{ij}}\bm{\hbeta}(\bm y,\bm X)
    = \bm{\dbeta}_{\bm X}(\bm e_i \bm e_j^T)$.
    In the following we consider a fixed direction $\bm{\dX\in\R^{n\times p}}$
    and write $\bm{\dbeta}$ for
    the directional derivative in direction $\bm{\dX}$, that is,
    $\bm{ \dbeta } = \frac{d}{dt} \bm{\hbeta}(\bm y, \bm X+t\bm {\bm \dX}) |_{t=0}$. 
    This is equivalent to $\bm{\dbeta} = \bm{\dbeta}_{\bm X}(\bm{\dX})$
    with the notation $\bm{\dbeta}_{\bm X}(\cdot)$ of the previous display;
    we write $\bm{\dbeta}$ for brevity.
    Then we have
    $\frac{d}{dt} (\bm{X}+t\bm{\dX})\bm{\hbeta}(\bm{y}, \bm{X}+t\bm \dX) |_{t=0}
    = \bm{X}\bm{\dbeta} + \bm \dX \bm\hbeta$ and
    $\frac{d}{dt} \bm{\hpsi}(\bm{y}, \bm{X}+t\bm \dX)) |_{t=0}
    = - \bm D(\bm{X}\bm{\dbeta} + \bm \dX \bm\hbeta)$
    by the chain rule.
    By bounding $n(\bm{\hbeta}-\bm{\tbeta})^T(\partial g(\bm{\hbeta}) - \partial g(\bm{\tbeta}))$
    from below by $n \tau\|\bm{\Sigma}^{1/2}(\bm{\hbeta}-\bm{\tbeta})\|^2$ in \eqref{eq:fundamental-KKT}
    for  $\bm{\tilde X} = \bm{X} + t\bm{\dX}$
    and $\bm{\tbeta}=\bm{\hbeta}(\bm{y},\bm{X}+t\bm{\dX})$, dividing by $t^2$
    and taking the limit as $t\to 0$, we
    find
    \begin{equation*}
    \tau n \|\bm{\Sigma}^{1/2}\bm \dbeta\|^2 + (\bm{X}\bm \dbeta + \bm \dX\bm{\hbeta})^T\bm D
    (\bm{X}\bm{\dbeta} + \bm \dX\bm{\hbeta})
    \le \bm{\dbeta}^T \bm \dX^T\bm \hpsi
    + \bm{\hbeta}^T\bm{\dX}^T \bm D (\bm{X}\bm \dbeta + \bm \dX \bm{\hbeta}).
    \end{equation*}
    Equivalently, noting that the terms $(\bm{\dX}\bm{\hbeta})^T\bm D (\bm \dX\bm{\hbeta} + \bm{X} \bm \dbeta)$ cancel out,
    \begin{equation}
        \tau n \|\bm{\Sigma}^{1/2}\bm \dbeta\|^2 + \|\bm D^{1/2}
    \bm{X}\bm{\dbeta}\|^2
    \le \bm{\dbeta}^T
    \mathcal L(\bm \dX) 
    \quad\text{ with }\quad
    \mathcal L(\bm \dX) = 
        \bm{\dX}^T\bm \hpsi
        - \bm{X}^T \bm D \bm{\dX} \bm \hbeta
        .
    \label{prev-display-derivatives}
    \end{equation}

    The matrix $n\tau \bm{\Sigma} + \bm X^T\bm D\bm X$ is positive definite
    thanks to $\tau>0$.
    Thus $\bm{\dbeta}=0$ for every direction $\bm \dX$ such that
    $\bm{\dX}^T\bm \hpsi - \bm X^T\bm D \bm \dX\bm \hbeta = 0$.
    We have established the inclusion of kernel of linear mappings
    $\R^{n\times p}\to \R^p$
    $$\ker(\bm{\dX} \mapsto
        \mathcal L(\bm{\dX})
    )\subset
    \ker(\bm{\dX} \mapsto \bm \dbeta).$$
    This implies the existence\footnote{Indeed,
    if $\ker \bm B \subset \ker \bm C$ and $\bm B$ has SVD $\bm B = \sum_i \bm u_i s_i \bm v_i^T$ then $\bm A = \sum_i \bm C \bm v_i \bm u_i^T/s_i$ is such that
    $\bm C = \bm A \bm B$.
}
    of $\bm{\hat{A}}\in\R^{p\times p}$
    with $\ker(\bm{\hat{A}})^\perp \subset \text{Range}(\mathcal L(\cdot)) = \{\mathcal L(\bm{\dX}), \bm \dX\in\R^{n\times p} \}$
    such that $\bm{\dbeta} = \bm{\hat{A}} \mathcal L(\bm \dX)$.
    The choice $\bm{\dX} = \bm{e}_i \bm{e}_j^T$ for canonical basis vectors
    $\bm{e}_i\in\R^n,\bm{e}_j\in\R^p$ gives the desired formula for $(\partial/\partial x_{ij})\bm\hbeta(\bm{y}, \bm X)$.

    Inequality \eqref{prev-display-derivatives}
    implies that for all $\bm u \in \text{Range}(\mathcal L(\cdot))\supset\ker(\bm{\hat{A}})^\perp$,
    \begin{equation}
    \tau n \|\bm \Sigma^{1/2}\bm{\hat{A}} \bm u\|^2
    + \|\bm D^{1/2}\bm X \bm{\hat{A}} \bm u\|^2
    \le \bm u^T \bm{\hat{A}}^T \bm u.
    \label{eq:basic-ineq-range-L}
    \end{equation}
    Let $\bm v$ with $\|\bm v\|=1$ such that
    $\|\bm \Sigma^{1/2}\bm{\hat{A}}\bm\Sigma^{1/2}\|_{op}
    =
    \|\bm \Sigma^{1/2}\bm{\hat{A}}\bm\Sigma^{1/2} \bm v\|$
    and $\bm v \in \ker(\bm{\hat{A}} \bm\Sigma^{1/2})^\perp$.
    Then $\bm u =\bm\Sigma^{1/2}\bm v$ satisfies
    $\|\bm \Sigma^{1/2}\bm{\hat{A}}\bm\Sigma^{1/2}\|_{op}
    = \|\bm \Sigma^{1/2}\bm{\hat{A}} \bm u\|$ and
    $\bm u\in \ker(\bm{\hat{A}})^\perp$ with $\|\bm\Sigma^{-1/2}\bm u\|=1$,
    so that the previous display gives
    $$\tau n \|\bm \Sigma^{1/2}\bm{\hat{A}} \bm\Sigma^{1/2}\|_{op}^2
    \le \bm u^T\bm{\hat{A}}^T \bm u
    \le \|\bm\Sigma^{1/2}\bm{\hat{A}} \bm\Sigma^{1/2}\|_{op}$$
    which proves \eqref{eq:boundA}.

    We now bound $\trace[\bm V]$ from below.
    For any $\bm v\in \R^n$, we would like to find some $\bm \dX$
    with $\mathcal L(\bm \dX) = \bm X^T \bm D \bm v$.
    Consider $\bm \dX$ of the form $ \bm S \bm D \bm X$ for symmetric $\bm S\in\R^{n\times n}$. Then
    $\mathcal L(\bm \dX) = \bm X^T\bm D \bm S( \bm \hpsi -\bm D \bm X \bm \hbeta )$.
    Since for any two vectors $\bm a,\bm b$ of the same dimension,
    there exists a symmetric $\bm S$ such that $\bm S \bm a = \bm b$,
    if $\bm \hpsi -\bm D \bm X \bm \hbeta \ne 0$ then we can
    always find some $\bm\dX$ such that $\bm X^T\bm D \bm v = \mathcal L(\bm \dX)$ and \eqref{eq:basic-ineq-range-L} with $\bm u = \mathcal L(\bm{\dX})$ yields
    $$
    \tau n \|\bm \Sigma^{1/2}\bm{\hat{A}} \bm X^T \bm D\bm v\|^2
    + \|\bm D^{1/2}\bm X \bm{\hat{A}} \bm X^T\bm D\bm v\|^2
    \le \bm v^T\bm D \bm X \bm{\hat{A}}^T \bm X^T \bm D \bm v.
    $$
    The LHS is further lower bounded by
    $(\tau n\|\bm D^{1/2} \bm X\bm\Sigma^{-1/2}\|_{op}^{-2}
    +1)\|\bm D^{1/2}\bm X\bm{\hat{A}} \bm X^T \bm D\bm v\|^2$.
    Since for any $\bm v'\in \R^n$ we can find $\bm v\in\R^n$
    such that $\bm D^{1/2}\bm v' = \bm D \bm v$, this shows
    that
    $$
    (\tau n\|\bm D^{1/2} \bm X\bm\Sigma^{-1/2}\|_{op}^{-2}
    +1)
    \|\bm M \bm v'\|^2
    \le
    \bm v'{}^T\bm M \bm v'
    \le
    \|\bm M\|_{op} \|\bm v'\|^2
    \quad
    \text{ for } \bm M = \bm D^{1/2}\bm X \bm{\hat{A}} \bm X^T\bm D^{1/2}.
    $$
    This proves 
    $\|\bm M\|_{op} \le  (\tau n\|\bm D^{1/2} \bm X\bm\Sigma^{-1/2}\|_{op}^{-2}
    +1)^{-1}$. Thus
    \begin{align*}
    \trace[\bm V] 
    = \trace[\bm D ] - \trace[\bm D^{\frac12}\bm M\bm D^{\frac12}]
    &= \trace[\bm D ] - \trace[\bm D^{\frac12}\bm M^s\bm D^{\frac12}]
  \\&\ge 
    \trace[\bm D]\bigl(1-
        (\tau n\|\bm D^{1/2} \bm X\bm\Sigma^{-\frac12}\|_{op}^{-2}
    +1)^{-1}
    \bigr)
    \end{align*}
    for $\bm M^s = \frac 1 2 (\bm M + \bm M^T)$ the symmetric part.
    We have established the desired lower bound on $\trace\bm V$
    in the case $\bm \hpsi - \bm D \bm X \bm\hbeta \ne\bm 0$.
    The previous displays also show that $\bm M^s$ is psd, so that
    $\trace[\bm V] \le \trace[\bm D] \le n$ and
    $0 \le \df=\trace[\bm M^s] \le n$. This proves
    \eqref{eq:def-V-df-bound-ideal-case} if 
    $\bm \hpsi \ne \bm D \bm X \bm\hbeta$.

    The situation is more delicate if $\bm \hpsi - \bm D \bm X \bm\hbeta=\bm 0$. 
    In this case, $\mathcal L(\bm \dX)=(\bm \dX^T\bm D\bm X - \bm X^T\bm D \bm \dX)\bm \hbeta$.
    If $\bm\hbeta=\bm 0$ then $\ker(\bm{\hat{A}})^\perp\subset\text{Range}(\mathcal L) = \{\bm 0\}$ implies that $\bm{\hat{A}} = \bm 0$ and all stated results \eqref{eq:boundA}-\eqref{eq:def-V-df-bound-ideal-case} hold trivially.
    Now assume $\bm\hbeta\ne \bm 0$ and Denote by $^\dagger$ the pseudo-inverse.
    Define the subspace
    $\mathcal V = \{\bm v'\in\R^n: 
    \bm \hpsi^T (\bm D^{1/2})^\dagger \bm v'=0
    \}$.
    For any $\bm v'\in \mathcal V$,
    let $\bm v = (\bm D^{1/2})^\dagger\bm v'$
    so that $\bm D\bm v = \bm D^{1/2}\bm v'$
    and $\bm \hpsi^T \bm v=0$.
    Set $\bm \dX = - \bm v \bm\hbeta^T \|\bm \hbeta\|^{-2}$
    so that $\mathcal L(\bm \dX) = \bm X^T\bm D^{1/2} \bm v'$.
    By \eqref{eq:basic-ineq-range-L} and $\bm M =\bm D^{1/2}\bm X\bm{\hat{A}}\bm
    X^T\bm D^{1/2}$ we find
    $(\tau n \|\bm D^{1/2}\bm X\bm\Sigma^{1/2}\|_{op}^{-2}+1)
    \|\bm M\bm v'\|^2
    \le
    \bm v'{}^T\bm M\bm v'
    $.
    If $\bm P_{\mathcal V}\in\R^{n\times n}$ is the orthogonal projector onto $\mathcal V$ with rank at least $n-1$,
    this proves that the symmetric matrix $\bm P_{\mathcal V}\bm M^s\bm P_{\mathcal V}$ is positive semi-definite
    with eigenvalues at most $(\tau n \|\bm D^{1/2}\bm X\bm\Sigma^{-1/2}\|_{op}^{-2}+1)^{-1}$, so that
    $$
    0 - \|\bm M^s\|_{op}
    \le
    \df = \trace[\bm M^s]
    \le (n-1) + \|\bm M^s\|_{op}.
    $$
    For $\bm V$ we find
    \begin{align*}
    &\trace[\bm V]
    -\trace[\bm D]
    +\trace[\bm D^{1/2}\bm P_{\mathcal V}\bm M^s\bm P_{\mathcal V}\bm D^{1/2}]
    =
    \trace[\bm D^{1/2}(-\bm M^s + \bm P_{\mathcal V}\bm M^s\bm P_{\mathcal V})\bm D^{1/2}].
    \end{align*}
    The matrix $\bm M^s - \bm P_{\mathcal V}\bm M^s\bm P_{\mathcal V}$
    is rank at most 2 and operator at most $2\|\bm M^s\|$.
    Thus the absolute value of the RHS
    is at most $4\|\bm M^s\|_{op}$.
    Since
    $\bm 0_{n\times n} \preceq \bm P_{\mathcal V}\bm M^s\bm P_{\mathcal V}
    \preceq
    (\tau n \|\bm D^{1/2}\bm X\bm\Sigma^{1/2}\|_{op}^{-2}+1)^{-1}
    \bm I_n$, we find the following upper and lower bounds on $\trace[\bm V]$:
    \begin{align*}
    \trace[\bm V] 
    &\ge \trace[\bm D](1-(\tau n \|\bm D^{1/2}\bm X\bm \Sigma^{-1/2}\|_{op}^{-2} + 1)^{-1})
    - 4\|\bm M^s\|_{op}
    =
    \tfrac{1}{1+\hat c}\trace[\bm D]
    - 4\|\bm M^s\|_{op}
    ,
    \\
    \trace[\bm V]
    &\le \trace[\bm D]
    + 4\|\bm M^s\|_{op}.
    \end{align*}
    We conclude with
    $\|\bm M^s\|_{op}\le \|\bm M\|_{op}
    \le \hat c = (\tau n)^{-1} \|\bm D^{1/2}\bm X\bm\Sigma^{-1/2}\|_{op}^2$
    thanks to \eqref{eq:boundA}.
\end{proof}

\section{Probabilistic tools}

\begin{lemma}
    [Variant of \cite{bellec_zhang2019second_poincare}]
    \label{lemma-SOS}
    Let $\bm{z}\sim N(\bm{0},\bm{I}_n)$ and $\bm{f}:\R^n\to \R^n$ be
    weakly differentiable with $\E \| \bm f (\bm z)\|<+\infty$.  Then  
    $Z=\|\E[\bm{f}(\bm z)]\|^{-1}\bm z^T\E[\bm f(\bm z)]\sim N(0,1)$ is such that
    $$
    \E\Bigl[
    \Bigl(
    \bm{z}^T \bm{f}(\bm{z}) - \sum_{i=1}^n \frac{\partial f_i}{\partial z_i}(\bm{z})
    - \|\bm{f}(\bm z)\| Z
    \Bigr)^2
    \Bigr]
    \le
    15
    \E\Bigl[
     \sum_{i=1}^n\sum_{l=1}^n
     \Bigl(\frac{\partial f_i}{\partial z_l}(z)\Bigr)^2
     \Bigr].
    $$
\end{lemma}
\begin{proof}
    Variants of the following argument were developed 
    in \cite{bellec_zhang2019second_poincare,bellec2021asymptotic}.
    A short proof is provided for completeness, and because the exact
    statement in \Cref{lemma-SOS} slightly differ from previous results.
    Let 
    $\bm{g}(\bm z) = \bm f(\bm z) - \E[\bm f(\bm z)]$
    and $W=\|\bm{f}(\bm z) \| - \|\E[\bm f(\bm z)]\|$.
    Then the square root of the left-hand side is
    $\E[(\bm{z}^T \bm g(\bm z) - \sum_{i=1}^n \frac{\partial g_i}{\partial z_i}(\bm z) -WZ)^2]^{1/2}$
    which is smaller than
    $\sqrt{E_1}
    + \sqrt{E_2}$
    by the triangle inequality,
    where $E_1=\E[(\bm{z}^T \bm g(\bm z) - \sum_{i=1}^n \frac{\partial g_i}{\partial z_i}(\bm z))^2]$
    and $E_2=\E[Z^2W^2]$.
    For $E_1$, by \cite{bellec_zhang2018second_stein}
    applied to $\bm g$ we find
    $$
    E_1 = \E\Bigl[\|\bm{f}(\bm z) - \E[\bm f(\bm z)]\|^2
        +\sum_{i=1}^n\sum_{l=1}^n
        \frac{\partial f_l}{\partial z_i}(\bm{z})
        \frac{\partial f_i}{\partial z_l}(\bm{z})
        \Bigr]
    \le
    2\E\sum_{i=1}^n \| \frac{\partial \bm{f}}{\partial z_i}\|^2
    $$
    by the Gaussian Poincar\'e inequality \cite[Theorem 3.20]{boucheron2013concentration} for the first term and the Cauchy-Schwarz inequality for the second.
    For $E_2$, by the triangle inequality
    $E_2 \le \E[Z^2\|g(\bm{z} )\|^2]$.
    Write $Z=\sum_{i=1}^n\sigma_i z_i$ for some $\sigma_i\ge 0$ with
    $\sum_{i=1}^n\sigma_i^2=1$.
    By Stein's formula,
    \begin{align*}
    \E\Bigl[Z^2\|\bm{g}(\bm z )\|^2\Bigr]
    =\sum_{i=1}^n
    \E\Bigl[\sigma_iz_iZ\|\bm{g}(\bm z )\|^2\Bigr]
    &=
    \sum_{i=1}^n
    \Bigl\{
    \sigma_i^2
    \E[\|\bm{g}(\bm z )\|^2] +
    \sigma_i
    \E\Bigl[Z\frac{\partial}{\partial z_i}(\|\bm{g}(\bm z )\|^2)\Bigr]
    \Bigr\}
    \\&= \E[\|\bm{g}(\bm z )\|^2]
    +2\sum_{i=1}^n\sum_{l=1}^n
    \E\Bigl[\sigma_iZ g_l(\bm{z}) \frac{\partial g_l}{\partial z_i}(\bm z)\Bigr]
    \\&\le
    RHS
    +
    2\Bigl(\sum_{i=1}^n\sum_{l=1}^n
    \E\Bigl[\sigma_i^2(Z g_l(\bm{z}))^2\Bigr]
    \Bigr)^{1/2}
    \Bigl(
        RHS
    \Bigr)^{1/2} 
    \end{align*}
    where $RHS=\sum_{i=1}^n\sum_{l=1}^n
        \E[(\frac{\partial g_l}{\partial z_i}(\bm{z}))^2]$
    thanks to $\E[\|\bm{g}(\bm z)\|^2]\le RHS$
    by the Gaussian Poincar\'e inequality for the first term
    and the Cauchy-Schwarz inequality for the second term.
    By completing the square,
    $(\E[Z^2\|\bm{g}(\bm z)\|^2]^{1/2}
    - (RHS)^{1/2})^2 \le 2 RHS$
    and $E_2 \le \E[Z^2\|\bm{g}(\bm z)\|^2]
    \le (1+\sqrt 2)^2 RHS$.
    Hence
    $\sqrt{E_1}+ \sqrt{E_2} \le (\sqrt 2 + 1+\sqrt 2) (RHS)^{1/2}$.
\end{proof}

\begin{corollary}
    \label{cor_SOS}
    Let $\bm X\in\R^{n\times p}$ with iid $N(\bm 0,\bm \Sigma)$ rows
    with invertible $\bm \Sigma$.

    (i) If $\bm a\in\R^p$ and 
    $\bm h:\R^{n\times p}\to \R^n$
    is weakly differentiable then
    for some $Z\sim N(0,1)$,
    $$
    \E_0\Bigl[
    \Bigl|\bm a^T\bm\Sigma^{-1}\bm X^T\bm h(\bm X)
    - \sum_{i=1}^n\sum_{k=1}^p a_k \frac{\partial h_i}{\partial x_{ik}}(\bm X)
    - \|\bm \Sigma^{-1/2}\bm a\|Z \|\bm h(\bm X)\|
    \Big|^2
    \Bigr]
    \le 15\E_0
    \sum_{i=1}^n
    \Big\|
    \sum_{k=1}^p a_k \frac{\partial \bm h}{\partial x_{ik}}(\bm X)
    \Big\|^2
    $$
    where $\E_0$ is the conditional
    expectation given
    $\bm X(\bm I_p - \frac{\bm \Sigma^{-1}\bm a \bm a^T}{\bm a^T\bm\Sigma^{-1}\bm a})$.

    (ii)
    If $\bm h:\mathcal Y^n \times \R^{n\times p}\to \R^n$
    is such that $\bm h(\bm y, \cdot)$ is weakly differentiable for
    all $\bm y$,
    if $\bm w$ is such that $\Var[\bm w^T\bm x_i]=\bm w^T\bm \Sigma \bm w=1$
    and $(\bm X\bm w,\bm y)$ is independent of
    $\bm X\bm\Sigma^{-1}(\bm I_p - \bm \Sigma \bm w \bm w^T)$
    as in the single index model \eqref{single-index},
    then for
    $\bm P =\bm I_p - \bm\Sigma \bm w \bm w^T$
    we have
    \begin{align}
        \label{SOS_all_j}
    \frac 1 2 \sum_{j=1}^p \E\Bigl[
    \Bigl|\bm e_j^T \bm P^T\bm\Sigma^{-1}\bm X^T\bm h(\bm y, \bm X)
    - \sum_{i=1}^n\sum_{k=1}^p P_{kj} \frac{\partial h_i}{\partial x_{ik}}(\bm y, \bm X)
    - Z_j \Omega_{jj}^{1/2} \|\bm h(\bm y, \bm X)\|
    \Big|^2
    \Bigr]
    \\
    \le 15
    \E
    \sum_{i=1}^n
    \Big\|
    \sum_{k=1}^p
    \frac{\partial \bm h}{\partial x_{ik}}(\bm y, \bm X)
    \bm e_k^T \bm P
    \Big\|_F^2
    + \sum_{j=1}^p \E\Bigl[
    Z_j^2
    \|\bm h(\bm y, \bm X)\|^2
    \Bigr]
    w_j^2
    \label{eq:last-line-cor_SOS}
    \end{align}
    where $\Omega_{jj} = \bm e_j^T \bm \Sigma^{-1}\bm e_j$
    and where $\E[\cdot]$ is either the conditional expectation
    given $(\bm y,\bm X\bm w)$ or the unconditional expectation.
\end{corollary}
\begin{proof}
    Without loss of generality, assume that $\bm a^T \bm \Sigma^{-1}\bm a = 1$.
    Following the notation and conditioning technique
    in \cite{bellec_zhang2019debiasing_adjust,bellec_zhang2019second_poincare},
    $\bm z = \bm X\bm \Sigma^{-1} \bm a$ is independent of 
    $\bm X(\bm I_p - \bm \Sigma^{-1} \bm a \bm a^T)$
    so that the conditional distribution of $\bm z$
    given 
    $\bm X(\bm I_p - \bm \Sigma^{-1} \bm a \bm a^T)$
    is $N(\bm 0, \bm I_n)$.
    The proof is completed by application
    of \Cref{lemma-SOS} to $\bm z$ conditionally
    on $\bm X(\bm I_p - \bm \Sigma^{-1} \bm a \bm a^T)$
    with $\bm f(\bm z) 
    =\bm h(\bm z \bm a^T +
    \bm X(\bm I_p - \bm \Sigma^{-1} \bm a \bm a^T))$.

    (ii) By application of the first part of the theorem
    to $\bm a = \bm a^{(j)} = \bm P \bm e_j$ for each $j\in[p]$ we obtain
    $$
    \E\sum_{j=1}^p\Bigl[
    \Bigl|\bm a^{(j)}{}^T\bm\Sigma^{-1}\bm X^T\bm h
    - 
    \Bigl(\sum_{i=1}^n\sum_{k=1}^p a^{(j)}_k \frac{\partial h_i}{\partial x_{ik}}
    \Bigr)
    - \tilde\Omega_{jj}^{1/2} Z_j  \|\bm h\|
    \Big|^2
    \Bigr]
    \le
    15\E
    \sum_{j=1}^p\sum_{i=1}^n
    \Big\|
    \sum_{k=1}^p a^{(j)}_k \frac{\partial \bm h}{\partial x_{ik}}
    \Big\|^2
    $$
    where $\tilde \Omega_{jj}^{1/2} = \|\bm\Sigma^{-1/2}\bm a^{(j)}\|$.
    Next, using $\frac 1 2 (u+v)\le u^2 + v^2$ we find
    \begin{align*}
        \eqref{SOS_all_j}&\le 15\E
    \sum_{j=1}^p
    \sum_{i=1}^n
    \Big\|
    \sum_{k=1}^p a^{(j)}_k \frac{\partial \bm h}{\partial x_{ik}}
    \Big\|^2
    + \sum_{j=1}^p \E\Bigl[
    Z_j^2
    \|\bm h\|^2
    \Bigr]
    \Big|\tilde\Omega_{jj}^{1/2} - \Omega_{jj}^{1/2}\Big|^2
    \end{align*}
    where we omit the arguments of $\bm h$ and its derivatives for brevity.
    Here the second term 
    appears due to $\Omega_{jj}^{1/2}$ in the term
    $Z_j\|\bm h\|$ in \eqref{SOS_all_j} instead of $\tilde\Omega_{jj}^{1/2}$.
    To bound from above the right-hand side of the previous
    display by \eqref{eq:last-line-cor_SOS}, we use
    for the first term
    $a_k^{(j)} = \bm e_k^T \bm a^{(j)} = \bm e_k^T \bm P \bm e_j$,
    while for the second term
    $$|\tilde\Omega_{jj}^{1/2} - \Omega_{jj}^{1/2}|
    =
    \big|
    \|\bm \Sigma^{-1/2}\bm P \bm e_j \| - \|\bm \Sigma^{-1/2} \bm e_j\|
    \big|
    \le
    \|\bm \Sigma^{-1/2}(\bm I_p - \bm P)\bm e_j\|
    $$
    by the triangle inequality and
    $\bm\Sigma^{-1/2}(\bm I_p - \bm P)\bm e_j = \bm \Sigma^{1/2} \bm w w_j$
    has squared euclidean norm equal to $w_j^2$ thanks to
    $\|\bm \Sigma^{1/2} \bm w\|=1$.
\end{proof}

\section{Proofs of Section~\ref{sec:main-strongly-convex}}
\label{sec:proofs}

\subsection{Proofs:
    Approximate normality
    and proximal representation
    for $\hat\beta_j$
}

\thmConfidence*

\begin{proof}[Proof of \Cref{thm:confidence-intervals-Sigma}]
By the product rule and \eqref{eq:derivatives-psi-hbeta},
$$
\frac{\partial}{\partial x_{ik} } \Bigl[ \frac{\bm{\hpsi} }{\|\bm \hpsi\| }\Bigr]
= \|\bm{\hpsi}\|^{-1} \Bigl(\bm{I}_n - \frac{\bm \hpsi\bm \hpsi^T}{\|\bm \hpsi\|^2}\Bigr) \frac{\partial}{\partial x_{ik} }\bm \hpsi
= \|\bm{\hpsi}\|^{-1} \Bigl(\bm{I}_n - \frac{\bm \hpsi\bm \hpsi^T}{\|\bm \hpsi\|^2}\Bigr)
\Bigl[
    -\bm D \bm X \bm{\hat{A}} \bm e_k \hpsi_i
    -\bm V \bm e_i \hbeta_k
\Bigr]
.$$
Let $\bm P = \bm I_p -\bm\Sigma \bm w \bm w^T$ as in \Cref{cor_SOS}.
Define for each $j=1,...,p$,
$$
\Rem_j \defas \frac{\trace[\bm V]\bm e_j^T\bm P^T\bm\hbeta}{{\|\bm\hpsi\|}}
+
\sum_{i=1}^n\sum_{k=1}^p P_{kj}
\frac{\partial}{\partial x_{ik}}
\Bigl(
\frac{\hpsi_i}{\|\bm \hpsi\|}
\Bigr)
=
\frac{
    (\bm e_j^T\bm P^T\bm\hbeta)  ~
\bm \hpsi^T \bm V \bm \hpsi}{
\|\bm \hpsi\|^3}.
$$
Notice that
$\sum_{j=1}^p\Rem_j^2 \le \|\bm P^T\bm\hbeta\|^2 \|\bm V\|_{op}^2 / \|\bm \hpsi\|^2$ which will be used below to show that $\Rem_j$ is negligible.
We apply the second part of \Cref{cor_SOS} to $\bm h(\bm X)=\bm \hpsi/ \|\bm \hpsi\|$.
With the notation
$$
U_j \defas 
\frac{\bm e_j^T\bm P^T(\bm \Sigma^{-1}\bm X^T\bm \hpsi
+ \bm \hbeta \trace[\bm V]
)
}{\Omega_{jj}^{1/2}\|\bm \hpsi\|}
- Z_j 
$$
for each $j=1,...,p$ for brevity where $Z_j$ is given by
\Cref{cor_SOS}, we find
\begin{align*}
    \sum_{j=1}^p \frac{\Omega_{jj}}{2}
    \E\Bigl[\Big(
    U_j
    - \frac{\Rem_j}{\Omega_{jj}^{1/2}}
    \Big)^2\Bigr]
&\le15\E
\sum_{i=1}^n
\frac{1}{\|\bm{\hpsi}\|^{2}}
\Big\|\Bigl(\bm{I}_n - \frac{\bm \hpsi\bm \hpsi^T}{\|\bm \hpsi\|^2}\Bigr)\sum_{k=1}^p\frac{\partial\bm \hpsi}{\partial x_{ik}}
\bm e_k^T\bm P
\Big\|_F^2
+ \|\bm w\|^2
  &\quad \text{\small by \eqref{SOS_all_j}-\eqref{eq:last-line-cor_SOS}}
\\&=
15\E
\sum_{i=1}^n
\frac{1}{\|\bm{\hpsi}\|^{2}}
\Big\|\Bigl(\bm{I}_n - \frac{\bm \hpsi\bm \hpsi^T}{\|\bm \hpsi\|^2}\Bigr)
\Bigl(\bm D \bm X \bm{\hat{A}} \hpsi_i
+ \bm V\bm e_i \bm \hbeta^T
\Bigr)
\bm P
\Big\|_F^2
+ \|\bm w\|^2
  &\quad \text{\small using \eqref{eq:derivatives-psi-hbeta}}
\\&\le
30\E\bigl[
\|
\bm D \bm X \bm{\hat{A}}\bm P\|_F^2
+
30 \|\bm V\|_F^2\|\bm P^T\bm \hbeta\|^2\big/ \|\bm{\hpsi}\|^{2}
\bigr]
+ \|\bm w\|^2
,
\end{align*}
thanks to $(a+b)^2\le 2a^2 + 2b^2$
and $\sum_{i=1}^n \|\bm{M}\bm{e}_i\|^2 = \|\bm{M}\|_F^2$
for the last inequality.
We further use $\min_{k=1,...,p}\Omega_{kk}\le \Omega_{jj}$ and
$\Omega_{jj}U_j^2/2 \le\Omega_{jj} (U_j-\frac{\Rem_j}{\Omega_{jj}^{1/2}})^2 + \Rem_j^2$ to lower bound the first line,
so that
\begin{align}
\label{Rem_star}
    \Rem_*&\defas
    \sum_{j=1}^p
\E\Bigl[\Big(
\frac{\bm e_j^T\bm P^T(\bm \Sigma^{-1}\bm X^T\bm \hpsi
+ \bm \hbeta \trace[\bm V]
)
}{{\Omega_{jj}^{1/2}}\|\bm \hpsi\|}
- Z_j
\Big)^2\Bigr]
=
    \sum_{j=1}^p
\E\Bigl[
        U_j^2
\Bigr]
\\&\le
\C
\bigl(
\min_{j=1,...,p}\Omega_{jj}
\bigr)^{-1}
\E\bigl[
\|
\bm D \bm X \bm{\hat{A}}\bm P\|_F^2
+
 \|\bm V\|_F^2\|\bm P^T\bm \hbeta\|^2\big/ \|\bm{\hpsi}\|^{2}
+ \|\bm w\|^2
\bigr]
.
\label{Rem_star_RHS}
\end{align}
For the middle term, we use
$\|\bm P^T\bm\hbeta\|^2
\le \|\bm P^T\bm\Sigma^{-1/2}\|_{op}^2 \|\bm\Sigma^{1/2}\bm\hbeta\|^2$.
By \eqref{eqStrongConvex} and since $\bm{0} \in\argmin_{\bm{b}\in\R^p}g(\bm{b})$ we have have $\bm 0\in \partial g(\bm 0)$ and
$\bm{\hbeta}^T\bm{X}^T\bm{\hpsi} = n \bm{\hbeta}^T \partial g(\bm{\hbeta})
= n(\bm{\hbeta} - \bm{0})^T(\partial g(\bm{\hbeta}) - \bm{0})
\ge n \tau \|\bm{\Sigma}^{1/2}\bm{\hbeta}\|^2$
so that 
\begin{equation}
n \|\bm{\Sigma}^{1/2}\bm{\hbeta}\|^2 /\|\bm{\hpsi}\|^2
\le \| n ^{-1/2} \bm{X}\bm\Sigma^{-1/2}\|_{op}^2 / \tau^2.
\label{eq:upper-bound-norm-beta}
\end{equation}
For the first term in \eqref{Rem_star_RHS}, we use
$\|\bm D\bm X \bm{\hat{A}}\bm P\|_F^2
\le
\|\bm D\bm X \bm{\hat{A}}\bm\Sigma^{1/2}\|_F^2
\|\bm\Sigma^{-1/2}\bm P\|_{op}^2$.
Next, the matrices inside the Frobenius norms
$\|\bm D \bm X \bm{\hat{A}}\bm\Sigma^{-1/2}\|_F$
and
$\|\bm V\|_F$
have rank at most $n$. We use $\|\cdot\|_F^2 \le \text{rank}(\cdot)\|\cdot\|_{op}^2$
and we bound the operator norms using
$\|\bm{\Sigma}^{1/2}\bm{\hat{A}}\bm \Sigma^{1/2}\|_{op} \le (n\tau)^{-1}$ by \eqref{eq:boundA},
$\|\bm{D}\|_{op} \le 1$ since $\ell_{y_0}'$ is assumed
1-Lipschitz, and
$\|\bm V\|_{op}\le \|\bm D\|_{op} + \|\bm D\bm X\bm\Sigma^{-1/2}\|_{op}^2
(n\tau)^{-1}$
by definition of $\bm V$ in \eqref{eq:def-V-df-bound-ideal-case}.
This implies that \eqref{Rem_star_RHS} is bounded from above by
$$
\C(\min_{j=1,...,p}\Omega_{jj})^{-1}
\Bigl(
     \E\Bigl[\|\frac{\bm X\bm\Sigma^{-1/2}}{n^{1/2}}\|_{op}^2\Bigr]
     \frac{\|\bm \Sigma^{-1/2}\bm P\|_{op}^2}{\tau^2}
+
\E\Bigl[1\vee \|\frac{ \bm X\bm\Sigma^{-1/2} }{n^{1/2}} \|_{op}^6\Bigr]
\frac{\|\bm P^T\bm \Sigma^{-1/2}\|_{op}^2}{\tau^4}
+ 
\|\bm w\|^2
\Bigr).
$$
We then use $\max_{j=1,...,p} \Omega_{jj}^{-1} \le \|\bm \Sigma\|_{op}$
as well as 
$\|\bm w\|^2 \le \|\bm \Sigma^{-1}\|_{op}$ thanks to $\|\bm \Sigma^{1/2}\bm w\|=1$.
Furthermore, 
$
    \|\bm{\Sigma}^{-1/2}\bm{P}\|_{op}^2
    = \|\bm{\Sigma}^{-1/2}\bm{P}\bm \Sigma^{1/2} \bm \Sigma^{-1/2}\|_{op}^2
    \le \|\bm{\Sigma}^{-1/2}\|_{op}^2$
    since $\bm{\Sigma}^{-1/2}\bm{P}\bm \Sigma^{1/2}$ is an orthogonal projection.
    Combined with
    $\E[\|n^{-1/2} \bm X\bm\Sigma^{-1/2} \|_{op}^6] \le \C(\delta)$
    due to $p/n\le \delta^{-1}$ (cf. \eqref{eq:7-norm-G} below)
    we have proved that
    $$
    \Rem_*
    =
    \eqref{Rem_star}
    =
    \E\sum_{j=1}^p
    \Bigl(
        \frac{\sqrt n}{\Omega_{jj}^{1/2}}
        \Bigl(
            \frac{\hat v}{\hat r}\debias_j 
        -
        \frac{t_*}{\hat r}
        {w_j}
        \Bigr)
    - Z_j
    \Bigr)^2
    \le
    \C(\delta) \kappa \Bigl(\frac{1}{\tau^2} + \frac{1}{\tau^4} + 1\Bigr).
    $$
    Finally, using $(a+b)^2\le 2 a^2 + 2 b^2$ we have
    with $\hat t = \max(0,\hat t^2)^{1/2}$
    \begin{equation*}
    \E\sum_{j=1}^p
    \Bigl(
        \frac{\sqrt n}{\Omega_{jj}^{1/2}}
        \Bigl(
            \frac{\hat v}{\hat r}\debias_j 
            -
            \frac{\pm \hat t}{\hat r}
            {w_j}
        \Bigr)
    - Z_j
    \Bigr)^2
    \le
    2\Rem_*
    +
    2\E\Bigl[\frac{n|\pm\hat t-t_*|^2}{\hat r^2}\Bigr]
    \sum_{j=1}^p \frac{w_j^2}{\Omega_{jj}}.
    \end{equation*}
    Since $\Var[\bm{x}_i^T\bm w] = 1$
    and $\max_{j\in[p]}(\Omega_{jj}^{-1}) \le \|\bm{\Sigma}\|_{op}$
    we find
    $\sum_{j=1}^p \frac{w_j^2}{\Omega_{jj}}
    \le
    \|\bm{\Sigma}\|_{op} \|\bm w\|^2
    \le \kappa$.
    Since $\pm$ is the sign of $t_*$,
    $|t_*|=\pm t_*$,
    the basic inequality
    $$
    |t_* -  \pm\hat t|^2
    =
    |\pm t_* - \hat t|^2
    =\big| |t_*| - |\hat t| \big|^2
    \le |t_*^2 - \hat t^2|$$
    and inequality
    $\E[ n|t_*^2 - \hat t^2| / \hat r^2]
    \le \C(\delta,\tau) \sqrt p
    $ from \Cref{thm:adjustments-approximation}
    completes the proof 
    of \eqref{eq:thmconfidence-eq1}.
    To prove \eqref{eq:second_conclusion_strongly_convex_with_event_E}
    from \eqref{eq:thmconfidence-eq1},
    we use the event and the lower bound on $\frac{|\hat v|}{\hat r^2}$
    established in
    \eqref{eq:bound-r-v-strongly-convex} below.
\end{proof}

\subsection{Rotational invariance, change of variable}
\label{sec:change-of-variable}

In the next proofs, the following change of variable will be useful
to transform the correlated design problem to an isotropic one
such that the index is concentrated on the first component.
With this in mind,
    let $\bm{Q}\in O(p)$ be any rotation such that $\bm{Q}^T\bm{Q} = \bm{Q}\bm{Q}^T=\bm{I}_p$ 
    and $\bm\theta^*\defas \bm{Q} \bm \Sigma^{1/2} \bm w = \bm e_1$ is the first canonical basis
    vector in $\R^p$. Define
    \begin{equation}
    \bm{G}=\bm{X}\bm{\Sigma}^{-1/2}\bm{Q}^T, 
    \quad
    \bm{\htheta}(\bm{y},\bm{G}) 
    = 
    \bm{Q}\bm{\Sigma}^{1/2}\bm{\bm{\hbeta}}
    =
    \argmin_{\bm{\theta}\in\R^p}
    \frac 1 n \sum_{i=1}^n \ell_{y_i}(\bm{g}_i^T\bm{\theta}) + h(\bm{\theta})
    \label{eq:change-of-variable-G}
    \end{equation}
    where 
    $h(\bm{\theta}) = g(\bm{\Sigma}^{-1/2}\bm{Q}^T\bm{\theta})$ is convex
    and
    where $\bm{g}_i = \bm{G}^T\bm{e}_i$ for each $i=1,...,n$ are the rows of $\bm{G}$.
    Then
    $\bm{G}$ has iid $N(0,1)$ entries, 
    $\bm{X}\bm{\hbeta} = \bm{G}\bm{\htheta}$, $h(\bm{\htheta}) = g(\bm{\hbeta})$,
    $(\bm{\theta}-\bm{\tilde \theta})^T(\partial h(\bm{\theta})-\partial h(\bm{\tilde \theta}))
    \ge \tau \|\bm{\theta}-\bm{\tilde \theta}\|^2$
    for all $\bm{\theta},\bm{\tilde\theta}\in\R^p$ thanks to \eqref{eqStrongConvex}.
    Since $\bm{Q} \bm \Sigma^{1/2} \bm w= \bm{e}_1$
    is the first canonical basis vector in $\R^p$, the matrix
    $\bm{G}(\bm{I}_p - \bm{e}_1 \bm{e}_1^T)$ is independent of $\bm{G} \bm{e}_1 = \bm{G}\bm{\theta^*} = \bm{X}\bm{w}$
    and thus $\bm{G}(\bm{I}_p - \bm{e}_1 \bm{e}_1^T)$ is independent of $\bm{y}$.
    Furthermore, by the chain rule we can deduce
    the derivatives of $\bm{\htheta}$ with respect to the entries of $\bm{G}$
    for a fixed $\bm{y}$ from the derivatives \eqref{eq:derivatives-psi-hbeta}
    of $\bm{\hbeta},\bm\hpsi$ with respect to the entries of $\bm{X}$:
    \begin{equation}
    \frac{\partial\bm{\htheta}}{ \partial g_{ij} }
    = \bm{A} \bm{e}_j \hpsi_i
    - \bm{A} \bm{G}^T \bm{D} \bm{e}_i \htheta_j,
    \qquad
    \frac{\partial \bm{\hpsi}}{\partial g_{ij}}
    =
    - \bm{D} \bm{G} \bm{A} \bm{e}_j \hpsi_i
    - \bm{V} \bm{e}_i \htheta_j
    \label{derivatives-htheta}
    \end{equation}
    where $\bm{A} =\bm{Q} \bm{\Sigma}^{1/2}\bm{\hat{A}}\bm \Sigma^{1/2}\bm{Q}^T$,
    while the quantities
    \begin{equation}
        \begin{aligned}
            \bm{D}&=\diag(\ell_{\bm{y}}''(\bm{X}\bm{\hbeta})) &&= \diag(\ell_{\bm{y}}''(\bm{G}\bm{\htheta})),\\
            \bm{V}&=\bm{D} - \bm{D}\bm{X}\bm{\hat{A}} \bm{X}^T \bm{D} &&= \bm{D} - \bm{D} \bm{G}\bm{A} \bm{G}^T \bm{D},\\
            \df &= \trace[\bm X \bm{\hat{A}}\bm X^T \bm D]
                &&= \trace[\bm G \bm{A}\bm G^T \bm D],\\
            \bm{\hpsi} &= - \ell_{\bm{y}}'(\bm{X}\bm{\hbeta}) &&= -\ell_{\bm{y}}'(\bm{G}\bm{\htheta})
        \end{aligned}
    \end{equation}
    are unmodified by the change of variable.
    The bound \eqref{eq:boundA} then reads
    $\|\bm A\|_{op} \le 1/(n\tau)$.

    Finally, let us rewrite $(\hat t^2,\hat a^2, \hat\sigma^2)$
    in \eqref{hat_v_r_t} after the change of variable:
\begin{equation}
    \label{eq:rewrite-after-change-variable}
    \left\{
    \begin{aligned}
        \hat t^2 &=
        \tfrac{1}{n^2}\|\bm G^T\bm \hpsi\|^2
        +  \tfrac{2 \hat v}{n} \bm{\hpsi}^T\bm G \bm{\htheta}
        + \tfrac{\hat v^2}{n}
        \|\bm G \bm{\htheta} - \hat\gamma \bm \hpsi\|^2
        - \tfrac{p}{n}\hat r^2
               \\&=
        \tfrac{1}{n^2}\|\bm G^T\bm \hpsi
        + \trace[\bm V]\bm\htheta\|^2
        + \hat v^2
        (
            \tfrac 1 n \|\bm G \bm{\htheta} - \hat\gamma \bm \hpsi\|^2
            -
            \|\bm\htheta\|^2
        )
        - \tfrac{p}{n}\hat r^2
        ,
    \\
    \hat \aaa^2 &=
    \hat t^{-2}
    \big(
    \tfrac{\hat v}{n}\|\bm{G} \bm \htheta - \hat\gamma \bm \hpsi\|^2
    + \tfrac{1}{n}\bm{\hpsi}^\top \bm G \bm\htheta
    - \hat\gamma \hat r^2
    \big)^2
    ,
    \\
    \hat \sigma^2
    &=
    \tfrac{1}{n}\|\bm{G}\bm \htheta - \hat \gamma \bm \hpsi\|^2
    - \hat \aaa^2.
    \end{aligned}
    \right.
\end{equation}

\subsection{Proofs:
    $(\hat \gamma,\hat t^2,\hat \aaa^2,\hat\sigma^2)$ estimate
    $(\gamma_*,t_*^2,\aaa_*^2,\sigma_*^2)$
}
\label{sec:proof_Gammas}


Before reading the following arguments, we recommend to first go through
\Cref{prop:improved-ridge} and its short proof,
as well as to understand how the informal approximations
\eqref{main_Gamma_4}-\eqref{main_Gamma_5_star}
lead to \eqref{314}, \eqref{316} and \eqref{leading_up_to}.
The techniques in \Cref{prop:improved-ridge} are used
in a simpler and more restricted setting than the general setting of the
present section, but are still representative of the arguments below.

\subsubsection{Notation and deterministic preliminary}
    Consider the change of variable and notation
    defined in \Cref{sec:change-of-variable}.
    Next, define  $\bm P_1^\perp = \bm I_p - \bm e_1 \bm e_1^T
    = \bm I_p - \bm\theta^* \bm\theta^*{}^T$ as well as 
    \begin{equation}
    \tilde t ^2
               \defas
               \Big\|\frac{\bm{G}^T\bm \hpsi}{n} + \hat v \bm \htheta \Big\|^2
        -
        \frac p {n} \hat r^2
        \label{eq:tilde-t-proof}
    \end{equation}
    which will be proved to be close to $\hat t^2$ and
    $t_*^2= [\bm e_1^T(\bm G^T\bm \hpsi + \trace[\bm V]\bm\htheta)]^2/n^2$.
    Consider
    $\Gamma_1^*,\Gamma_2,\Gamma_3,\Gamma_5^*,\Gamma_6^*,\Gamma_7$ defined by
    \begin{equation}
        \left\{
        \begin{aligned}
        \myred
        +\hat v \frac{\sigma_*^2}{\hat r^2} 
        -  \gamma_*
        &= 
        \Gamma_1^*,
        \\
        \frac{-\tilde t^2+t_*^2}{\hat r^2}
        =
        \frac{p}{n}
        -
        \frac{1}{\hat r^2}\|\bm P_1^\perp(\frac{\bm G^T\bm \hpsi}{n} + \hat v \bm\htheta)\|^2
        &= \Gamma_2,
        \\
        \frac{\|\bm P_1^\perp \bm G^T\bm \hpsi\|^2 }{n^2 \hat r^2}
        + \hat v \myred
        + \hat v \hat\gamma
        - \frac{p}{n}
        &= \Gamma_3
        ,
        \\
        -\hat v^2 \frac{\sigma_*^2}{\hat r^2} 
        - \hat v \myred
        + \hat v \hat\gamma
        &=\Gamma_3 + \Gamma_2\\
        \mybro - \frac{\sigma_*^2}{\hat r^2}
        - \gamma_* \myred
        + \hat v \hat\gamma  \frac{\sigma_*^2}{\hat r^2}
        &=  \Gamma_5^*,
        \\
        \frac{\hat\theta_1(\bm{G}\bm e_1)^\top
        (
        \bm{G}\bm P_1^\perp\bm \htheta
        - \gamma_* \bm{\hpsi})
        }{n\hat r^2}
        &= 
        \Gamma_6^*,
        \\
        \hat v \frac{\hat\theta_1^2
        (\|\bm G\bm e_1\|^2 - n)
        }{n\hat r^2}
        & = \Gamma_7
        \end{aligned}
        \right.
        \label{Rem_i}
    \end{equation}
    The quantities $\Gamma_i$ and $\Gamma_i^*$ will be proved
    to be of order $n^{-1/2}$ in \Cref{lemma:fiveqs} below,
    so that each right-hand side is of negligible order.
    Before proving that the $\Gamma_i,\Gamma_i^*$ are at most of order $n^{-1/2}$
    under our working assumptions, which follows from techniques
    already developed in the linear model \cite{bellec2021derivatives},
    we explain 
    how well-chosen weighted combinations of the above quantities
    provide the desired relationships \eqref{eq:consistency-hat-gamma}-\eqref{eq:consistency-aaa}.
    We have
    \begin{equation}
    \hat v(\hat\gamma - \gamma_*)
    =
    \Gamma_1^*\hat v +\Gamma_2 + \Gamma_3.
    \label{algebra-bound-gamma_star}
    \end{equation}
    If the previous display is of negligible order,
    then $\hat v\hat\gamma \approx \hat v \gamma_*$:
    This means that after multiplication by $\hat v$,
    we can replace $\gamma_*$ by $\hat \gamma$ in
    \eqref{Rem_i} without significantly
    enlarging the right-hand sides
    in the three equations involving $\Gamma_1^*,\Gamma_5^*,\Gamma_6^*$.
    With this in mind, define $\Gamma_1,\Gamma_5,\Gamma_6$ by
    \begin{align*}
        \hat v\Bigl(    \myred
        + \hat v \frac{\sigma_*^2 }{\hat r^2}
        - \hat \gamma
        \Bigr)
                 &=
        \hat v \Gamma_1^*
        + (\gamma_* - \hat\gamma) \hat v 
                 &&\defas\Gamma_1
        ,
        \\
        \hat v \Bigl(\mybro - \frac{\sigma_*^2}{\hat r^2}
        - \hat \gamma \myred
        + \hat v \hat\gamma  \frac{\sigma_*^2}{\hat r^2}
        \Bigr)
        &=
         \hat v
         \Gamma_5^* 
         + (\gamma_*-\hat \gamma)\hat v\myred 
        &&\defas
        \Gamma_5 ,
        \\
        \hat v
        \frac{\hat\theta_1(\bm{G}\bm e_1)^\top
        (
        \bm{G}\bm P_1^\perp\bm \htheta
        - \hat\gamma \bm{\hpsi})
        }{n\hat r^2}
                 &=
        \hat v
        \Gamma_6^*
        +(\gamma _* - \hat\gamma)\hat v \frac{\hat\theta_1(\bm{G} \bm e_1)^T\bm \hpsi}{n \hat r^2}
                 &&\defas \Gamma_6.
    \end{align*}
    Here $\Gamma_1,\Gamma_5,\Gamma_6$ are the analogous of $\Gamma_1^*,\Gamma_5^*,\Gamma_6^*$ after multiplication by $\hat v$ and replacing $\gamma_*$ by $\hat \gamma$ in the left-hand side
    of \eqref{Rem_i}.
    Completing the square, we find using $\Gamma_5,\Gamma_1$
    and $\hat \gamma^2\hat r^2 = \|\hat \gamma\bm \hpsi\|^2/n$
    \begin{align*}
        \hat v
        \Bigl[
        \frac{\|\bm{G}\bm P_1^\perp \bm \htheta - \hat \gamma \bm \hpsi\|^2}{n\hat r^2}
        -
        \frac{\sigma_*^2}{\hat r^2}
        \Bigr]
        =&
        \Gamma_5 - \hat\gamma \Gamma_1.
    \end{align*}
    For the following, recall
    that $\aaa_*^2 = \hat\theta_1^2$,
    $\sigma_*^2 = \|\bm P_1^\perp \bm\htheta\|^2$
    and $\|\bm\htheta\|^2 = \aaa_*^2 + \sigma_*^2$.
    Expanding now the square with
    $\bm{G} \bm \htheta - \hat\gamma \bm \hpsi
    = (\bm{G} \bm P_1^\perp \bm \htheta
    - \hat\gamma \bm{\hpsi})
     + \bm{G} \bm e_1 \htheta_1$,
    \begin{align}
        \hat v
        \Bigl(
        \frac{\|\bm{G}\bm \htheta - \hat \gamma \bm \hpsi\|^2}{n \hat r^2}
        -
        \frac{\sigma_*^2+\aaa_*^2}{\hat r^2}
        \Bigr)
        =&
           (\Gamma_5 - \hat\gamma \Gamma_1)
       + \Gamma_7  
        +   2 \Gamma_6
       \defas \Gamma_8.
        \label{algebra-bound-squared_norm}
    \end{align}
    This will justify the approximation
    $
        \hat v
        \|\bm{G}\bm \htheta - \hat \gamma \bm \hpsi\|^2/n
        \approx
        \hat v(
        \sigma_*^2
        +\aaa_*^2
        )
        =\hat v\|\bm \htheta\|^2
    $ when the right-hand side of the previous display
    is of negligible order.

    %

    We now focus on $\tilde t^2$ in \eqref{eq:tilde-t-proof},
    $\hat t^2$ in \eqref{hat_v_r_t},
    and 
    $t_* = 
    \frac{\bm e_1^T(\bm G^T\bm\hpsi + \trace[\bm V]\bm\htheta)}{n}
    =
        \frac{(\bm{G} \bm e_1)^T \bm \hpsi}{n}
        + \hat v
        \aaa_*
    $
    as defined in \Cref{thm:confidence-intervals-Sigma}.
    We have by the second line in \eqref{eq:rewrite-after-change-variable}
    and the definition of $\Gamma_8$,
    \begin{align}
    \frac{\tilde t^2 - \hat t^2}{\hat r^2}
    =
    -
       \hat v \Gamma_8,
       \qquad\qquad\qquad
    \frac{t_*^2 - \hat t^2}{\hat r^2}
    =
    \Gamma_2
    -
       \hat v \Gamma_8.
    \label{algebra-bound-t}
    \end{align}
    %
    This will justify the approximation $\hat t^2\approx t_*^2$
    in \eqref{eq:consistency-hat-t}
    when the right-hand sides are negligible.
    Using $\bm{G} \bm \htheta =
    \bm{G} \bm{P}_1^\perp\bm \htheta  +  \bm G \bm e_1 \htheta_1$
    and $\aaa_* = \htheta_1$ we find
    $$
    \frac{\bm \hpsi^T\bm G\bm \htheta}{n} - t_* \aaa_*
    = \frac{
        \bm \hpsi^T \bm G \bm P_1^\perp \bm\htheta
        + \bm \hpsi^T \bm G \bm e_1 \htheta_1
    }{n}
    - \aaa_*\Bigl(
        \frac{\bm \hpsi^T (\bm G \bm e_1)}{n}
        + \hat v \aaa_*
    \Bigr)
    = \frac{\bm \hpsi^T\bm G\bm P_1^\perp\bm \htheta}{n} - \hat v \aaa_*^2
    $$ so that by definition of $\Gamma_1$,
    \begin{equation}
        \label{Gamma_9}
        \hat v\Bigl[
            \frac{\bm{\hpsi}^T \bm G \bm\htheta}{n\hat r^2}
            + \hat v\frac{\|\bm{G} \bm \htheta - \hat\gamma\bm \hpsi \|^2}{n\hat r^2}
            -\hat \gamma
            - \frac{\aaa_* t_*}{\hat r^2}
        \Bigr]
    =
    \hat v^2\Bigl[
    \frac{\|\bm{G} \bm \htheta - \hat\gamma\bm \hpsi \|^2}{n\hat r^2}
    -
    \frac{\aaa_*^2 +\sigma_*^2}{\hat r^2}
    \Bigr]
    +
    \Gamma_1 
    \defas \Gamma_9
    \end{equation}
    with $\Gamma_9 \defas \hat v\Gamma_8 + \Gamma_1$.
    If the right-hand side is small, this means that the
    approximation
    $\frac{\bm{\hpsi}^T \bm G \bm\htheta}{n\hat r^2}
    + \hat v\frac{\|\bm{G} \bm \htheta - \hat\gamma\bm \hpsi \|^2}{n\hat r^2}
    -\hat \gamma
    \approx
    \frac{\aaa_* t_*}{\hat r^2}$ holds and one can estimate the product
    $\frac{\aaa_* t_*}{\hat r^2}$ by the left-hand side of the approximation.

    To find an estimate  for $\aaa_*^2=\hat\theta_1^2$,
    let $W=\frac{\bm{\hpsi}^T \bm G \bm\htheta}{n\hat r^2}
    + \hat v\frac{\|\bm{G} \bm \htheta - \hat\gamma\bm \hpsi \|^2}{n\hat r^2}
    -\hat \gamma$ 
    so that $\hat v(W-\frac{\aaa_*t_*}{\hat r^2})=\Gamma_9$.
    Expanding the square $(\hat v W)^2 = (\hat v \frac{\aaa_* t_*}{\hat r^2} + \Gamma_9)^2$
    and noticing that $W^2 = \frac{\hat t^2 \hat \aaa^2}{\hat r^4}$
    by definition of $(\hat t^2,\hat\aaa^2,\hat r^2)$,
    we obtain
    \begin{align*}
        \hat v^2 \frac{\hat t^2 \hat \aaa^2}{\hat r^4}  &=
      \hat v^2 \frac{\aaa_*^2 t_*^2}{\hat r^4}
                                       + 2 \hat v \frac{\aaa_* t_*}{\hat r^2} \Gamma_9 
      + \Gamma_9^2
    \end{align*}
    so that by
    subtracting $\hat v^2\frac{\aaa_*^2\hat t^2}{\hat r^4}$ on both sides ,
    \begin{align}
      \nonumber
    \hat v^2  \frac{\hat t^2(\hat\aaa^2 - \aaa_*^2)}{\hat r^4}
                                       &=
    \hat v^2 \frac{\aaa_*^2(t_*^2 - \hat t^2)}{\hat r^4}
    + 2 
    \hat v \frac{\aaa_* t_*}{\hat r^2}\Gamma_9 
    +
    \Gamma_9^2
                                       \\&=
    \hat v^2\frac{\aaa_*^2}{\hat r^2}(\Gamma_2 - \hat v\Gamma_8)
    + 2 
    \hat v \frac{\aaa_* t_*}{\hat r^2}\Gamma_9 
    +
    \Gamma_9^2
    .
    \label{algebra-bound-a}
    \end{align}
    This will justify
    the approximation $\hat \aaa^2\approx \aaa_*^2$ when $\Gamma_2,\Gamma_8,\Gamma_9$ have negligible order and
    $\hat v,\hat t^2,\hat r^2, \aaa^*$ all have constant order.


\subsubsection{All $\Gamma_i,\Gamma_i^*$ are of order at most $n^{-1/2}$}
\label{subsubsec:Gammas-sqrtn}

Controlling the terms $\Gamma_i$ relies on the two following probabilistic
propositions developed for analysing M-estimators in linear models. 

\begin{proposition}
    \label{prop26}
    [Theorem 7.2 in \cite{bellec2020out_of_sample}]
    Let $\bm{Z}\in\R^{K\times Q}$ be a matrix with iid $N(0,1)$ entries.
    Let $\bm{u}: \R^{K\times Q}\to \R^K$
    be weakly differentiable such that $\|\bm{u}(\bm Z)\|\le 1$ almost surely.
    Then,
    with $\bm{e}_q$ the $q$-th canonical basis vector in $\R^Q$,
    \begin{equation}
    \E
    \Big|
    Q \|\bm{u}(\bm Z)\|^2
    -
    \sum_{q=1}^Q
    \Bigl(\bm{e}_q^T \bm Z^T \bm u(\bm Z) 
    -
    \sum_{k=1}^K\tfrac{\partial u_k}{\partial z_{kq}}
    (\bm{Z})
    \Bigr)^2
    \Big|
    \le
    \C (\sqrt Q(1+\Xi^{1/2})
    + \Xi)
    \end{equation}
    where $\Xi = \E \sum_{k=1}^K\sum_{q=1}^Q
    \| \frac{\partial \bm{u}}{\partial z_{kq}}(\bm Z)\|^2$.
\end{proposition}

\begin{proposition}
    \label{prop25}
    [Proposition 6.5 in \cite{bellec2020out_of_sample}]
    Let $\bm{Z}\in\R^{K\times Q}$ be a matrix with iid $N(0,1)$ entries.
    Let $\bm{f}: \R^{K\times Q}\to \R^Q$,
    $\bm{u}:\R^{K\times Q}\to \R^K$
    be weakly differentiable. Then,
    omitting the dependence on $\bm{Z}$ in $\bm u(\bm Z),\bm f(\bm Z)$
    and their derivatives,
    \begin{equation*}
    \E\Bigl[
    \Bigl(
    \bm{u}^T \bm Z \bm f
    -
    \sum_{k=1}^K
    \sum_{q=1}^Q
    \frac{\partial\bm (f_q u_k)}{\partial z_{kq}}
    \Bigr)^2
    \Bigr]
    \le
    \E\Bigl[
    \|\bm{u}\|^2
    \|\bm{f}\|^2
    +
    \sum_{k=1}^K
    \sum_{q=1}^Q
    \Big\|
    \bm{f}
    \frac{\partial\bm{u}^T}{\partial z_{kq}}
    +
    \Bigl(
    \frac{\partial\bm{f}}{\partial z_{kq}}
    \Bigr)
    \bm{u}^T
    \Big\|_F^2
    \Bigr].
    \end{equation*}
\end{proposition}

\begin{restatable}{lemma}{lemmaFiveEq}
    \label{lemma:fiveqs}
    Let \Cref{assum,assumStrongConvex} be fulfilled. Then
$\E[\Gamma_1^*{}^2
+\Gamma_3^2
+ \Gamma_5^*{}^2
+ \Gamma_6^*{}^2
+ \Gamma_7^2
]\le \C(\delta,\tau)/n$
and $\E[|\Gamma_2|]
\le \C(\delta,\tau)/\sqrt n$.
\end{restatable}
Let us recall some bounds that will be useful throughout the proof:
\begin{align}
    \label{eq:bound-A-op-recap-before-proof}
    \|\bm A\|_{op} &\le (n\tau)^{-1}
    \qquad\qquad\qquad\qquad \text{[by \eqref{eq:boundA}]},
    \\
        \|\bm{V}\|_{op}=
        \|\bm D - \bm D \bm G \bm A \bm G \bm D\|_{op}
        &\le 1+\|n^{-1/2}\bm{G}\|_{op}^2/\tau
        \quad\text{[by \eqref{eq:bound-A-op-recap-before-proof} and def. of $\bm V$]},
        \label{eq:bound-V-op}
        \\
    \label{eq:7-v}
    |\hat v| = |\tfrac{\trace \bm V}{n}| &\le n (1+
            \|n^{-1/2}\bm{G}\|_{op}^2 / \tau
    )
    \quad\quad\text{[by \eqref{eq:bound-V-op} or \eqref{eq:bound-trV-D}]},\\
\tfrac{n\|\bm{\htheta}\|^2}{\|\bm{\hpsi}\|^2}
= \tfrac{\|\bm \htheta\|^2}{\hat r^2}
                          &
\le \|n^{-1/2}\bm{G}\|_{op}^2 / \tau^2
\quad\qquad
\quad\quad\text{[see \eqref{eq:upper-bound-norm-beta}]},
    \label{eq:bound-htheta}
\\
    \E[\|n^{-1/2 }\bm G \|_{op}^c] &\le \C(\delta,c)
    \label{eq:7-norm-G}
\end{align}
for any absolute constant $c\ge 1$. Here the last line follows,
for instance, from
\cite[Corollary 7.3.3]{vershynin2018high} or \cite[Theorem II.13]{DavidsonS01}.
As we explain next, the bound on $\E[|\Gamma_2|]$ follows from \Cref{prop26}
while the bounds $\E[\Gamma_1^*{}^2],\E[\Gamma_3^2],\E[\Gamma_5^*{}^2],
\E[\Gamma_6^*{}^2]$ are consequences of \Cref{prop25}. 

\begin{proof}[Proof of \Cref{lemma:fiveqs}]

\textbf{Proof of $\E[\Gamma_7^2]\le \C(\delta,\tau)/n$.}
By the Cauchy-Schwarz inequality, with $\|\bm G\bm e_1\|^2\sim \chi^2_n$ we have
$\E[\Gamma_7^2]
\le 
\E[(\chi^2_n-n)^4]^{1/2}\E[\hat v^4\hat \theta_1^8 / \hat r^{8}]^{1/2}/n^2
$.
Next,
$\E[(\chi^2_n-n)^4]^{1/4} \le \C \sqrt n$
by concentration properties of the $\chi^2_n$ distribution,
and $\E[\Gamma_7^2]\le \C(\delta,\tau)/n$ is obtained by combining
\eqref{eq:7-v}-\eqref{eq:7-norm-G}. \qed

\textbf{
        Proof of 
        $\E[|\Gamma_2|]\le \C(\delta,\tau)n^{-1/2}$.
}
    We apply
    \Cref{prop26}
    with respect to the Gaussian matrix $\bm{G}$
    with the first column removed. By construction,
    since $\bm{y}$ is independent of the submatrix of $\bm G$
    made of columns indexed in $\{2,...,p\}$,
    we are in a position to apply \Cref{prop26} 
    conditionally on $(\bm{y}, \bm G \bm e_1)$ where $\bm e_1$ is the first
    canonical basis vector in $\R^p$.
    With
    $\Xi = \E\sum_{j=2}^p\sum_{i=1}^n \|\frac{\partial \bm{u}}{\partial g_{ij}}\|^2$,
    the choice $\bm{u} = \bm{\hpsi}/\|\bm \hpsi\|$ in
    \Cref{prop26}
    yields
    \begin{equation}
        \E\Big|(p-1) - \sum_{j=2}^p
    \bigl(
    \bm{e}_j^T \bm{G}^T \bm{u} 
    - \sum_{i=1}^n \frac{\partial u_i}{\partial g_{ij}}
    \bigr)^2
    \Big|
    \le \C(1+\sqrt \Xi) \sqrt p + \C \Xi
    .
    \label{eq:chi-square-proba-result}    
    \end{equation}
    By \eqref{derivatives-htheta} we have
    $(\partial/\partial g_{ij})
    (\tfrac{\bm{\hpsi} }{\|\bm \hpsi\|})
    = (\bm{I}_n - \tfrac{\bm{\hpsi}\bm \hpsi^T}{\|\bm \hpsi\|^2})[
    - \bm{D}\bm{G}\bm{A}\bm{e}_j \hpsi_i
    - \bm{V} \bm{e}_i \htheta_j
    ]$ by the chain rule, hence
    \begin{equation}
        \sum_{i=1}^n \frac{\partial u_i}{\partial g_{ij}}
    =
        \sum_{i=1}^n \frac{\partial}{\partial g_{ij}}\Bigl(\frac{\hpsi_i}{\|\bm{\hpsi}\|}\Bigr)
        = - \frac{\trace[(\bm{I}_n - \bm{\hpsi}\bm \hpsi^T/\|\bm \hpsi\|^2) \bm{V}]}{\|\bm \hpsi\|}\htheta_j 
    \end{equation}
    in the left-hand side of \eqref{eq:chi-square-proba-result}.
    In the right-hand side of \eqref{eq:chi-square-proba-result},
    we bound $\Xi$ using
    \begin{align}
        \sum_{i=1}^n
        \sum_{j=2}^p
        \Big\|
        \frac{\partial}{\partial g_{ij}}
        \Bigl(
        \frac{\bm\hpsi}{\|\bm{\hpsi}\|}
        \Bigr)
        \Big\|^2
        &\le
        \frac{1}{\|\bm{\hpsi}\|^2}
        \sum_{i=1}^n
        \sum_{j=2}^p
        \Big\|
        \frac{\partial \bm\hpsi}{\partial g_{ij}}
        \Big\|^2
        \nonumber
      \\&\le
        2
        \sum_{i=1}^n
        \sum_{j=2}^p
        \Bigl(
        \frac{\|\bm{D}\bm{G}\bm{A }\bm{e}_j\|^2 \hpsi_i^2}{\|\bm{\hpsi}\|^2}
        +
        \frac{\|\bm{V} \bm{e}_i\|^2 \htheta_j^2}{\|\bm{\hpsi}\|^2}
        \Bigr)
        \nonumber
      \\&\le
        2 \bigl(
          \|\bm{D}\bm{G}\bm{A}\|_F^2
          + n \|\bm{V}\|_{op}^2 \|\bm{\htheta}\|^2/ \|\bm{\hpsi}\|^2
        \bigr).
        \label{eq:inequality-square-gradient-psi}
    \end{align}
    Using inequalities
    \eqref{eq:bound-V-op}-\eqref{eq:bound-htheta} we find
    $$
    \Xi \le 2 \E\bigl[n\|\bm{G}\|_{op}^2/(n\tau)^2\bigr]
    +
    2 \E\bigl[(1+\|n^{-1/2}\bm{G}\|_{op}^2/\tau)^2 \|n^{-1/2}\bm{G}\|_{op}^2/\tau^2\bigr]
    \le \C(\delta,\tau).
    $$
    Thanks to \eqref{eq:7-norm-G} this yields
    \begin{equation}
    \E\Big|
    (p-1)
    - \frac{\|\bm{G}^T\bm{\hpsi} + \trace[\bm{\tilde V}]\bm{\htheta}\|^2}{\|\bm \hpsi\|^2}
    +
    \frac{(\bm{e}_1^T\bm{G}^T\bm{\hpsi} + \trace[\bm{\tilde V}]\bm{e}_1^T\bm{\htheta})^2}{\|\bm \hpsi\|^2}
    \Big|
    \le \C(\delta,\tau) \sqrt p
    \label{conclusion-tilde-V}
    \end{equation}
    where $\bm{\tilde V} = (\bm{I}_n-\bm{\hpsi}\bm \hpsi^T/\|\bm \hpsi\|^2) \bm{V}$.
    Note that $p-1$ can be replaced by $p$ in the left-hand side
    by changing the right-hand side constant if necessary.
    It remains to show that we can replace $\trace[\bm{\tilde V}]$
    by $\trace[\bm{V}]$ in \eqref{conclusion-tilde-V}.
    Thanks to
    $\|\bm{a}\|^2 - \|\bm{b} \|^2
    = (\bm{a} - \bm{b} )^T(\bm{a}  + \bm{b} )$ we have
    with $\bm P_1^\perp = \bm{I}_p - \bm{e}_1 \bm{e}_1^T$
    and
    $\trace[\bm V - \bm{\tilde V}] = \bm \hpsi^T \bm{V} \bm \hpsi / \|\bm \hpsi\|^2$
    \begin{align*}
    \frac{
        \|\bm P_1^\perp(\bm{G}^T \bm{\hpsi} + \trace[\bm{V} ]\bm{\htheta}) \|^2
        -\| \bm P_1^\perp ( \bm{G}^T \bm{\hpsi} + \trace[\bm{\tilde V} ]\bm{\htheta} )\|^2
    }{\|\bm{\hpsi}\|^2}
    =
    \frac{
    \bm{\hpsi}^T\bm{V} \bm{\hpsi}}{\|\bm \hpsi\|^2}
    \frac{
        \bm{\htheta}^T \bm P_1^\perp
    (2\bm{G}^\top\bm{\hpsi} + (\trace[\bm{V} + \bm{\tilde V}] )\bm{\htheta})
}{
\|\bm{\hpsi}\|^2}
    \end{align*}
    which is smaller in absolute value
    than
    $2\|\bm{V} \|_{op}
    (
    \|\bm{G}\|_{op} \|\bm{\htheta}\| / \|\bm{\hpsi}\|
    + n
    \|\bm{\htheta}\|^2/ \|\bm{\hpsi}\|^2
    )$.
    Thanks to the bounds \eqref{eq:bound-V-op}-\eqref{eq:bound-htheta}
    and $\E[\|n^{-1/2}\bm{G} \|_{op}^6]\le \C(\delta)$,
    by the triangle inequality
    $\trace[\bm{\tilde V}]$ in \eqref{conclusion-tilde-V}
    can be replaced by $\trace[\bm{V}]$.
    We have thus established
    \begin{equation*}
    n\E\Big|\frac{\tilde t^2-t_*^2}{\hat r^2}\Big|
    =
    \E\Big|
    p
    - \frac{\|\bm{G}^T\bm{\hpsi} + \trace[\bm V ]\bm{\htheta}\|^2}{\|\bm \hpsi\|^2}
    +
    \frac{(\bm{e}_1^T\bm{G}^T\bm{\hpsi} + \trace[\bm V ]\bm{e}_1^T\bm{\htheta})^2}{\|\bm \hpsi\|^2}
    \Big|
    \le \C(\delta,\tau) \sqrt p
    \end{equation*}
    as well as $\E|\Gamma_2|\le \C(\delta,\tau) n^{-1/2}$.
    \qed

\textbf{
        Proof of 
        $\E[|\Gamma_6^*|^2]\le \C(\delta,\tau)/n$.
}
Let $K=n, Q=p-1$
and let $\bm{Z}\in\R^{K\times Q}$ be the matrix $\bm G$ with the first column removed.
Then $\bm{Z}$ is independent of $(\bm y, \bm G \bm e_1)$
and \Cref{prop25} is applicable conditionally on $(\bm{y}, \bm G \bm e_1)$.
Chose $\bm{f} = (\frac{\htheta_j}{\|\bm \hpsi\|})_{j=2,...,p}$
valued in $\R^{p-1}$
and $\bm{u} = \bm G \bm e_1$ valued in $\R^n$.
Here, $\bm{u}$ has zero derivatives with respect to $\bm Z$.
Using the derivatives in \eqref{derivatives-htheta},
\begin{align*}
\bm{u}^T \bm Z \bm f
    -
    \sum_{k=1}^K
    \sum_{q=1}^Q
    \frac{\partial\bm (f_q u_k)}{\partial z_{kq}}
&=
\frac{(\bm G \bm e_1)^T \bm G \bm P_1^\perp \bm\htheta}{\|\bm \hpsi\|}
-\sum_{i=1}^n
\sum_{j=2}^p
G_{i1}\Bigl(
    \frac{1}{\|\bm \hpsi\|}
\frac{\partial \htheta_j}{\partial g_{ij}}
+
\htheta_j 
\frac{\partial}{\partial g_{ij}}
\Bigl(
    \frac{1}{\|\bm \hpsi\|}
\Bigr)
\Bigr)
\\&=
\bigl[
\|\bm{\hpsi}\|^{-1}
    (\bm{G}\bm e_1)^T(\bm G \bm P_1^\perp \bm \htheta
-\trace[\bm{A}]\bm\hpsi)
\bigr]
+
\Rem_6'
+\Rem_6''
\end{align*}
where the square bracket on the last line equals 
$
\sqrt n \Gamma_6^* \tfrac{\hat r}{\hat\theta_1}  = 
$ and
\begin{align*}
\Rem_6'&=
\frac{(\bm{G}\bm e_1)^T
\trace[\bm{A}]\bm\hpsi}{\|\bm \hpsi\|}
-
\sum_{i=1}^n
\frac{G_{i1}}{\|\bm \hpsi\|}
\sum_{j=2}^p\frac{\partial\htheta_j}{\partial g_{ij}}
=
\frac{
    \bm{A}_{11}\bm (\bm G \bm e_1)^T\bm \hpsi
    + \bm\htheta^T \bm{P}_1^\perp \bm{A} \bm G^T \bm D \bm G\bm e_1
}{\|\bm{\hpsi}\|}
\\
\Rem_6''
&=- 
\|\bm{\hpsi}\|^{-3}\bigl[\bm \hpsi^T\bm D\bm G \bm{A} \bm P_1^\perp\bm \htheta
+
\bm{\hpsi}^T 
\bm{V}\bm G \bm e_1
\|\bm P_1^\perp \bm \htheta\|^2
\bigr] 
\end{align*}
thanks to
$\trace[\bm A]= \bm A_{11} + \sum_{j=2}^p \bm A_{jj}$ for $\Rem_6'$.
Here, $\Rem_6''$ comes from differentiation of $\|\bm \hpsi\|^{-1}$
thanks to
$\frac{\partial}{\partial g_{ij}}
(\|\bm{\hpsi}\|^{-1})
= - \|\bm\hpsi\|^{-3}\bm{\hpsi}^T\frac{\partial \bm\hpsi}{\partial g_{ij}}$.
Now, $\Rem_6'$ and $\Rem_6''$ both have
second moment bounded by $\C(\delta,\tau)$
thanks to \eqref{eq:bound-htheta}, \eqref{eq:bound-A-op-recap-before-proof}
and \eqref{eq:7-norm-G}.
In the right-hand side of \Cref{prop25},
$\|\bm{u}\|^2\| \bm f\|^2$ has expectation smaller than $\C(\delta,\tau)$
again thanks to
\eqref{eq:bound-htheta} and \eqref{eq:7-norm-G}.
The derivative term in the right-hand
side of \Cref{prop25} is bounded by
$\C(\delta,\tau)$ by explicitly
computing the derivatives
using \eqref{derivatives-htheta}
and using again
\eqref{eq:bound-A-op-recap-before-proof}-\eqref{eq:7-norm-G}.
\qed

Bounds on $\Gamma_1^*,\Gamma_3, \Gamma_5^*$
are obtained similarly by the following applications
of \Cref{prop25}. As the precise calculations using
\eqref{eq:bound-A-op-recap-before-proof}-\eqref{eq:7-norm-G}
follow the same arguments as for $\Gamma_2,\Gamma_6^*$ above,
we omit some details.

\textbf{
        Proof of 
        $\E[|\Gamma_1^*|^2]\le \C(\delta,\tau)/n$.
}
The bound on $\Gamma_1^*$ is obtained similarly using
\Cref{prop25}
with the same $\bm{Z}\in\R^{n\times (p-1)}$
(that is, $\bm Z$ is the matrix $\bm G$ with the first column removed),
this time with $\bm{f}$ valued in $\R^{p-1}$ with components
$f_j = \htheta_j/\|\bm{\hpsi}\|$ for each $j=2,...,p$
and $\bm{u}= \bm \hpsi/\|\bm \hpsi\|$.
The key algebra is that using the derivatives \eqref{derivatives-htheta},
\begin{align*}
    \sum_{i=1}^n
    \hpsi_i \sum_{j=2}^p \frac{\partial\htheta_j}{\partial g_{ij}}&= 
    \Bigl(\sum_{j=2}^p \bm A_{jj}\Bigr)\|\bm \hpsi\|^2
    - \bm\htheta^T\bm P_1^\perp\bm{A}\bm G^T\bm D\bm \hpsi
                                                                  &&=
    (\gamma_*-\bm A_{11}) n\hat r^2
    - \bm\htheta^T\bm P_1^\perp\bm{A}\bm G^T\bm D\bm \hpsi
    ,
    \\
    \sum_{j=2}^p
    \htheta_j \sum_{i=1}^n \frac{\partial\psi_i}{\partial g_{ij}}&=
-\bm \hpsi^T\bm D \bm G\bm{A}\bm P_1^\perp\bm\htheta
- \trace[\bm V] \|\bm P_1^\perp\bm\htheta\|^2
                                                                 &&=
-\bm \hpsi^T\bm D \bm G\bm{A}\bm P_1^\perp\bm\htheta
- n\hat v \sigma_*^2
,
\end{align*}
so that 
$\bm u^T \bm Z \bm f
- \sum_{ij} \frac{\partial}{\partial z_{ij}}(u_i f_j)$ appearing in the left-hand side of \Cref{prop25} equals
$$
\Big[
\myred
+\hat v \tfrac{\sigma_*^2}{\hat r^2} 
-
\gamma_*
\Big]
+
\Bigl[
    \bm A_{11}
    +
    \frac{
        \bm \hpsi^T \bm D \bm G \bm A \bm P_1^\perp \bm\htheta
        + \bm\htheta^T \bm P_1^\perp \bm A \bm G^T \bm D \bm \hpsi
    }{\|\bm \hpsi\|^2}
-
\sum_{i=1}^n\sum_{j=2}^p
    \hpsi_i\htheta_j
    \frac{\partial}{\partial g_{ij}}
        \frac{1}{\|\bm \hpsi\|^2}
\Bigr]
$$
where the first square bracket is exactly $\Gamma_1^*$
and the second moment of the second bracket is smaller than
$\C(\delta,\tau)/n$
using \eqref{eq:bound-A-op-recap-before-proof}-\eqref{eq:7-norm-G}.
Similarly to $\Gamma_6^*$, the derivative term in the right-hand
side of \Cref{prop25} is bounded by
$\C(\delta,\tau)$ by explicitly
computing the derivatives
using \eqref{derivatives-htheta}
and using again
\eqref{eq:bound-A-op-recap-before-proof}-\eqref{eq:7-norm-G}.

\textbf{
        Proof of 
        $\E[|\Gamma_3|^2]\le \C(\delta,\tau)/n$.
}
The bound on $\Gamma_3$ is obtained similarly using
\Cref{prop25}
with the same $\bm{Z}$,
this time with $\bm{f}$ valued in $\R^{p-1}$ with components
$f_j = \bm{e}_j^T \bm G^T \bm\hpsi/\|\bm \hpsi\|$ for each $j=2,...,p$
and $\bm{u}= n^{-1} \bm \hpsi/\|\bm \hpsi\|$. The key algebra is that
$\bm u^T \bm Z \bm f
- \sum_{ij} \frac{\partial}{\partial z_{ij}}(u_i f_j)$ appearing in the left-hand side of \Cref{prop25} is then equal to
\begin{multline*}
\Bigl[
\tfrac{\|\bm P_1^\perp \bm G^T\bm \hpsi\|^2 }{n^2 \hat r^2}
+ \hat v \myred
+ \tfrac{\df-p}{n}
\Bigr]
+
\Bigl[
\tfrac{1-(\bm G^T\bm D \bm G\bm A)_{11}}{n}
+\tfrac{\bm \htheta^T \bm P_1^\perp \bm G^T  \bm V \bm \hpsi
+ \bm \hpsi^T \bm D \bm G \bm A \bm P_1^\perp \bm G^T\bm \hpsi
}{n  \|\bm\hpsi\|^2}
\\    -
\sum_{i=1}^n\sum_{j=2}^p
    \frac{\hpsi_i \bm \hpsi^T \bm G \bm e_j}{n}
    \frac{\partial}{\partial g_{ij}}
        \frac{1}{\|\bm \hpsi\|^2}
\Bigr]
\end{multline*}
with $\df/n=\trace[\bm G^T\bm D \bm G \bm A]/n=\hat v \hat\gamma$ so that the first square bracket 
is exactly $\Gamma_3$ and the second bracket is negligible
using \eqref{eq:bound-A-op-recap-before-proof}-\eqref{eq:7-norm-G}.

\textbf{
        Proof of 
        $\E[|\Gamma_5^*|^2]\le \C(\delta,\tau)/n$.
}
The bound on $\Gamma_5^*$ is obtained similarly using
\Cref{prop25}
with the same $\bm{Z}$,
this time with $\bm{f}$ valued in $\R^{p-1}$ with components
$f_j = \htheta_j/\|\bm{\hpsi}\|$ for each $j=2,...,p$
and $\bm{u}= \bm G \bm P_1^\perp \bm \htheta/\|\bm \hpsi\|$.
The key algebra is that
$\bm u^T \bm Z \bm f
- \sum_{ij} \frac{\partial}{\partial z_{ij}}(u_i f_j)$ appearing in the left-hand side of \Cref{prop25} is then equal to
\begin{multline*}
\Bigl[
        \mybro 
        - \tfrac{\sigma_*^2}{\hat r^2}
        - \gamma_* \myred
        + \tfrac{\df}{n} \tfrac{\sigma_*^2}{\hat r^2}
\Bigr]
+
\Bigl[
    \tfrac{
        \bm\htheta^\top\bm P_1^\perp \bm A \bm G^T \bm D \bm G \bm P_1^\perp\bm\htheta
        - \bm \hpsi^T\bm G \bm P_1^\perp \bm A \bm P_1^\perp \bm\htheta
        -(\bm A \bm G^T \bm D \bm G)_{11}\|\bm P_1^\perp \bm \htheta\|^2
    }{\|\bm \hpsi\|^2}
    \\
    + \bm A_{11}\myred
    -
    \sum_{i=1}^n\sum_{j=2}^p
    \htheta_j \bm e_i^T\bm G \bm P_1^\perp \bm\htheta
    \frac{\partial}{\partial g_{ij}}
        \frac{1}{\|\bm \hpsi\|^2}
\Bigr].
\end{multline*}
The first bracket is exactly $\Gamma_5^*$ and the second bracket is negligible
using \eqref{eq:bound-A-op-recap-before-proof}-\eqref{eq:7-norm-G}.

\end{proof}

\thmAdjustments*

\begin{proof}[Proof of \Cref{thm:adjustments-approximation}]
    Using \eqref{eq:bound-A-op-recap-before-proof} and \eqref{eq:bound-V-op}-\eqref{eq:bound-htheta},
    we have almost surely
    \begin{equation}
    \label{all_bounded}
    \max \Bigl\{|\hat\gamma|,|\gamma_*|,|\hat v|,|\tfrac{\hat \theta_1^2}{\hat r^2}|,|\tfrac{\sigma_*^2}{\hat r^2}|,
        \tfrac{\|\bm G\bm\htheta - \hat\gamma\bm\hpsi\|^2}{n\hat r^2},
    |\mybro|,
    \tfrac{t_*^2+\hat t^2}{\hat r^2}
    \Bigr\}
    \le \C(\delta,\tau) (1+\|\tfrac{\bm G}{\sqrt n}\|_{op}^2)^c
    \end{equation}
    for some numerical constant $c\ge 1$.
    The event
    $\bar E = \{\|n^{-1/2}\bm G\|_{op} \le 2+\delta^{-1/2} \}$
    has exponentially large probability,
    $\P(\bar E^c) \le e^{-n/2}$, by
    \cite[Theorem II.13]{DavidsonS01}.
    Let
    $$
    \Rem =
        \big|\hat v(\hat \gamma - \gamma_*)\big|
        +
         \tfrac{1}{\hat r^2} \big|\hat t^2 - t_*^2\big| 
         +
        \tfrac{|\hat v|}{\hat r^2}
        \big| \tfrac 1 n \|\bm{X}\bm \hbeta - \hat\gamma \bm \hpsi\|^2  
        - \|\bm\htheta\|^2
        \big|
        +
            \tfrac{\hat v^2  \hat t^2  }{\hat r^4}
            \big|
                \hat \aaa^2 -  \aaa_*^2
            \big|.
    $$
    Using \eqref{algebra-bound-gamma_star},
    \eqref{algebra-bound-squared_norm},
    \eqref{algebra-bound-t} and
    \eqref{algebra-bound-a}
    for the first term and the Cauchy-Schwarz inequality for the second,
    we find
    \begin{align*}
        \E\Rem
        \le
        \E[I_{\bar E}\Rem]
        +
        \E[I_{\bar E^c}\Rem]
        \le
        \C(\gamma,\tau)\E|\tilde \Gamma|
        +
        \P(\bar E^c)^{1/2} \E[\Rem^2]^{1/2}
    \end{align*}
    where $\tilde \Gamma = \max\{|\Gamma_1^*|,|\Gamma_2|,|\Gamma_3|,|\Gamma_5^*|,|\Gamma_6^*|,|\Gamma_7|\}$.
    \Cref{lemma:fiveqs} shows that
    $\E|\tilde \Gamma|\le \C(\delta,\tau) n^{-1/2}$.
    By \eqref{all_bounded} and \eqref{eq:7-norm-G} we have
    $\E\Rem^2\le\C(\delta,\tau)$ so that the exponential small
    probability of $\bar E^c$ completes the proof
    of \eqref{eq:consistency-hat-gamma}-\eqref{eq:consistency-aaa}.

    Under the additional \Cref{assumExtra}, we have $\hat r^2 = \|\bm \hpsi\|^2/n \le 1$
    hence by \eqref{eq:bound-htheta} and in $\bar E$,
    $\|\bm \htheta\|^2\le \C(\delta,\tau)$ and
    $\|\bm G\bm \htheta\|^2/n\le \C(\delta,\tau)$.
    Hence with $u_i = \bm e_i^T \bm G \bm \hbeta$, there exists
    at least $n/2$ indices $i\in[n]$ such that
    $|u_i| \le K$ for some constant $K=\C(\delta,\tau)$.
    By \Cref{assumExtra}, continuity and compactness,
    there exists deterministic constants $c_*, m_*>0$
    depending only on $K$ and the loss $\ell$ such that
    $\inf_{y_0\in\mathcal Y}\min_{u\in\R: |u|\le K} \ell_{y_0}''(u) \ge c_*$
    and
    $\inf_{y_0\in\mathcal Y}\min_{u\in\R: |u|\le K} \ell_{y_0}'(u)^2 \ge m_*$.
    Since at least $n/2$ components $u_i$ are such that $|u_i|\le K$,
    this implies 
    $\hat r^2 = \|\bm \hpsi\|^2/n \ge m_*/2$
    and
    $\trace[\bm D] = \sum_{i=1}^n \ell_{y_i}''(u_i)
    \ge (n/2) c_*$.
    The lower bound in \eqref{eq:bound-trV-D} then yields
    $n\hat v = \trace[\bm V]
    \ge \C(\delta,\tau)\trace[\bm D] - \C(\delta,\tau)
    \ge n \C(\delta,\tau,c_*)=n \C(\delta,\tau,\ell)$
    for $n\ge \C(\delta,\tau,\ell)$.
    We have proved that
    \begin{equation}
        \text{in the event }
        \bar E = \{\|n^{-1/2}\bm G\|_{op} \le 2+\delta^{-1/2} \}
        ,~
        \max\{
            \hat r^2, \tfrac{1}{\hat r^2},
            |\hat v|, |\tfrac{1}{\hat v}|
        \}\le \C(\delta,\tau,\ell)
        \label{eq:bound-r-v-strongly-convex}
    \end{equation}
    and the proof of \eqref{thm_adjustments_additional} is complete.
\end{proof}
\subsection{Proofs:
    Proximal representation
    for predicted values}
\thmProxLoss*

\begin{proof}[Proof of \Cref{thm:loss-proximal-representation}]
    Consider the change of variable and notation
    defined in \Cref{sec:change-of-variable}.
    Then
    \eqref{eq:conclusion-proximal-loss}
    holds if and only if
    $$\E[
    \hat r^{-2}
    (\bm{g}_i^T \bm \htheta
    -\prox[\trace[\bm{A}] \ell_{y_i}(\cdot)](\aaa_* U_i + \sigma_* Z_i)
    )^2
    ]
    \le \C(\delta,\tau,\kappa)/n$$
    for independent standard normals $U_i,Z_i$ where $U_i$ is the $(i,1)$ element of the matrix $\bm{G}$.
    Furthermore the quantities
    in \eqref{hat_v_r_t}
    can be expressed in terms of
    $\bm{\htheta},\bm G$;
    we thus work with $\bm{G}$, 
    $\bm{\htheta}$ and its derivatives instead of $\bm \hbeta$. 
    Next, we have the decomposition
    $$
    \bm{x}_i^T \bm \hbeta
    =
    \bm{g}_i^T \bm \htheta
    =
    \aaa_* U_i
    +
    \hat r 
    ~\bm{z}_i^T \bm f
    $$
    where we recall
    $\aaa_* \defas \bm{e}_1^T \bm \htheta = \bm w^T\bm \Sigma \bm \hbeta$,
    the standard normal
    $\bm{z}_i\sim N(0, \bm I_{p-1})$
    as $\bm{z}_i = (g_{ik})_{k=2,...,p}$
    and $\bm{f}\in\R^{p-1}$ as 
    $\bm{f} = (\frac{\htheta_k}{\hat r})_{k=2,...,p}$.
    We apply \Cref{lemma-SOS}
    conditionally on $(\bm{y},(\bm I_n - \bm e_i \bm e_i^T) \bm G, G_{i1})$.
    Since $\bm{z}$ is independent of 
    $(\bm{y},(\bm I_n - \bm e_i \bm e_i^T) \bm G, G_{i1})$,
    the expectations in
    \Cref{lemma-SOS} are simply
    integrals with respect to the Gaussian measure of $\bm{z}$.
    Rewriting \Cref{lemma-SOS} with the notation of the present context
    yields
    \begin{equation}
    \E\Bigl[\Bigl(
    \bm{z}_i^T \bm f
    - \frac{\sigma_*}{\hat r} Z_i
    -\sum_{k=2}^p
    \frac{\partial}{\partial g_{ik}}
    \Bigl( 
    \frac{
    \htheta_k}{\|\bm{\hpsi}\|/\sqrt n}
    \Bigr)
    \Bigr)^2\Bigr]
    \le \C
    \E\sum_{k=2}^p
    \Bigl\|
    \frac{\partial}{\partial g_{ik}}
    \Bigl(
        \frac{\bm{\htheta} }{\|\bm \hpsi\|/\sqrt n}
    \Bigr)
    \Bigr\|^2
    \label{eq:SOS-applied-for-proximal-prediction}
    \end{equation}
    where
    $\sigma_*
    = \|\bm{P}_1^\perp \bm \htheta\|$
    for $\bm{P}_1^\perp \defas (\bm I_p - \bm e_1\bm e_1^T)$.
    We focus first on the sum in the left-hand side.
    By the product rule
    \begin{equation}
        \label{product-rule-htheta-hpsi}
        \frac{\partial}{\partial g_{ik}}
        \Bigl(
        \frac{\bm\htheta}{\|\bm{\hpsi}\|}
        \Bigr)
        =
        \frac{1}{\|\bm\hpsi\|}
        \Bigl[
            \frac{\partial}{\partial g_{ik}}\bm\htheta
        \Bigr]
        -
        \frac{\bm\hpsi^T }{\|\bm{\hpsi}\|^3}
        \Bigl[
            \frac{\partial}{\partial g_{ik}}\bm\hpsi
        \Bigr]
        \bm\htheta 
        ,
    \end{equation}
    the derivatives \eqref{derivatives-htheta},
    the definition of $\gamma_*,\hat r$
    and $\sum_{k=2}^p \bm{e}_k^T\bm{A} \bm e_k
    = \trace[\bm{A}] - \bm{A}_{11}$,
    \begin{align}
    \sum_{k=2}^p
    \frac{\partial}{\partial g_{ik}}
    \Bigl( 
    \frac{\sqrt n\htheta_k}{\|\bm{\hpsi}\|}
    \Bigr)
    - \frac{\gamma_*}{\hat r} \hat\psi_i
    &=
    -\frac{\bm{A}_{11}\bm\hat\psi_i}{\hat r}
    -
    \sum_{k=2}^p
    \frac{\sqrt n \htheta_k}{\|\bm{\hpsi}\|^3}
    \bm{\hpsi}^T \frac{\partial \bm\hpsi}{\partial g_{ik}}
    \nonumber
    \\&=
    -\frac{\bm{A}_{11}\hat\psi_i}{\hat r}
    +
    \sum_{k=2}^p
    \frac{\sqrt n}{\|\bm{\hpsi}\|^3}
    \bm{\hpsi}^T\Bigl(
        \bm{D} \bm G\bm{A} \bm P_1^\perp \bm \htheta \hpsi_i
        + \bm{V} \bm e_i \sigma_*^2\Bigr).
        \label{eq:previous-term-divergence-theta}
    \end{align}
    For the first term on the right-hand side,
    all observations $i=1,...,n$ are exchangeable so that
    $\E[\bm{A}_{11}^2 \frac{\hpsi_i^2}{\hat r^2}]
    = \frac 1 n 
    \sum_{l=1}^n
    \E[\bm{A}_{11}^2 \frac{\hpsi_l^2}{\hat r^2}]
    =
    \E[\bm{A}_{11}^2]$
    since $\hat r^2=\|\bm{\hpsi}\|^2/n$.
    By a similar symmetry argument for the second term on the right-hand side,
    $$
    \E\Bigl[
    \frac{n
    \sigma_*^2
    }{\|\bm{\hpsi}\|^4}
    \|\bm{D}\bm G\bm{A}\|_{op}^2
    \hpsi_i^2
    \Bigr]
    =
    \frac1n\sum_{l=1}^n
    \E\Bigl[
    \frac{n \sigma_*^2}{\|\bm{\hpsi}\|^4}
    \|\bm{D}\bm G\bm{A}\|_{op}^2
    \hpsi_l^2
    \Bigr]
    =
    \E\Bigl[
    \frac{\sigma_*^2}{\hat r^2}\|\bm{D}\bm G\bm{A}\|_{op}^2
    \Bigr]
    $$
    and the third term in the right-hand side, again by symmetry,
    satisfies
    $$
    \E[n\sigma_*^4  \|\bm{\hpsi}\|^{-6} (\bm\hpsi^T\bm V \bm e_i)^2]
    =
    \E[\sigma_*^4  \|\bm{\hpsi}\|^{-6} \|\bm V^T\bm\hpsi\|^2]
    \le
    \E[\sigma_*^4  \|\bm{\hpsi}\|^{-4} \|\bm V\|_{op}^2].
    $$
    The bounds \eqref{eq:bound-A-op-recap-before-proof}, \eqref{eq:bound-V-op}
    and \eqref{eq:bound-htheta} thus show that
    $\E[\eqref{eq:previous-term-divergence-theta}^2]
    \le \C(\delta,\tau,\kappa)/n$.
    It remains to bound from above the right-hand side
    of \eqref{eq:SOS-applied-for-proximal-prediction}.
    By exchangeability of $i=1,...,n$
    and the product rule
    \eqref{product-rule-htheta-hpsi},
    \begin{align}
    \E\sum_{k=2}^p
    \Bigl\|
    \frac{\partial}{\partial g_{ik}}
    \Bigl(
        \frac{\bm{\htheta} }{\|\bm \hpsi\|/\sqrt n}
    \Bigr)
    \Bigr\|^2
    &=
    \frac1n\sum_{l=1}^n
    \E\sum_{k=2}^p
    \Bigl\|
    \frac{\partial}{\partial g_{lk}}
    \Bigl(
        \frac{\bm{\htheta} }{\|\bm \hpsi\|/\sqrt n}
    \Bigr)
    \Bigr\|^2
    \text{(by exchangeability)}
    \nonumber
    \\&=\frac1n
    \E
    \sum_{l=1}^n
    \sum_{k=2}^p
    \Bigl\|
    \frac{\sqrt n }{\|\bm{\hpsi}\|}
    \frac{\partial\bm\htheta}{\partial g_{lk}}
    -
    \bm\htheta \frac{\sqrt n \bm\hpsi^T}{\|\bm{\hpsi}\|^3}
    \frac{\partial\bm\hpsi}{\partial g_{lk}}
    \Bigr\|^2
    \text{(chain rule)}
    \nonumber
    \\&\le
    \frac1n
    \E
    \sum_{l=1}^n
    \sum_{k=2}^p
    \frac{2n }{\|\bm{\hpsi}\|^2}
    \Bigl(
    \Bigl\|
        \frac{\partial\bm\htheta}{\partial g_{lk}}
    \Bigl\|^2
    +
    \frac{\|\bm\htheta\|^2}{\|\bm{\hpsi}\|^2}
    \Bigl\|
        \frac{\partial\bm\hpsi}{\partial g_{lk}}
    \Bigl\|^2
    \Bigr)
    \label{eq:to-bound-proximal-prediction}
    \end{align}
    by $(a+b)^2\le 2a^2+2b^2$.
    For the first term,
    using the explicit derivatives in \eqref{derivatives-htheta}
    and again $(a+b)^2\le 2a^2+2b^2$,
    \begin{equation}
    \frac 1 2 
    \sum_{l=1}^n
    \sum_{k=2}^p
    \frac{1}{\|\bm{\hpsi}\|^2}
    \Bigl\|
        \frac{\partial\bm\htheta}{\partial g_{lk}}
    \Bigl\|^2
    \le
    \|\bm{A}\|_F^2
    + \frac{\|\bm{A} \bm G^T \bm D\|_F^2 \|\bm \htheta\|^2}{\|\bm \hpsi\|^2}.
    \end{equation}
    while
    $\sum_{l=1}^n
    \sum_{k=2}^p
    \frac{1}{\|\bm{\hpsi}\|^2}
    \|
        \frac{\partial}{\partial g_{lk}} 
        \bm\hpsi
    \|^2
    $ is already bounded in \eqref{eq:inequality-square-gradient-psi}.
    The bounds
    \eqref{eq:bound-A-op-recap-before-proof}, \eqref{eq:bound-V-op} and \eqref{eq:bound-htheta}
    completes the proof of
    $
    \eqref{eq:SOS-applied-for-proximal-prediction}{}^2
    \le 
    \eqref{eq:to-bound-proximal-prediction}{}^2
    \le \C(\delta,\tau,\kappa)/n$.
\end{proof}

\subsection{Proof of \Cref{prop:sign}}

The first display in \Cref{prop:sign}
is bounded from above
by the same argument as for
as $\Gamma_6^*$:
after the change of variable detailed in \Cref{sec:change-of-variable},
the left-hand side of the first inequality of \Cref{prop:sign} equals
\begin{equation}
\E\Bigl[\Big|\frac{
 \varrho(\bm y)^T(\bm G\bm P_1^\perp \bm\htheta- \hat \gamma\bm\hpsi)
}{\|\varrho(\bm y)\| \|\bm \hpsi\|}
\Big|\Bigr].
\end{equation}
Applying \Cref{Rem_star_RHS} to the Gaussian matrix
$\bm G \bm P_1^\perp$ and to $\bm u = \varrho(\bm y)$ and $\bm f = \bm\htheta$,
we get after calculations similar
to the previous subsections
\begin{equation}
\E\Bigl[\Big|\frac{
 \varrho(\bm y)^T(\bm G\bm P_1^\perp \bm\htheta- \gamma_*\bm\hpsi)
}{\|\varrho(\bm y)\| \|\bm \hpsi\|}
\Big|^2\Bigr]^{1/2}
\le \C(\tau,\delta) / \sqrt n.
\end{equation}
Using \eqref{algebra-bound-gamma_star}
and the bounds in \Cref{subsubsec:Gammas-sqrtn}
lets us replace
$\gamma_*$ by $\hat\gamma$
which proves the first inequality
in \Cref{prop:sign}.

The left-hand side of the second display
in \Cref{prop:sign} is equal to
$\E[|\Gamma_9|]$
where $\Gamma_9$ is defined 
in \eqref{Gamma_9} which is bounded
as in the previous proofs as a 
consequence of the upper bounds in
\Cref{subsubsec:Gammas-sqrtn}.

\section{Proof in the unregularized case with $p<n$}
\label{sec:proof:unreg}

\subsection{Derivatives}
\Cref{lemma:derivative_unregularized}
is restated before its proof for convenience.

\derivativeUnregularized*
\begin{proof}[Proof of \Cref{lemma:derivative_unregularized}]
    If a minimizer $\bm \hbeta$ of \eqref{hbeta} exists
    and $\bm X^T\bm X$ is invertible,
    the optimality conditions read
    $\bm \varphi(\bm X, \bm \hbeta) = \bm 0$
    for $\bm \varphi(\bm X, \bm b) = \sum_{i=1}^n \bm x_i \ell_{y_i}'(\bm x_i^T\bm b)$ (for the derivatives in this paragraph,
    $\bm y$ is considered a constant).
    Then the Jacobian $\frac{\partial \bm \varphi}{\partial \bm b}\in\R^{p\times p}$ is the symmetric matrix
    $
    \sum_{i=1}^n \bm x_i \ell_{y_i}''(\bm x_i^T\bm b) \bm x_i^T
    $
    which is continuous and positive semi-definite in the open
    set $\{\bm X: \det(\bm X^T\bm X)>0\}$ since $\ell_{y_i}''>0$.
    By the implicit function theorem, there exists a continuously
    differentiable function $\bm b:\R^{n\times p}\to\R^p$
    in a neighborhood of $\bar{\bm X}$ such that
    $\bm \varphi(\bm X, \bm b(\bm X)) = \bm 0$ in this neighborhood.
    In other words, in this neighborhood, $\bm b(\bm X)$
    is a solution of \eqref{hbeta} and
     $\bm X\mapsto \bm \hbeta(\bm y, \bm X)$
    is continuously differentiable.
    By differentiation of $\sum_{i=1}^n
    \bm x_i \ell_{y_i}'(\bm x_i^T\bm \hbeta) = \bm 0$,
    we obtain \eqref{eq:derivatives-psi-hbeta}
    with $\bm{\hat{A}} = \Aunreg$. 
\end{proof}

\subsection{Non-separable losses}
\label{sec:thm:nonseparable-mathcalL}
\Cref{thm:nonseparable-mathcalL} is restated before its proof for convenience.

\ThmNonSeparable*
\begin{proof}[Proof of \Cref{thm:nonseparable-mathcalL}]
    By the same argument as the proof of
    \Cref{lemma:derivative_unregularized},
    the partial derivatives are given by \eqref{eq:derivatives-psi-hbeta}.
    Consider the change of variable
    in \Cref{sec:change-of-variable},
    so that $\bm G\in\R^{n\times p}$
    has iid $N(0,1)$ entries,
    and $\bm G\bm\htheta = \bm X \bm \hbeta$
    for 
    $\bm\htheta = \argmin_{\bm\theta\in\R^p}\sum_{i=1}^n \mathcal L(\bm G \bm\theta)$.
    Then
    $$
    \bm D = \nabla^2\mathcal L(\bm G\bm\htheta),
    \quad
    \bm\hpsi = -\nabla \mathcal L(\bm G\bm\htheta),
    \quad
    \bm A = (\bm G^T\bm D\bm G)^{-1},
    \quad
    \bm V = \bm D - \bm D \bm G \bm A \bm G^T \bm D
    $$
    and $\htheta_1 = \aaa_* = \bm\htheta^T\bm\Sigma \bm w$
    as well as $\sigma_*^2 = \|\bm P_1^\perp\bm\htheta\|^2$ 
    for $\bm P_1^\perp = \bm I_p - \bm e_1\bm e_1^T$.
    After the change of variable, the formula
    \eqref{derivatives-htheta} for the derivatives holds.
    Applying \Cref{prop26} conditionally on $(\bm y,\bm G\bm e_1)$
    with $K=n,Q=p-1$ to the matrix $\bm Z\in\R^{n\times (p-1)}$
    made of the last $p-1$ columns of $\bm G$ and with
    $\bm u(\bm Z) = \bm\hpsi/\sqrt n$, we obtain
    $$
    \frac1n\E\Big|(p-1)\|\bm \hpsi\|^2
    - \sum_{j=2}^p\bigl(\bm e_j^T\bm G^T\bm \hpsi - \sum_{i=1}^n \frac{\partial \hpsi_i}{\partial g_{ij}}\bigr)^2
    \Big|
    \le \C(\sqrt p (1+\Xi^{1/2})+\Xi)
    $$
    with $\Xi=\E\sum_{j=2}^p\sum_{i=1}^n\frac1n\|\frac{\partial}{\partial g_{ij}} \bm \hpsi\|^2$ as defined in \Cref{prop26}.
    Since $\bm G^T\bm \hpsi = \bm 0_p$ by the KKT conditions of $\bm\htheta$,
    the first term inside the square in the left-hand side is $0$.
    For the second term inside the square in the left-hand side,
    $
    - \sum_{i=1}^n \frac{\partial \hpsi_i}{\partial g_{ij}}
    = \trace[\bm V] \htheta_j + \bm \hpsi^T \bm D \bm G \bm A \bm e_j$.
    With the notation $\bm c\in\R^p, c_j\defas \bm\hpsi^T\bm D \bm G \bm A \bm e_j$,
    by expanding the square
    $(\trace[\bm V]\bm\htheta_j + c_j)^2$
    for each $j=2,...,p$
    using the triangle inequality, and dividing both sides by $n$ gives
    \begin{align}
        \nonumber
    &\tfrac{1}{n^2}\E\big|
    p\|\bm \hpsi\|^2
    -\trace[\bm V]^2\|\bm P_1^\perp\bm\htheta\|^2
    \big|
    \\&\le 
    \tfrac{1}{n^2}
    \E|
    2\trace \bm[\bm V]\bm \htheta^T\bm P_1^\perp \bm c
     + \|\bm P_1^\perp\bm c\|^2-\|\bm\hpsi\|^2
    |
    +
    \tfrac{1}{n}\C(\sqrt p (1+\Xi^{1/2})+\Xi).
    \label{unreg:to-bound-rhs-1}
    \end{align}
    The left-hand side reads $\E|p\hat r^2-n \hat v^2 \sigma_*^2|=
    \E|n\hat v^2(\hat\sigma^2-\sigma_*^2)|$
    thanks to $\hat\sigma^2=\frac{p}{n}\frac{\hat r^2}{\hat v^2}$,
    while right-hand side will be bounded later on.
    We now bound from above $|a_*^2 - \hat a^2|$.
    We start with
    $$\|\bm G\bm\htheta\|^2 
    = \|\bm G \bm P_1^\perp \bm\htheta + \bm G \bm e_1\htheta_1\|^2 
    = \bm \htheta^T\bm P_1^\perp \bm G^T (\bm G \bm P_1^\perp\bm\htheta + 2\bm G \bm e_1 \hat\theta_1) + \|\bm G\bm e_1\|^2\hat\theta_1^2.$$
    We wish to study the first term in the right-hand side.
    Applying \Cref{prop25} conditionally on $(\bm y,\bm G\bm e_1)$
    to $K=n, Q=p-1$,
    $\bm Z$ being the last $p-1$ columns of $\bm G$,
    $\bm u$ equal to the last $p-1$ components of $\bm\htheta$
    and $\bm f = \bm G[\bm I_p + \bm e_1\bm e_1^T]\bm\htheta$,
    we find
    $\|\bm G\bm\htheta\|^2 - \|\bm G\bm e_1\|^2\htheta_1^2
    = \bm f^T\bm Z\bm u$ so
    $$
    \E\Bigl[\Bigl(
    \|\bm G \bm \htheta\|^2-\|\bm G \bm e_1\|^2\htheta_1^2
    -\sum_{j=2}^p\sum_{i=1}^n\frac{\partial (f_i\htheta_j)}{\partial g_{ij}}
    \Bigr)^2\Bigr]
    \le \E \sum_{j=2}^p\sum_{i=1}^n \Big\| \frac{\partial(\bm \htheta \bm f^T)}{\partial g_{ij}}\Big\|^2.
    $$
    Applying the product rule for $\partial(f_i\htheta_j)$ in the left-hand side, we study each resulting term separately. First
    using $(\partial/\partial g_{ij})\htheta_j = \bm e_j^T\bm A \bm e_j \hpsi_i
    - \bm e_j^T\bm A \bm G^T\bm D \bm e_i \htheta_j$ by \eqref{derivatives-htheta} we get
    $$\sum_{i=1}^nf_i \sum_{j=2}^p \frac{\partial \htheta_j}{\partial g_{ij}}
    = \trace[\bm P_1^\perp\bm A] \bm f^T\bm \hpsi
    - \bm\htheta^T\bm P_1^\perp \bm A \bm G^T\bm D \bm f
    =
    - \bm\htheta^T\bm P_1^\perp \bm A \bm G^T\bm D \bm f
    = - \|\bm P_1^\perp\bm\htheta\|^2
    $$
    because $\bm f^T\bm \hpsi=0$ thanks to $\bm G^T\bm \hpsi=\bm 0_p$
    which are KKT conditions for the second equality,
    and using $\bm A \bm G^T\bm D \bm G=\bm I_p$ for the third equality.
    Now let us focus on the second term appearing from the product rule:
    $$
    \sum_{j=2}^p\htheta_j \sum_{i=1}^n\frac{\partial f_i}{\partial g_{ij}}
    =
    n \|\bm P_1^\perp\bm\htheta\|^2
    +
    \sum_{j=2}^p
    \htheta_j
    \sum_{i=1}^n
    \bm e_i^T\bm G (\bm I_p + \bm e_1\bm e_1^T) 
    \frac{\partial \bm \htheta}{\partial g_{ij}}.
    $$
    For the right-most double sum,
    using $(\partial/\partial g_{ij})\bm\htheta = \bm A (\bm e_j \hpsi_i-\bm G^T\bm D \bm e_i \htheta_j)$
    in \eqref{derivatives-htheta},
    the term stemming from
    $\bm A \bm e_j \hpsi_j$ equals 0
    thanks to $\sum_{i=1}^n\hpsi_i\bm e_i^T\bm G = \bm 0_p^T$ by the KKT conditions. This gives that the previous display equals
    $(n - \tilde p)\|\bm P_1^\perp\bm\htheta\|^2$
    where $\tilde p \defas
    \trace[\bm G(\bm I_p + \bm e_1\bm e_1^T)\bm A \bm G^T\bm D]=
    p+1$ since $\bm A \bm G^T\bm D \bm G = \bm I_p$.
    Combining the above identities yields
    $$
    \E\Bigl[\Bigl(
    \|\bm G \bm \htheta\|^2
    - \|\bm G \bm e_1\|^2\htheta_1^2
    - (n- p -2)\|\bm P_1^\perp\bm\htheta\|^2
    \Bigr)^2\Bigr]
    \le \E \sum_{j=2}^p\sum_{i=1}^n \Big\| \frac{\partial(\bm \htheta \bm f^T)}{\partial g_{ij}}\Big\|_F^2.
    $$
    We apply Jensen's inequality
    $\E|\cdot|\le \E[(\cdot)^2]^{1/2}$ to the left-hand side
    and replace $\|\bm G\bm e_1\|^2\htheta_1^2$ by
    $n\htheta_1^2$, incurring the error term $|\|\bm G\bm e_1\|^2-n|\htheta_1^2$ by the triangle inequality. This yields
    \begin{align}
        \nonumber
    \frac{1}{n}
    \E\Bigl|
    \|\bm G \bm \htheta\|^2
    - n\|\bm\htheta\|^2
    + (p+2) \|\bm P_1^\perp\bm\htheta\|^2
    \Bigr|
    &\le
    \E\Bigl[\Big|\frac{\|\bm G\bm e_1\|^2}{n} - 1\Big|\htheta_1^2\Bigr]
    +
    \frac 1n
    \E\Bigl[
    \sum_{j=2}^p\sum_{i=1}^n \Big\| \frac{\partial(\bm \htheta \bm f^T)}{\partial g_{ij}}\Big\|_F^2
    \Bigr]^{1/2}
    \\&
    \le
    \Bigl(
    \frac{2}{n}
    \E\Bigl[
        \htheta_1^4
    \Bigr]
    \Bigr)^{1/2}
    +
    \frac 1n
    \E\Bigl[
    \sum_{j=2}^p\sum_{i=1}^n \Big\| \frac{\partial(\bm \htheta \bm f^T)}{\partial g_{ij}}\Big\|_F^2
    \Bigr]^{1/2}
    \label{unreg:to-bound-rhs-2}
    \end{align}
    thanks to the Cauchy-Schwarz inequality
    and $\Var[\chi^2_n]=2n$ for the second inequality.
    Finally, for any vector $\bm a^2\in\R^p$ with unit norm
    orthogonal to $\bm e_1$,
    consider an orthonormal basis $\bm a^2,\bm a^3...,\bm a^{p}$ of $\{0\}\times \R^{p-1}$.
    By application of \Cref{lemma-SOS} to $\bm z = \bm G \bm a^k$
    conditionally on $(U,\bm G(\bm I_p-\bm a^k (\bm a^k)^T))$
    and $\bm f$ in \Cref{lemma-SOS} equal to $\bm \hpsi$,
    we have that for random variables $Z_k\sim N(0,1)$,
    $$
    \sum_{k=2}^p\E\Bigl[\Bigl(\sum_{i=1}^n \sum_{j=2}^p a_j^k \frac{\partial \hpsi_i}{\partial g_{ik}} - \|\bm \hpsi\|Z_k\Bigr)^2\Bigr]
    \le 15\E \sum_{k=2}^p \sum_{i=1}^n \|\sum_{j=2}^p a^k_j\frac{\partial \bm \hpsi}{\partial g_{ik}}\|^2 
    $$
    after observing that $\bm z^T\bm f = 0$.
    Furthermore, in the left-hand side
    $\sum_{i=1}^n \sum_{j=2}^p a_j^k \frac{\partial \hpsi_i}{\partial g_{ij}}
    = - \trace[\bm V] \bm\htheta^T\bm a^k
    - \bm \hpsi^T \bm D \bm G\bm A\bm a^k$.
    We further move
    $- \bm \hpsi^T \bm D \bm G\bm A\bm a^k$ to the right-hand side
    using the triangle inequality to find
    $$\sum_{k=2}^p\E\Bigl[\Bigl(
- \trace[\bm V] \bm\htheta^T\bm a^k
    - \|\bm \hpsi\|Z_k\Bigr)^2\Bigr]
    \le 30\E \sum_{k=2}^p \sum_{i=1}^n \|\sum_{j=2}^p a^k_j\frac{\partial \bm \hpsi}{\partial g_{ij}}\|^2 
    + 2\E \sum_{k=2}^p (\bm \psi^T\bm D\bm G\bm A \bm a^k)
    $$
    Using $\sum_{k=2}^p \bm a^k (\bm a^k)^T = \bm P_1^\perp$
    for the second term in the right-hand side, and
    $$\sum_{k=2}^p\|\sum_{j=2}^p \|\bm M\bm e_j \bm e_j^T\bm a^k\|^2
    = \sum_{k=2}^p\sum_{j,j'\ge 2}\trace[ \bm M^T\bm M \bm e_j\bm e_j^T\bm a^k (\bm a^k)^T \bm e_{j'} \bm e_{j'}^T]
    = \trace[\bm M^T\bm M \bm P_1^\perp]$$
    for the matrix $\bm M$ with columns $\bm M\bm e_j = \frac{\partial \bm\hpsi}{\partial g_{ij}}$
    again thanks to $\sum_{k=2}^p \bm a^k (\bm a^k)^T = \bm P_1^\perp$,
    we find that the right-hand side equals
    $30\E\sum_{i=1}^n\sum_{j=2}^p \|\frac{\partial \bm\hpsi}{\partial g_{ij}}\|^2
    + 2\E[\|\bm A\bm G\bm D \bm \hpsi\|^2]$.
    By symmetry and rotational invariance,
    in the summands of $\sum_{k=2}^p$ in the left-hand side,
    the expectations
    $\E[(
    - \trace[\bm V] \bm\htheta^T\bm a^k
    - \|\bm \hpsi\|Z_k)^2]$
    are all equal to other other,
    so that for any $\bm a\in\{0\}\times \R^{p-1}$ with $\|\bm a\|=1$,
    \begin{equation}
        \frac{ p-1 }{n^2}\E\Bigl[\Bigl( - \trace[\bm V]\bm a^T\bm\htheta - \|\bm \hpsi\|Z\Bigr)^2\Bigr]
        \le \frac{30}{n^2}\E \sum_{j=2}^p \sum_{i=1}^n \|\frac{\partial \bm \hpsi}{\partial g_{ij}}\|^2 
        + \frac{2}{n^2} \E[\|\bm A\bm G^T\bm D\bm \hpsi\|^2]
    \label{unreg:to-bound-rhs-3}
    \end{equation}
    for some $Z\sim N(0,1)$.
    Note that to obtain \eqref{unreg:to-bound-rhs-3}, we divided
    both sides by $n^2$.
    It remains to bound from above
    the right-hand sides in
    \eqref{unreg:to-bound-rhs-1},
    \eqref{unreg:to-bound-rhs-2},
    \eqref{unreg:to-bound-rhs-3}.
    We will use inequality
    \begin{equation}
    \|\bm D \bm G \bm A\|_F^2
    \le 
        \|\bm D^{1/2}\|_F^2 \|\bm A\|_{op}
        =\trace[\bm D]\|\bm A\|_{op},
    \label{bound-DGA}
    \end{equation}
    which follows from writing
    $\bm D \bm G \bm A
    = \bm D^{1/2}(\bm D^{1/2}\bm G \bm A^{1/2})\bm A^{1/2}
    $ with
    the matrix $\bm D^{1/2}\bm G \bm A^{1/2}$ inside the parenthesis
    having all singular values equal to one since $\bm A^{-1} = \bm G^T\bm D \bm G$.
    We will also use $\|\bm P_1^\perp\|_{op}\le 1$ and
    $\bm 0_{n\times n}\preceq \bm V \preceq \bm D$ in the sense of psd matrices which implies $\trace[\bm V]\le \trace[\bm D]$
    and $\|\bm V\|_F^2\le \|\bm D\|_F^2$.
    We first focus on the quantity
    $
    \Xi =   \frac 1 n\E \sum_{j=2}^p\sum_{i=1}^n \|\frac{\partial}{\partial g_{ij}} \bm \hpsi\|^2$
    appearing in the right-hand side
    of \eqref{unreg:to-bound-rhs-1} and \eqref{unreg:to-bound-rhs-3}.
    By \eqref{derivatives-htheta},
    \begin{align*}
    \tfrac 1 2 
    \Xi 
    &\le
    \tfrac{1}{n}
    \E
    \bigl[
    \|\bm\hpsi\|^2 \|\bm D \bm G \bm A \bm P_1^\perp\|_F^2
    +
    \|\bm V\|_F^2 \|\bm P_1^\perp \bm\htheta\|^2
    \bigr] 
    \qquad
    &&\text{ by } \tfrac12\|\cdot+\cdot\|^2\le \|\cdot\|^2+ \|\cdot\|^2
    \\&\le
    \E
    \bigl[
    \|\bm D \bm G \bm A \bm P_1^\perp\|_F^2
    +
    \tfrac 1 n\|\bm V\|_F^2 \|\bm P_1^\perp \bm\htheta\|^2
    \bigr] 
    \qquad
    &&
    \text{ using } \|\bm\hpsi\|^2\le n
    \\&\le
    \E
    \bigl[
    \trace[\bm D]\|\bm A\|_{op}
    +
    \tfrac 1 n\|\bm D\|_F^2 \|\bm\htheta\|^2
    \bigr] 
    \qquad
    &&
    \text{ using \eqref{bound-DGA} and } 
    \|\bm V\|_F\le \|\bm D\|_F
    \\  &\le
    \E
    \bigl[
    (\tfrac{1}{\sqrt n}\|\bm D\|_F)\|n \bm A\|_{op}
    +
    (\tfrac 1 n\|\bm D\|_F^2) \|\bm\htheta\|^2
    \bigr] 
    \qquad
    &&
    \text{ using }\trace[\bm D]\le \sqrt n \|\bm D\|_F
    \end{align*}
    We are left with only the quantities $\|\bm\htheta\|^2$,
    $\frac 1 n \|\bm D\|_F^2$ and $\|n\bm A\|_{op}$
    that appear in the right-hand side of
    \eqref{eq:conclusion_nonseparable}.
    We now bound from above the terms in the right-hand side of 
    \eqref{unreg:to-bound-rhs-1}, using the loose bound
    $\|\bm D \bm G \bm A\|_{op}\le \|\bm D \bm G \bm A\|_F$
    and \eqref{bound-DGA}:
    \begin{align*}
        \tilde\Xi &\defas 
    \tfrac{1}{n^{3/2}}
    \E\big|
    2\trace \bm[\bm V]\bm \htheta^T\bm P_1^\perp \bm c
     + \|\bm P_1^\perp\bm c\|^2
    \big|
    \\&\le
    \tfrac{1}{n^{3/2}}
    \E
    \bigl[
    2 \trace[\bm V]
    \|\bm\hpsi\|\|\bm D \bm G \bm A\|_{op}\|\bm\htheta\|
    +
    \|\bm\hpsi\|^2
    \|\bm D\bm G \bm A\|_{op}^2
    \bigr]
      &&\text{ since } \bm c \defas \bm A \bm G^T\bm D \bm \hpsi
    \text{ in \eqref{unreg:to-bound-rhs-1}}
    \\&\le
    \E[\tfrac2n
    \trace[\bm D]^{3/2}\|\bm A\|_{op}^{1/2} \|\bm\htheta\|
    + \tfrac{1}{\sqrt n}\trace[\bm D] \|\bm A\|_{op}
    ]
      &&\text{ by }\trace[\bm V]\le\trace[\bm D],
    \eqref{bound-DGA},\|\bm\hpsi\|^2\le n
    \\&\le
    \E[2
    (\tfrac{1}{\sqrt n}\|\bm D\|_F)^{3/2}
    \|n \bm A\|_{op}^{1/2} \|\bm\htheta\|
    +
    \tfrac1n\|\bm D\|_F) \|n\bm A\|_{op}
    ]
      && \text{ using }
    \trace[\bm D]\le \sqrt n \|\bm D\|_F
    \end{align*}
    Since 
    $\eqref{unreg:to-bound-rhs-1}
    \le \C (n^{-1/2}(\tilde \Xi +\Xi^{1/2}) + n^{-1}(1+\Xi))$,
    we have proved that
    $[\eqref{unreg:to-bound-rhs-3}]^{1/2}$ and
    \eqref{unreg:to-bound-rhs-1}
    are bounded from above by the right-hand side
    of \eqref{eq:conclusion_nonseparable} as desired.

    It remains to bound \eqref{unreg:to-bound-rhs-2},
    again using the derivatives in \eqref{derivatives-htheta}
    to differentiate $\bm \htheta \bm f^T$ in \eqref{unreg:to-bound-rhs-2}.
    Since $\bm f = \bm G (\bm I_p + \bm e_1\bm e_1^T)\bm \htheta$ 
    and $(\bm I_p + \bm e_1\bm e_1^T)$ has operator norm at most 2,
    denoting $\partial$ for $\frac{\partial}{\partial g_{ij}}$ for brevity,
    if $j\ge 2$,
    \begin{align*}
    \|\partial
    (\bm\htheta \bm f^T)
    \|
    &=
    \|(\partial\bm\htheta)\bm\htheta^T(\bm I_p+\bm e_1\bm e_1^T)\bm G^T
    + \bm\htheta \htheta_j \bm e_i^T + \bm\htheta (\partial \bm\htheta)^T(\bm I_p+\bm e_1\bm e_1^T)\bm G^T\|
  \\&\le
    8\|\partial \bm\htheta\| \|\bm \htheta\| \|\bm G\|_{op}
    + \|\bm\htheta\| |\htheta_j|
  \\&\le
    8\|\bm A \bm e_j\| | \psi_i| \|\bm \htheta\| \|\bm G\|_{op}
    +
    8\|\bm A \bm G^T\bm D \bm e_i\| |\htheta_j| \|\bm \htheta\| \|\bm G\|_{op}
    + \|\bm\htheta\| |\htheta_j|
    \end{align*}
    thanks to \eqref{derivatives-htheta} for the last inequality.
    Taking the square, taking the double sum
    $\sum_{i=1}^n\sum_{j=2}^p$ and using $\sum_{i=1}^n\|\bm M \bm e_i\|^2=\|\bm M\|_F^2$ for any matrix $\bm M$ we find
    $$
    \sum_{i=1}^n
    \sum_{j=2}^p
    \|\frac{\partial (\bm \htheta \bm f^T)}{\partial g_{ij}}\|^2
    \le 
    \C\Bigl[
    \|\bm \hpsi\|^2 \|\bm A\|_F^2 \|\bm G\|_{op}^2\|\bm \htheta\|^2
    +\|\bm A \bm G^T\bm D\|_F^2 \|\bm G\|_{op}^2 \|\bm \htheta\|^4
    +n \|\bm \htheta\|^4.
    \Bigr].
    $$
    We use $\|\bm \hpsi\|^2\le n$ for the first term
    and
    $\|\bm A\bm G^T\bm D\|_F^2
    \le \trace[\bm D]\|\bm A\|_{op}$ by \eqref{bound-DGA} for the second.
    We have proved that \eqref{unreg:to-bound-rhs-2}
    is bounded from above by
    $$
    \E[(\tfrac 1 n +\tfrac{1}{n^2}\|\bm G\|_{op}^2\trace[\bm D]\|\bm A\|_{op})\|\bm\htheta\|^4]^{1/2}
    +
    \E[\|\bm A\|_F^2\|\bm G\|_{op}^2\|\bm\htheta\|^2/n]^{1/2}
    .$$
    We complete the proof using $0\le \trace[\bm D]\le \sqrt n \|\bm D\|_F$
    and $\|\bm A\|_F^2\le p \|\bm A\|_{op}^2$, so that
    \eqref{unreg:to-bound-rhs-1}, \eqref{unreg:to-bound-rhs-2}
    and \eqref{unreg:to-bound-rhs-3} are all bounded from above by
    the second line in \eqref{eq:conclusion_nonseparable}.
\end{proof}

\subsection{Lemmas using the density of the smallest eigenvalue of Wishart matrices}

    \begin{lemma}
        If $\frac p n \le \delta^{-1} < (1-\alpha)$
        for constants $\delta,\alpha>0$
        and
        $\bm G\in\R^{n\times p}$ has iid $N(0,1)$ entries 
        then for $\bm P_I = \sum_{i\in I} \bm e_i \bm e_i^T$ we have
        $$\P\Bigl[
            \min_{
                I\subset [n],
                |I| = \lceil (1-\alpha)n \rceil
            }
            \lambda_{\min}\Bigl(\frac{\bm G^T\bm P_I\bm G}{n}\Bigr)
            < c_0
        \Bigr]
        \le
        \sum_{
            I\subset [n],
            |I| = \lceil(1-\alpha)n\rceil
        }
        \P\Bigl[
            \lambda_{\min}\Bigl(\frac{{\bm G^T\bm P_I\bm G}}{n}\Bigr)
            < c_0
        \Bigr]
        \to 0$$
        as $n,p\to+\infty$
        for some constant $c_0= c_0(\delta,\alpha)>0$ depending on
        $(\delta,\alpha)$ only.
    \end{lemma}
    \begin{proof}
    Variants of the following argument were used previously
    to study restricted isometry properties in \cite{blanchard2011compressed}
    as a consequence of results from \cite{edelman1988eigenvalues}.
    By \cite[proof of Lemma 4.1]{chen2005condition}
    (with $\frac{A}{x}$ there equal to $t>0$ here),
    if $N=|I|$
    \begin{align*}
    \P\Bigl(
        \lambda_{\min}(\frac{ \bm G^T\bm P_I\bm G}{N})
        \le t^2
    \Bigr)
    \le  \frac{(tN)^{N-p+1}}{ \Gamma(N-p+1)}
    \le
        \bigl(\frac{etN}{N-p+1}\bigr)^{N-p+1}
        \frac{1}{(2\pi(N-p+1))^{1/2}}
    \end{align*}
    where the second inequality follows from
    the lower bound on $\Gamma(N-p+1)$ given for instance in
    \cite[p10, Proof of Theorem 4.5]{chen2005condition}.
    Taking the union bound over $\binom{n}{N}$ possible sets
    $I\subset [n]$, it is sufficient to show that for a small enough constant
    $t>0$,
    \begin{equation}
    \binom{n}{N}
    \bigl(\frac{etN}{N-p+1}\bigr)^{N-p+1}
    \bigl(2\pi(N-p+1)\bigr)^{-1/2}
    \label{eq:bound-edelman}
    \end{equation}
    converges to 0.
    First, $\binom{n}{N}\le \exp( n \frac{N}{n}\log(e\frac{n}{N}))$
    by a standard bound on binominal coefficient, so that
    $\binom{n}{M}\le \exp(n \log(e(1-\alpha)^{-1})
    \le \exp((N-p+1) C(\delta,\alpha))$ 
    where $C(\delta,\alpha)$ depends on $\delta,\alpha$ only.
    This proves that the previous display converges to 0
    exponentially fast in $n$
    if $t=c_0(\delta,\alpha)$ is a small enough constant 
    depending only on $\delta,\alpha$.
    \end{proof}

\begin{lemma}
    \label{lem:edelman_exp}
    If $\frac p n \le \delta^{-1} < (1-\alpha)$
    for constants $\delta,\alpha>0$
    and
    $\bm G\in\R^{n\times p}$ has iid $N(0,1)$ entries 
    then for $\bm P_I = \sum_{i\in I} \bm e_i \bm e_i^T$ we have for any $k\ge1$
    \begin{equation}
        \label{eq:edelman_exp}
    \E\Bigl[
        \max_{
            I\subset [n],
            |I| = \lceil (1-\alpha)n \rceil
        }
        \lambda_{\min}\Bigl(\frac{\bm G^T\bm P_I\bm G}{n}\Bigr)^{-k}
    \Bigr]
    \le
    \C(\delta,\alpha,k)
    \end{equation}
    for a constant depending on $\delta,\alpha,k$ only.
\end{lemma}
\begin{proof}
    Let $N=\lceil (1-\alpha)n\rceil$,
    $W_I=
    \lambda_{\min}(\frac{1}{N}\bm G^T\bm P_I\bm G)$ and $u_0>0$ a constant to be specified. Using $\E[W]=\int_0^\infty \P(W>u)du\le u_0+\int_{u_0}^\infty P(W>u)du$ for any non-negative $W$,
$$
\E\Bigl[\max_{I\subset [n]:|I|=N}
W_I^{-k}\Bigr]
\le 
u_0 +
\int_{u_0}^\infty \P\Bigl[\max_{I\subset[n]:|I|=N} W_I^{-k}>u\Bigr] du
\le u_0 +
\int_{u_0}^\infty 
\binom{n}{N}
\P(W_I^{-k}>u) du
$$
by the union bound for the second term.
The integrand is thus smaller than
\eqref{eq:bound-edelman}
for $t = u^{-1/(2k)}$.
Integrating with respect to $u$ over $[u_0,\infty]$ gives the upper bound
$$
u_0 +
\binom{n}{N}
\Bigl(\frac{e}{N-p+1}\Bigr)^{N-p+1}
\bigl(2\pi(N-p+1)\bigr)^{-1/2}
\frac{u_0^{-(N-p+1)/(2k)+1}}{(N-p+1)/(2k)-1}.
$$
The previous display is bounded from above by $\C(\delta,\alpha,k)$ if $u_0$ is a large enough constant depending only on $\delta,\alpha,k$ by the same argument as that used after \eqref{eq:bound-edelman}.
\end{proof}

\subsection{Adding a coercive term to the M-estimation objective}
\label{sec:proof:thm:unregularized_with_H}
\Cref{thm:unregularized_with_H} is restated here before its proof for convenience.
\ThmUnregularizedExtra*
\begin{proof}[Proof of \Cref{thm:unregularized_with_H}]
    Note that $\mathcal L$ is coercive thanks to the second term,
    so that $\bm{\hat b}$ always exists. 
    Let $\hat u = \frac{1}{2}(\frac{1}{n}\|\bm X \bm{\hat{b}}\|^2 - K)$.
    The KKT conditions of $\bm{\hat b}$ read
    $\bm X^T(\ell_{\bm y}'(\bm X \bm{\hat b}) + h(\hat u)\bm X \bm{\hat{b}})=0$
    Note that by construction,
    $\hat u <  1$ if and only if $\frac 1 n \|\bm X\bm{\hat b}\|^2 < K+2$
    if and only if $h(\hat u) < 1$.
    If $\frac 1n \|\bm X\bm{\hat{b}}\|^2>K+2$ then $h(\hat u)=1$ and the KKT conditions imply
    $\bm X\bm{\hat{b}} = - \bm X(\bm X^T\bm X)^{-1}\bm X^T\ell_{\bm y}'(\bm X\bm{\hat b})$ so that $\|\bm X\bm{\hat{b}}\|^2 \le \|\ell_{\bm y}'(\bm X\bm{\hat b})\|^2\le n$ because $|\ell'_{y_i}(\cdot)|\le 1$ by \Cref{assumExtra}.
    This contradicts 
    $\frac 1n \|\bm X\bm{\hat{b}}\|^2>K+2$ so that
    $\frac 1n \|\bm X\bm{\hat{b}}\|^2\le K+2$ always holds.
    We also have 
    \begin{align}
        \nonumber
        \tfrac{\partial}{\partial v_i}\mathcal L(\bm v) &= (1+\sqrt{K+2})^{-1}[\ell_{y_i}'(v_i) + h(\tfrac{1}{2}(\tfrac{1}{n}\|\bm v\|^2-K)) v_i]
    \\
    \|\nabla\mathcal L(\bm v)\|&\le (1+\sqrt{K+2})^{-1}[\sqrt n + \|\bm v\|].
    \label{eq:bound-r-mathcalL}
    \end{align}
    Since $\|\bm X\bm{\hat b}\|^2\le n(K+2)$, it follows that
    $\|\nabla\mathcal L(\bm X\bm{\hat b})\|\le \sqrt n$ always holds.
    Let $\alpha>0$ be a constant such that $\delta^{-1}\le 1-\alpha$
    (for instance, $\alpha=1-\delta^{-1/2})$. Since $\frac 1 n \|\bm X \bm{\hat b}\|^2 \le K+2$, by Markov's inequality
    there exists
    a random set $I\subset[n]$ of size $\lceil (1-\alpha)n\rceil$ such that
    $\max_{i\in I} (\bm x_i^T\bm{\hat b})^2 \le (K+2)/\alpha$.
    Let $\bm P_I = \sum_{i\in I}\bm e_i\bm e_i^T$. Since the Hessian of a convex
    function is positive, we have
    \begin{equation}
    \nabla^2 \mathcal L(\bm X\bm{\hat b})
    \succeq \sum_{i\in I} \ell_{y_i}''(\bm x_i^T\bm{\hat b}) \bm e_i \bm e_i^T
    \succeq 
    \min_{i\in I} \ell_{y_i}''(\bm x_i^T\bm{\hat b})  \bm P_I 
    \succeq 
    \bm P_I c(\alpha, K, \ell)
    \label{eq:strong_convexity_local}
    \end{equation}
    in the sense of the psd order,
    where
    $c(\alpha,K,\ell) = \min_{y\in\mathcal Y}\min_{u\in\R:|u|\le\sqrt{K+2}/\sqrt\alpha} \ell_{y}''(u)$ which is a positive constant by
    \Cref{assumExtra}.
    If $\hbA=(\bm X^T\nabla^2\mathcal L(\bm X\bm{\hat b})\bm X)^{-1}$,
    \begin{equation}
    \|n \bm \Sigma^{1/2}\hbA\bm\Sigma^{1/2}\|_{op}
    \le 
    \frac{1}{\lambda_{\min}(\frac1n \bm\Sigma^{-1/2}\bm X^T\bm D \bm X\bm\Sigma^{-1/2})}
    \le 
    \frac{c(\alpha,K,\ell)^{-1}}{
    \lambda_{\min}(\frac1n \bm\Sigma^{-1/2}\bm X^T\bm P_I \bm X\bm\Sigma^{-1/2})
    }.
    \label{eq:integrable-A-unregularized}
    \end{equation}
    The idea behind \eqref{eq:strong_convexity_local} is to ensure
    strong convexity of the loss $\mathcal L$ when restricted to an L2 ball
    of any radius, with the strong convexity constant depending on the radius;
    such local strong convexity has been used before, e.g.,
    \cite[Lemma 4]{sur2019likelihood} or \cite{alquier_cottet_lecue2019lipschitz_loss}. We now combine this local strong convexity
    with the moment bounds
    in \Cref{lem:edelman_exp}: By \eqref{eq:edelman_exp}, any finite moment of the above display
    is bounded from above by $\C(\alpha,\delta,K,\ell)$.
    Finally,
    \begin{align}
    \nabla^2 \mathcal L(\bm v)
    &= \diag(\{\ell_{y_i}''(v_i)\}_{i\in [n]})
    + h(\tfrac12(\tfrac1n\|\bm v\|^2 -K)) \bm I_n
    + {\tfrac1n}h'(\tfrac12(\tfrac1n\|\bm v\|^2 -K)) \bm v\bm v^T
    \label{eq:bound-hessian}
    \\
    \|\nabla^2\mathcal L(\bm X\bm {\hat b})\|_{op}
    &
    \le \max_{y\in\mathcal Y}\sup_{u\in\R:|u|\le\sqrt{K+2}/\sqrt \alpha}
    \ell_{y}''(u) + 1 + \sup|h'|(K+2)
    .
    \nonumber
    \end{align}
    We have proved that the right-hand side in
    \eqref{eq:conclusion_nonseparable} is bounded from above
    by $\C(\delta,K,\ell)$.

    We now show that $\max\{\hat r,\frac{1}{\hat r},\hat v,\frac{1}{\hat v}\}$
    is bounded by a constant.
    Inequality \eqref{eq:bound-r-mathcalL} gives the desired
    upper bound on $\hat r$. Since $\min_{y\in\mathcal Y}\ell_y'(\cdot)$ is positive continuous and all $(\bm x_i^T\bm b)_{i\in I}$ are bounded from above
    with $I$ of size proportional to $n$,
    this shows that $\hat r^2\ge \C(\ell,\delta,K)$.

    We have $\hat v = \trace[\bm V]$ for $\bm V=\bm D-\bm D \bm X^T\hbA\bm X \bm D$
    for $\bm D = \nabla^2 \mathcal L(\bm X \bm{\hat b})$.
    Using $\bm V \preceq \bm D$
    in the sense of the psd order 
    gives $\hat v\le \|\bm D\|_{op}$
    so that $\hat v\le \C(\ell,\delta,K)$ thanks to \eqref{eq:bound-hessian}.
    It remains to show that $\hat v^{-1}$ is bounded from above.
    Since $\bm P \defas \bm D^{1/2} \bm X^T\hbA\bm X \bm D^{1/2}$ is an orthogonal
    projector of rank $p$, by the interlacing eigenvalue theorem,
    $\trace[\bm D (\bm I_n - \bm P)]\ge
    \sum_{i=1}^{n-p} \lambda_i(\bm D)$
    where $(\lambda_i(\bm D))_{i=1,...,n-p}$ are the $n-p$ smallest eigenvalues
    of $\bm D$.
    Let $J\subset [n]$ of size $|J|=n-p$ be the set of indices of these $n-p$ smallest eigenvalues. Then $\sum_{i\in J}(\bm x_i^T\bm{\hat b})^2 \le n(K+2)$
    so at least $(1-\alpha)(n-p)$ indices $i\in J$ satisfy
    $(\bm x_i^T\bm{\hat b})^2\le \frac{n}{n-p}(K+2)/\alpha$.
    This implies that at leat $(1-\alpha)(n-p)$ eigenvalues of $\bm D$
    are at least larger than
    $\min_{y\in \mathcal Y}\min_{u\in\R:|u|\le \sqrt{\delta^{-1}(K+2)}/\sqrt \alpha}
    \ell_{y}''(u)$ and
    $\trace[\bm V]\ge \C(\alpha,\delta,\ell,K) (n-p)$ must hold.
\end{proof}

\subsection{Proof of \Cref{thm:unregularized}}
\label{sec:proof_thm_unregularized}
The result is restated for convenience.
\ThmUnregularized*
\begin{proof}[Proof of \Cref{thm:unregularized}]
By \Cref{thm:unregularized_with_H} we know that $\bm{\hat b} = \bm\hbeta$
and $\hat v^{-1} \le \C(K,\ell,\delta)$
in the event $E$ from \Cref{thm:unregularized}, and that
the right-hand side of \eqref{eq:conclusion_nonseparable}
is bounded by $\C(K,\ell,\delta)/\sqrt n$.
This the upper bound on $\E[\hat v^2|\hat \sigma^2-\sigma_*^2|]+
\E[\hat v^2|\hat\aaa^2-\aaa_*^2|]$ implies the desired upper bound
on $\E[I_E|\hat\sigma^2-\sigma_*^2|]$ and
$\E[I_E|\hat\aaa^2-\aaa_*^2|]$ stated in
\eqref{eq:bound-correlation-unreg}.

Let $\bm P =\bm I_p - \bm\Sigma \bm w \bm w^T$.
Finally, \eqref{eq:thm-Z_j-unregularized} is a direct consequence
of the upper bound \eqref{eq:conclusion_nonseparable}
on 
$\E[
    (\sqrt n \hat v \bm\hbeta^T\bm u - \hat r Z )^2
    ]^{1/2}$
by taking $\bm u = \frac{\bm P \bm e_j}{\|\bm\Sigma^{-1/2}\bm P \bm e_j\|}$,
since we have 
$\bm w^T\bm P \bm e_j = w_j - \bm w^T\bm\Sigma\bm w w_j = 0$
so that $\bm u^T\bm w=0$ on the one hand, and on the other
$$
\bm\hbeta^T\bm P\bm e_j = \hbeta_j - \aaa_* w_j,
\qquad
\|\bm\Sigma^{-1/2}\bm P\bm e_j\|^2
=
\bm e_j^T\bm P^T\bm\Sigma^{-1}\bm P \bm e_j
= \Omega_{jj} - w_j^2
$$
so that $\bm\hbeta^T\bm u = (\Omega_{jj}-w_j^2)^{-1/2}(\hbeta_j - \aaa_*w_j)$
as desired in \eqref{eq:thm-Z_j-unregularized}.
\end{proof}

\section{Proof in the case of Ridge regularization}
\label{sec:proof_easy_ridge}

The theorem is restated for convenience.

\easyRidgeProposition*

\begin{proof}[Proof of \Cref{prop:improved-ridge}]
    By rotational invariance and without loss of generality,
    assume that $\bm w=p^{1/2}\bm e_1$ where $\bm e_1$ is the first
    canonical basis vector in $\R^p$.
    Let $\bm P = \bm I_p - \frac{\bm w \bm w^T}{\|\bm w\|^2}
    = \bm I_p - \bm e_1 \bm e_1^T$.
    \Cref{prop25} applied to the Gaussian matrix $\bigl(p^{1/2}x_{ij}\bigr)_{i\in[n],j=2,...,p}$,
        $\bm u = \bm \hpsi$ and $\bm f = (\hbeta_j / \sqrt p)_{j=2,...,p}$ gives
    $$
    \E\Bigl[
    \Bigl(
    \bm\hpsi ^T \bm X \bm P \bm\hbeta
    -\sum_{j=2}^p\sum_{i=1}^n \frac{1}{p} \frac{\partial (\hbeta_j\hpsi_i)}{\partial x_{ij}}
    \Bigr)^2
    \Bigr]
    \le
    \E\Bigl[
    \|\bm \hpsi\|^2
    \frac{\|\bm P \bm \hbeta\|^2}{p}
    +
    \sum_{j=2}^p\sum_{i=1}^n
    \frac{1}{p^2}
    \Big\|
    \frac{\partial(\bm\hpsi \bm \hbeta^T\bm P)}{\partial x_{ij}}
    \Big\|_F^2
    \Bigr].
    $$
    In the left-hand side,
    $\bm \hpsi^T\bm X \bm P \bm\hbeta = \frac{\lambda n}{p} \|\bm P\bm\hbeta\|^2$
    by \eqref{eq:actualKKT_ridge}.
    By \eqref{eq:derivatives-psi-hbeta} we also have
    \begin{align*}
        \sum_{i=1}^n
        \hpsi_i \sum_{j=2}^p \frac{\partial\hbeta_j}{\partial x_{ij}}&= 
        \Bigl(\sum_{j=2}^p\hat A_{jj}\Bigr)\|\bm \hpsi\|^2
        - \bm\hbeta^T\bm P\bm{\hat{A}}\bm X^T\bm D\bm \hpsi,
        \\
        \sum_{j=2}^p
        \hbeta_j \sum_{i=1}^n \frac{\partial\psi_i}{\partial x_{ij}}&=
    -\bm \hpsi^T\bm D \bm X\bm{\hat{A}}\bm P\bm\hbeta
    - \trace[\bm V] \|\bm P\bm\hbeta\|^2.
    \end{align*}
    Since $\sum_{j=2}^p\hat A_{jj} = \trace[\bm P\bm{\hat{A}}] = \trace[\bm{\hat{A}}] - \hat{A}_{11}$,
    moving the terms involving $\hat{A}_{11}$ and
    $\bm \hpsi^T\bm D \bm X\bm{\hat{A}}\bm P\bm\hbeta$ to the right-hand
    side, this proves
    \begin{align*}
    &n^2 \E\Bigl[
    \Bigl(
        (\lambda+ \hat v)\sigma_*^2
        - \gamma_* \hat r^2
    \Bigr)^2
    \Bigr]
    =\E\Bigl[
    \Bigl(
        \frac{(\lambda n + \trace[\bm V])\|\bm P\bm\hbeta\|^2}{p}
        - \frac{\trace[\bm{\hat{A}}]}{p}
        \|\bm\hpsi\|^2
    \Bigr)^2
    \Bigr]
    \\&\le
    \C
    \E\Bigl[
    \|\bm \hpsi\|^2
    \Bigl[
\frac{\hat{A}_{11}^2\|\bm \hpsi\|^2}{p^2} + \Bigl(\frac{1}{p} + \frac{\|\bm D\bm X\bm{\hat{A}}\|_{op}^2}{p^2}\Bigr) \|\bm P \bm \hbeta\|^2
    \Bigr]
    +
    \sum_{j=2}^p\sum_{i=1}^n
    \frac{1}{p^2}
    \Big\|
    \frac{\partial(\bm\hpsi \bm \hbeta^T\bm P)}{\partial x_{ij}}
    \Big\|_F^2
    \Bigr].
    \end{align*}
    We further bound the last line using
    $\|\bm\hbeta\|  \le \frac{p}{\lambda n} \|\bm X\|_{op} \|\bm \hpsi\|$
    by \eqref{KKT-ridge},
    $\|\bm D\|_{op} \le 1$ by \Cref{assum},
    inequality $\|\bm \hpsi\|^2\le n$ granted by the additional assumption
    $\max_{y_0}|\ell_{y_0}'|\le 1$, inequality 
    $\|\bm{\hat{A}}\|_{op} \le  \frac{p}{\lambda n}$,
    the derivatives formula in \eqref{eq:derivatives-psi-hbeta}
    for the rightmost term,
    and $\E[\|\bm X\|_{op}^c] \le \C(\delta,c)$
    for any numerical constant $c>0$
    by \eqref{eq:7-norm-G} below.
    The conclusion \eqref{eq:conclusion-improved-ridge}
    is then obtained by dividing by $n^2(\lambda+\hat v)$.
\end{proof}

\section{Proof: $\hat t/\hat v$ estimates the signal strength in linear models}

\Cref{prop:hat-t-linear-model} is restated before its proof for convenience.

\PropHatTLinearModel*
\begin{proof}[Proof of \Cref{prop:hat-t-linear-model}]
    Recall that $t_* \defas \bm w^T (\bm X^T  \bm \hpsi/n + \hat v\bm \Sigma \bm \hbeta)$.
    It is easier to work with the change of variable of
    \Cref{sec:change-of-variable}, so that after the change of variable
    $t_* = \bm e_1^T(\bm G^T \bm \hpsi + \trace[\bm V] \bm \htheta)/n$
    with $\bm e_1\in\R^p$ the first canonical basis vector.
    Let also $\bm\Theta$ be the true regression vetor
    after change of variable, that is, $\bm\Theta\in\R^p$ such that
    $\bm y = \bm G \bm\Theta + \bm \eps$
    (i.e., $\bm\Theta = \bm e_1 \|\bm\Sigma^{1/2}\bm\beta^*\|)$.
    We apply \Cref{lemma-SOS} to $\bm z= \bm G\bm e_1 \sim N(\bm 0,\bm I_n)$
    conditionally on $\bm G (\bm I_p - \bm e_1 \bm e_1^T)$
    (i.e., conditionally on the last $p-1$ columns of $\bm G$),
    and to the function $\bm f(\bm z) = \frac{\bm \hpsi}{(\|\bm \hpsi\|^2 + n\Theta_1^*)^{1/2}}$.
    Here, we have to take into account
    that $\bm y$ is not independent of $\bm G \bm e_1$, and to take into account the derivatives of $\bm y$ with
    respect to $\bm G \bm e_1$ conditionally
    on the last $p-1$ columns of $\bm G$ and the noise $\bm \eps$.
    This gives
    \begin{align*}
    \frac{\partial \bm f}{\partial z_i}
    &=
    \Bigl(\bm I_n
    -
    \frac{\bm \hpsi \bm \hpsi^T}{\text{denom}^2}
    \Bigr)
    \frac{1}{\text{denom}}
    \Bigl\{
        \frac{\partial}{\partial g_{i1}}
        +
        \Theta_1
        \frac{\partial}{\partial y_i}
    \Bigr\}
    \bm \hpsi
    \\
    &=
    \Bigl(\bm I_n
    -
    \frac{\bm \hpsi \bm \hpsi^T}{\text{denom}^2}
    \Bigr)
    \frac{1}{\text{denom}}
    \Big[
        -\bm D\bm G \bm{A}\bm e_1 \hpsi_i
        + \bm V \bm e_i (\Theta_1 - \htheta_1)
    \Big]
    \end{align*}
    where $\text{denom}= (\|\bm \hpsi\|^2 + n \Theta_1^2)^{1/2}$ for brevity, and the second line follows
    from \cite[Theorem 2.1]{bellec2021derivatives}.
    The matrix in the first parenthesis has operator norm at most $1$,
    so that using $(a+b)^2/2\le a^2+b^2$ gives
    $$
    \frac 1 2
    \sum_{i=1}^n\Big\|\frac{\partial \bm f}{\partial z_i}\Big\|^2
    \le \|\bm D \bm G \bm{A} \bm e_1\|^2
    \frac{\|\bm \hpsi\|^2}{\text{denom}^2}
    + \|\bm V\|_F^2 
    \frac{(\htheta_1 - \Theta_1)^2}{\text{denom}^2}.
    $$
    By \eqref{eq:bound-V-op}, \eqref{eq:bound-htheta} and \eqref{eq:7-norm-G},
    the expectation of the previous display is at most $\C(\delta,\tau)/n$.
    We also have that
    $
    |\sum_{i=1}^n \frac{\partial f_i}{\partial z_i}
    + \frac{\trace[\bm V](\htheta_j - \Theta_j)}{\text{denom}} |
    \le \C(\tau,\delta) O_\P(n^{-1/2})
    $
    again thanks to
    \eqref{eq:bound-V-op}, \eqref{eq:bound-htheta} and \eqref{eq:7-norm-G}.
    We conclude that, ommiting constants depending only on $\tau,\delta$,
    $$
    |
    t_* - \hat v \|\bm \Sigma^{1/2}\bm\beta^*\|
    |
    =
    \tfrac{|
    \bm e_1^T\bm G^T\bm \hpsi
    + \trace[\bm V] (\hat\theta_1 - \Theta_1)
    |}{n}
    \le O_\P(n^{-3/2}) \text{denom} 
    \le O_\P(n^{-1})(\hat r + \|\bm \Sigma^{1/2}\bm\beta^*\| ).
    $$
    The bound $|\hat t- t_*| \le O_\P(n^{-1/2}) \hat r$
    is provided by \eqref{eq:consistency-hat-t}
    and completes the proof.
\end{proof}

\end{document}